\documentclass{article}

\usepackage{amsmath}
\usepackage{amsthm}
\usepackage{hyperref}
\usepackage[capitalise]{cleveref}
\usepackage{float}
\usepackage[utf8]{inputenc}
\usepackage[T1]{fontenc}
\usepackage{subcaption}
\usepackage{caption}
\usepackage{enumitem}
\usepackage{algorithm}
\usepackage{algpseudocode}
\usepackage{comment}
\usepackage{lipsum}
\usepackage{amsfonts}
\usepackage{graphicx}
\usepackage{epstopdf}
\usepackage{csvsimple}
\usepackage{tikz}
\usepackage{pgfplots}
\usepackage{pgfplotstable}
\pgfplotsset{compat=newest}
\usepackage{xcolor}
\usepackage{amsopn}
\usepackage{microtype}
\usepackage{booktabs}
\usetikzlibrary{positioning}

\usepackage[giveninits=true,style = numeric,backend=biber,sorting=nyvt]{biblatex}
\DeclareNameAlias{default}{family-given}
\newcommand{\Krasnoselsky}{Krasnosel'ski\u{\i}}
\newcommand{\PrimS}{\ensuremath{\mathcal{H}}}

\newcommand{\nat}{\ensuremath{\mathbb{N}}}

\newcommand{\weakto}{\ensuremath{\rightharpoonup}}
\newcommand{\Id}{\ensuremath{\mathop{\mathrm{Id}}}}
\newcommand{\prt}[1]{\ensuremath{\mathord{\left( #1 \right)}}}
\newcommand{\norm}[1]{\ensuremath{\mathord{\left\Vert #1 \right\Vert}}}

\newcommand{\setcond}[2]{\ensuremath{\mathord{\left\lbrace #1 : #2 \right\rbrace}}}
\newcommand{\inpr}[2]{\ensuremath{\mathord{\left\langle #1, #2 \right\rangle}}}

\newcommand{\btheta}{\Bar{\theta}}
\newcommand{\htheta}{\Hat{\theta}}
\newcommand{\ttheta}{\Tilde{\theta}}
\newcommand{\ptheta}{\theta^{\prime}}

\newcommand{\balpha}{\Bar{\alpha}}

\newcommand{\hbeta}{\Bar{\beta}}
\newcommand{\hvarphi}{\widehat{\varphi}}

\newcommand{\zer}[1]{\operatorname{\mathrm{zer}}(#1)}
\newcommand{\reals}{\mathbb{R}}

\newcommand{\posop}[1]{\mathcal{M}\prt{\PrimS}}
\newcommand{\lp}{\left(}
\newcommand{\rp}{\right)}
\newcommand{\xs}{x^\star}

\newcommand{\gra}[1]{\operatorname{gra}(#1)}
\newcommand{\seqq}[1]{\ensuremath{\prt{#1}_{n\in\mathbb{N}}}}

\newcommand{\defeq}{\ensuremath{\mathrel{\mathop:}=}}

\newcommand{\el}[1]{\ell_{#1}}
\newcommand{\elll}[1]{\ell_{#1}}
\newcommand{\wphi}{\widehat{\phi}}

\allowdisplaybreaks

\definecolor{matplotlibcolor1of11}{rgb}{0.5, 0.0, 1.0}
\definecolor{matplotlibcolor2of11}{rgb}{0.3039, 0.3032, 0.9882}
\definecolor{matplotlibcolor3of11}{rgb}{0.1, 0.5878, 0.9511}
\definecolor{matplotlibcolor4of11}{rgb}{0.0961, 0.8054, 0.8924}
\definecolor{matplotlibcolor5of11}{rgb}{0.3, 0.9511, 0.8090}
\definecolor{matplotlibcolor6of11}{rgb}{0.5039, 1.0, 0.7049}
\definecolor{matplotlibcolor7of11}{rgb}{0.7, 0.9511, 0.5878}
\definecolor{matplotlibcolor8of11}{rgb}{0.9039, 0.8054, 0.4512}
\definecolor{matplotlibcolor9of11}{rgb}{1.0, 0.5878, 0.3090}
\definecolor{matplotlibcolor10of11}{rgb}{1.0, 0.3032, 0.1534}
\definecolor{matplotlibcolor11of11}{rgb}{1.0, 0.0, 0.0}

\definecolor{matplotlibcolor1of21}{rgb}{0.5000, 0.0000, 1.0000}
\definecolor{matplotlibcolor2of21}{rgb}{0.4059, 0.1473, 0.9973}
\definecolor{matplotlibcolor3of21}{rgb}{0.3039, 0.3032, 0.9882}
\definecolor{matplotlibcolor4of21}{rgb}{0.2020, 0.4512, 0.9727}
\definecolor{matplotlibcolor5of21}{rgb}{0.1000, 0.5878, 0.9511}
\definecolor{matplotlibcolor6of21}{rgb}{0.0020, 0.7093, 0.9233}
\definecolor{matplotlibcolor7of21}{rgb}{0.0961, 0.8054, 0.8924}
\definecolor{matplotlibcolor8of21}{rgb}{0.1980, 0.8896, 0.8534}
\definecolor{matplotlibcolor9of21}{rgb}{0.3000, 0.9511, 0.8090}
\definecolor{matplotlibcolor10of21}{rgb}{0.4020, 0.9882, 0.7594}
\definecolor{matplotlibcolor11of21}{rgb}{0.5039, 1.0000, 0.7049}
\definecolor{matplotlibcolor12of21}{rgb}{0.5980, 0.9882, 0.6506}
\definecolor{matplotlibcolor13of21}{rgb}{0.7000, 0.9511, 0.5878}
\definecolor{matplotlibcolor14of21}{rgb}{0.8020, 0.8896, 0.5212}
\definecolor{matplotlibcolor15of21}{rgb}{0.9039, 0.8054, 0.4512}
\definecolor{matplotlibcolor16of21}{rgb}{1.0000, 0.7005, 0.3784}
\definecolor{matplotlibcolor17of21}{rgb}{1.0000, 0.5878, 0.3090}
\definecolor{matplotlibcolor18of21}{rgb}{1.0000, 0.4512, 0.2319}
\definecolor{matplotlibcolor19of21}{rgb}{1.0000, 0.3032, 0.1534}
\definecolor{matplotlibcolor20of21}{rgb}{1.0000, 0.1473, 0.0739}
\definecolor{matplotlibcolor21of21}{rgb}{1.0000, 0.0000, 0.0000}

\definecolor{matplotlibcolor1of6}{rgb}{0.5, 0.0, 1.0}
\definecolor{matplotlibcolor2of6}{rgb}{0.1, 0.5878, 0.9511}
\definecolor{matplotlibcolor3of6}{rgb}{0.3, 0.9511, 0.8090}
\definecolor{matplotlibcolor4of6}{rgb}{0.7, 0.9511, 0.5878}
\definecolor{matplotlibcolor5of6}{rgb}{1.0, 0.5878, 0.3090}
\definecolor{matplotlibcolor6of6}{rgb}{1.0, 0.0, 0.0}

\definecolor{matplotlibcolor1of19}{rgb}{0.5000, 0.0000, 1.0000}
\definecolor{matplotlibcolor2of19}{rgb}{0.3902, 0.1716, 0.9963}
\definecolor{matplotlibcolor3of19}{rgb}{0.2804, 0.3382, 0.9852}
\definecolor{matplotlibcolor4of19}{rgb}{0.1706, 0.4947, 0.9667}
\definecolor{matplotlibcolor5of19}{rgb}{0.0608, 0.6365, 0.9411}
\definecolor{matplotlibcolor6of19}{rgb}{0.0569, 0.7674, 0.9059}
\definecolor{matplotlibcolor7of19}{rgb}{0.1667, 0.8660, 0.8660}
\definecolor{matplotlibcolor8of19}{rgb}{0.2765, 0.9390, 0.8197}
\definecolor{matplotlibcolor9of19}{rgb}{0.3863, 0.9841, 0.7674}
\definecolor{matplotlibcolor10of19}{rgb}{0.5039, 1.0000, 0.7049}
\definecolor{matplotlibcolor11of19}{rgb}{0.6137, 0.9841, 0.6412}
\definecolor{matplotlibcolor12of19}{rgb}{0.7235, 0.9390, 0.5727}
\definecolor{matplotlibcolor13of19}{rgb}{0.8333, 0.8660, 0.5000}
\definecolor{matplotlibcolor14of19}{rgb}{0.9431, 0.7674, 0.4235}
\definecolor{matplotlibcolor15of19}{rgb}{1.0000, 0.6365, 0.3382}
\definecolor{matplotlibcolor16of19}{rgb}{1.0000, 0.4947, 0.2558}
\definecolor{matplotlibcolor17of19}{rgb}{1.0000, 0.3382, 0.1716}
\definecolor{matplotlibcolor18of19}{rgb}{1.0000, 0.1716, 0.0861}
\definecolor{matplotlibcolor19of19}{rgb}{1.0000, 0.0000, 0.0000}

\newtheorem{remark}{Remark}

\theoremstyle{plain}
\newtheorem{example}{Example}
\newtheorem{assumption}{Assumption}
\newtheorem{lemma}{Lemma}
\newtheorem{theorem}{Theorem}
\newtheorem{proposition}{Proposition}
\newtheorem{corollary}{Corollary}
\crefname{assumption}{Assumption}{Assumptions}

\theoremstyle{definition}

\newtheorem*{keywords}{Key words}
\newtheorem*{AMS}{AMS Subject Classification (2020)}

\title{Incorporating History and Deviations in Forward--Backward Splitting}

\author{Hamed Sadeghi\thanks{Email: \{\href{mailto:hamed.sadeghi@control.lth.se}{hamed.sadeghi}, \href{mailto:sebastian.banert@control.lth.se}{sebastian.banert}, \href{mailto:pontus.giselsson@control.lth.se}{pontus.giselsson}\}@control.lth.se. Affiliation: Lund University, Lund, Sweden.}
   \and Sebastian Banert\footnotemark[1]
   \and Pontus Giselsson\footnotemark[1]}

\begin{document}
\maketitle

\begin{abstract}

We propose a variation of the forward--backward splitting method for solving structured monotone inclusions. Our method integrates past iterates and two deviation vectors into the update equations. These deviation vectors bring flexibility to the algorithm and can be chosen arbitrarily as long as they together satisfy a norm condition. 
We present special cases where the deviation vectors, selected as predetermined linear combinations of previous iterates, always meet the norm condition. 
Notably, we introduce an algorithm employing a scalar parameter to interpolate between the conventional forward--backward splitting scheme and an accelerated $\mathcal{O}\lp\tfrac{1}{n^2}\rp$-convergent forward--backward method that encompasses both the accelerated proximal point method and the Halpern iteration as special cases. The existing methods correspond to the two extremes of the allowed scalar parameter range.
By choosing the interpolation scalar near the midpoint of the permissible range, our algorithm significantly outperforms these previously known methods when addressing a basic monotone inclusion problem stemming from minimax optimization.

\end{abstract}

\begin{keywords}
forward--backward splitting, monotone inclusions, inertial algorithms, convergence rate, Halpern iteration, deviations
\end{keywords}

\begin{AMS}
47H05, 47N10, 65K05.
\end{AMS}

\section{Introduction}\label{sec:introduction}
We consider the problem of finding $x\in\PrimS$ such that
\begin{align}\label{eq:monotone-inclusion_intro}
    0 \in Ax+Cx,
\end{align}
where $A\colon \PrimS\to 2^\PrimS$ is a maximally monotone operator, $C\colon \PrimS\to\PrimS$ is a cocoercive operator, $\PrimS$ is a real Hilbert space and $2^\PrimS$ denotes its power set. This monotone inclusion has optimization problems \cite{Eckstein1989SplittingMF,Raguet_2015}, convex-concave saddle-point problems \cite{chambolle2011first}, and variational inequalities  \cite{Attouch_coupling,Chen_1997,Tseng_2000} as special cases.

The forward--backward (FB) splitting method \cite{Bruck1975AnIS,Lions_1979,Passty_1979} has been widely used to solve the monotone inclusion problem \eqref{eq:monotone-inclusion_intro}. 
The gradient method, the proximal point algorithm \cite{rockafellar1976monotone}, the proximal-gradient method \cite{Combettes2011ProximalSM}, the Chambolle--Pock method \cite{chambolle2011first}, the Douglas--Rachford method~\cite{Lions_1979,Eckstein1989SplittingMF}, and the \Krasnoselsky--Mann iteration~\cite[Section~5.2]{bauschke2017convex} can all be considered special instances of the FB method. 
Various attempts have been made to improve the convergence of the FB splitting algorithm by incorporating information from previous iterations. Notable examples include the heavy-ball method \cite{polyak1964some}, the inertial proximal point algorithm  \cite{Alvarez_2000,Alvarez_2001}, and inertial FB algorithms  \cite{apidopoulos2020convergence,attouch2020convergence,attouch2016rate,attouch2018fast,beck2009fast,chambolle2015convergence,Cholamjiak_2018,lorenz2015inertial,Morin2021Nonlinear}, which integrate prior information into the current iteration through a momentum term.

In this paper, we propose an extension to the conventional FB algorithm 
that includes momentum-like terms and two \emph{deviation vectors}. 
These deviations have the same dimension as the underlying space of the problem and serve as adjustable parameters that provide the algorithm with great flexibility. This flexibility can be exploited to control the trajectory of the iterations with the aim to enhance algorithm convergence. To guarantee convergence, we require the deviations to satisfy a safeguarding condition that restricts the norm, but not the direction, of the deviation vectors.
Our safeguarding approach is similar to those in \cite{Banert2021AcceleratedFO, dwifob, nofob-inc}---which indeed are special instances of our algorithm---while it distinctly contrasts with the safeguarding conditions presented in \cite{giselsson2016line, sad2021Hyb, themelis2019supermann, zhang2020globally} that choose between a globally convergent and locally fast methods depending on the fulfillment of their respective safeguarding conditions.


We also introduce two special cases where the deviation vectors are predetermined linear combinations of prior iteration data. This construction ensures that the safeguarding condition is met in all iterations, implying that it does not require online evaluation. The two special cases incorporate different scalar parameters controlling the behaviour of their respective algorithms. In one case, the scalar parameter $\kappa\in(-1,1)$ regulates the momentum used in the algorithm, with $\kappa=0$ yielding the standard FB method. This algorithm converges weakly towards a solution of the inclusion problem for all $\kappa$ within the permitted range. In the other case, the scalar parameter $e\in[0,1]$ acts as an interpolator between the standard FB method ($e=0$) and an accelerated FB method ($e=1$) featuring the accelerated proximal point method from \cite{kim2021accelerated} and the Halpern iteration analyzed in \cite{lieder2021convergence} as special cases. The scalar parameter $e$ regulates the convergence rate of the squared norm of the fixed-point residual, converging as $\mathcal{O}\lp\min\lp 1/n^{2e},1/n\rp\rp$, with $e=1$ offering the accelerated FB method with an $\mathcal{O}\lp1/n^2\rp$ convergence rate, consistent with the rates in \cite{kim2021accelerated,lieder2021convergence}.


We perform numerical evaluation of these two special cases on a simple skew-symmetric monotone inclusion problem arising from optimality conditions for the minimax problem $\max_{y\in\reals}\min_{x\in\reals}xy$. Our findings suggest that with $\kappa\in[0.8,0.9]$, our first special case performs an order of magnitude better than the FB method ($\kappa=0$) on this problem. Furthermore, by allowing $e\in[0.4,0.5]$, we observe that our second special case outperforms the FB method ($e=0$) by an order of magnitude and performs several orders of magnitude better than the accelerated FB method ($e=1$), despite the latter's stronger theoretical convergence guarantee.

The analysis of our base algorithm relies on a Lyapunov inequality. We derive this inequality by applying the monotonicity inequality of operator $A$ and the cocoercivity inequality of operator $C$ (that are referred to as {\emph{interpolation conditions}} in the terminology of performance estimation (PEP), see for instance \cite{ryu1812operator,taylor2017performance,taylor2017exact}), both between the last iterate and a solution of the problem, as well as the last two points generated by the algorithm. This is in contrast to the analysis conducted in \cite{nofob-inc}, which restricts the use of these inequalities to only the last iteration and a solution. The inclusion of additional inequalities allows for deriving special cases such as the one that interpolates between FB splitting and accelerated FB splitting along with the associated convergence rate of $\mathcal{O}(1/n^{2e})$ where $e\in[0,1]$. This result is not achievable via the algorithm proposed in \cite{nofob-inc}.


The paper is organized as follows. In \cref{sec:preliminaries}, we establish the basic definitions and notations used throughout the paper. \Cref{sec:algorithm} contains the formal presentation of the problem and introduces our proposed algorithm that is analyzed in \cref{sec:convergence-analysis}. In \cref{sec:special_cases}, we present several special instances of our algorithm, two of which we examine numerically in \cref{sec:numerical-experiments}. Proofs omitted for the sake of brevity are shared in \cref{sec:omitted-proofs} and \cref{sec:conclusions} concludes the paper.

\section{Preliminaries}\label{sec:preliminaries}
The set of real numbers is denoted by $\reals$. $\PrimS$  denotes a real Hilbert space that is equipped with an inner product and an induced norm, respectively denoted by $\inpr{\cdot}{\cdot}$ and $\norm{\cdot}:=\sqrt{\inpr{\cdot}{\cdot}}$. $\posop{\PrimS}$ denotes the set of bounded linear, self-adjoint, strongly positive operators on $\PrimS$. For $M\in\posop{\PrimS}$ and all $x,y\in\PrimS$, the $M$-induced inner product and norm are denoted and defined by $\inpr{x}{y}_M:=\inpr{x}{My}$ and $\norm{x}_M = \sqrt{\inpr{x}{Mx}}$, respectively.

The \emph{power set} of $\PrimS$ is denoted by $2^\PrimS$. A map $A\colon\PrimS\to 2^{\PrimS}$ is characterized by its graph $\gra{A} = \setcond{\prt{x, u}\in\PrimS\times\PrimS}{u\in Ax}$. An operator $A\colon \PrimS\to 2^{\PrimS}$ is \emph{monotone} if $\inpr{u-v}{x-y}\geq0$ for all $\prt{x, u}, \prt{y, v}\in\gra{A}$. A monotone operator $A\colon \PrimS \to 2^{\PrimS}$ is \emph{maximally monotone} if there exists no monotone operator $B\colon\PrimS\to 2^\PrimS$ such that $\gra{B}$ properly contains $\gra{A}$.

Let $M\in \posop{\PrimS}$. An operator $T\colon\PrimS\to\PrimS$ is said to be 
\renewcommand{\labelenumi}{{(\roman{enumi})}}
\begin{enumerate}
    \item \emph{$L$-Lipschitz continuous} ($L \geq 0$) w.r.t.\ $\norm{\cdot}_M$ if
        \begin{equation*}
            \norm{Tx-Ty}_{M^{-1}}\leq L\norm{x-y}_M \qquad \text{for all } x,y\in\PrimS;
        \end{equation*}
    \item \emph{$\frac{1}{\beta}$-cocoercive} ($\beta \geq 0$) w.r.t. $\norm{\cdot}_M$ if
        \begin{equation*}
            \beta\inpr{Tx -Ty}{x-y}\geq\norm{Tx-Ty}_{M^{-1}}^2\qquad \text{for all } x,y\in\PrimS;
        \end{equation*}
    \item \emph{nonexpansive} if it is $1$-Lipschitz continuous w.r.t. $\norm{\cdot}$.
\end{enumerate}
Note that a $\frac{1}{\beta}$-cocoercive operator is $\beta$-Lipschitz continuous. This holds trivially for $\beta=0$ and for $\beta>0$ it follows from the Cauchy-Schwarz inequality.


\section{Problem statement and proposed algorithm}
\label{sec:algorithm}

We consider structured monotone inclusion problems of the form
\begin{align} \label{eq:monotone_inclusion}
    0 \in Ax + Cx,
\end{align}
that satisfy the following assumption.

\begin{assumption}
\label{assum:monotone_inclusion}
Let $\beta \geq 0$ and $M\in\posop{\PrimS}$ and assume that
\renewcommand{\labelenumi}{\emph{(\roman{enumi})}}
\begin{enumerate}
    \item $A\colon \PrimS \to 2^{\PrimS}$ is maximally monotone,
    \item $C\colon \PrimS \to \PrimS$ is $\frac{1}{\beta}$-cocoercive with respect to $\norm{\cdot}_M$, 
    \item the solution set $\zer{A+C} \defeq \setcond{x\in\PrimS}{0\in Ax+Cx}$ is nonempty.
\end{enumerate}
\end{assumption}
Since operator $C$ has a full domain and is maximally monotone as a cocoercive operator \cite[Corollary 20.28]{bauschke2017convex}, the operator $A+C$ is also maximally monotone \cite[Corollary 25.5]{bauschke2017convex}.

We propose the following variant of FB splitting which incorporates momentum terms and deviations in order to solve the inclusion problem in \eqref{eq:monotone_inclusion}. The algorithm has many degrees of freedom that we will specify later in this section and in the special cases found in \cref{sec:special_cases}.

\begin{algorithm}[H]
	\caption{}
	\begin{algorithmic}[1]
	    \State \textbf{Input:} initial point $x_0 \in \PrimS$;  the strictly positive sequences \seqq{\gamma_n} and \seqq{\lambda_n};  the non-negative sequences \seqq{\zeta_n} and  \seqq{\mu_n}; $\hbeta\geq\beta\geq 0$; and the metric $\norm{\cdot}_M$ with $M\in\posop{\PrimS}$. 
	    \State \label{alg:main:auxiliary-parameters} Given the input parameters, for all $n\in\nat$, define:
	    \renewcommand\theenumi\labelenumi
        \renewcommand{\labelenumi}{{(\roman{enumi})}}
        \begin{enumerate}[noitemsep, ref=\Cref{def:parameters}~\theenumi]
            \item $\alpha_n\defeq\tfrac{\mu_n}{\lambda_n+\mu_n}$;\label{itm:def-param-alpha}
            \item $\balpha_n\defeq\tfrac{\gamma_n\mu_n}{\gamma_{n-1}(\lambda_n+\mu_n)}$;\label{itm:def-param-alpha-bar}
            \item $\theta_n\defeq(4-\gamma_n\hbeta)(\lambda_n+\mu_n)-2\lambda_n^2$;\label{itm:def-param-theta}
            \item $\htheta_n\defeq2\lambda_n+2\mu_n-\gamma_n\hbeta\lambda_n^2$;\label{itm:def-param-theta-hat}
            \item $\btheta_n\defeq\lambda_n+\mu_n-\lambda_n^2$;\label{itm:def-param-theta-bar}
            \item $\ttheta_n\defeq(\lambda_n+\mu_n)\gamma_n\hbeta$.\label{itm:def-param-theta-tilde}
        \end{enumerate}

	    \State \textbf{set:} $y_{-1}=p_{-1}=z_{-1}=x_0$ and $u_0=v_0=0$ and $\gamma_{-1}=\gamma_0$
		\For {$n=0,1,2,\ldots$}
		    \State $y_n=x_n+\alpha_n(y_{n-1}-x_n)+u_n$\label{alg:main:y}
		    \State $z_n=x_n+\alpha_n(p_{n-1}-x_n)+\balpha_n(z_{n-1}-p_{n-1})+\tfrac{\btheta_n\gamma_n\hbeta}{\htheta_n}u_n+v_n$\label{alg:main:z}
		    \State $p_n = \prt{M+\gamma_nA}^{-1}\prt{M z_n - \gamma_n C y_n}$ \label{alg:main:FB-step}
		    \State $x_{n+1} = x_n + \lambda_n(p_n - z_n) + \balpha_n\lambda_n(z_{n-1}-p_{n-1})$ \label{alg:main:x}
		    \State choose $u_{n+1}$ and $v_{n+1}$ such that \label{alg:main:deviations}
		    \begin{equation}\label{eq:bound_on_deviations}         
      \left(\lambda_{n+1}+\mu_{n+1}\right)\lp\frac{\ttheta_{n+1}}{\htheta_{n+1}}\norm{u_{n+1}}_{M}^2 + \frac{\htheta_{n+1}}{\theta_{n+1}}\norm{v_{n+1}}_{M}^2\rp \leq \zeta_{n+1}\elll{n}
            \end{equation}
            ~~~~~is satisfied, where
            \begin{equation}\label{eq:ell_n}
            \begin{aligned}
                \elll{n} &= \tfrac{\theta_n}{2}\norm{p_n-x_n+\alpha_n(x_n-p_{n-1}) +\tfrac{\gamma_n\hbeta\lambda_n^2}{\htheta_n}u_n - \tfrac{2\btheta_n}{\theta_n}v_n}_M^2\\               
                &\qquad+2\mu_{n}\gamma_{n}\inpr{\tfrac{z_{n}-p_{n}}{\gamma_{n}}-\tfrac{z_{n-1}-p_{n-1}}{\gamma_{n-1}}}{p_{n}-p_{n-1}}_M\\
                &\qquad+\tfrac{\mu_{n}\gamma_{n}\hbeta}{2}\norm{p_{n}-y_{n}-(p_{n-1}-y_{n-1})}_M^2
            \end{aligned}
            \end{equation}
		\EndFor
	\end{algorithmic}
\label{alg:main}
\end{algorithm}

At the core of the method is a forward--backward type step, found in Step~\ref{alg:main:FB-step} of \cref{alg:main}, which reduces to a nominal forward--backward step when $z_n=y_n$. The update equations for the algorithm sequences $y_n$, $z_n$, and $x_n$ involve linear combinations of momentum-like terms and the so-called {\emph{deviations}}, $u_n$ and $v_n$. These deviations are arbitrarily chosen provided they satisfy safeguarding condition in \eqref{eq:bound_on_deviations}, where $\elll{n}$ is defined in \eqref{eq:ell_n}. When selecting the deviations, all other quantities involved in \eqref{eq:bound_on_deviations} are computable. These deviations offer a degree of flexibility that can be used to control the algorithm trajectory with the aim of improving convergence. In \cref{sec:special_cases}, we present examples of nontrivial deviations that a priori satisfy this condition, thus removing the need for online evaluation.

For the algorithm to be implementable, let alone convergent, the algorithm parameters must be constrained. For the FB step in Step~\ref{alg:main:FB-step} to be implementable, and the safeguarding step to be satisfied for some $u_{n+1}$ and $v_{n+1}$, we require for all $n\in\nat$ that $\gamma_n$, $\lambda_n$, $\theta_n$, $\htheta_n$, and $\ttheta_n$ are strictly positive and the parameters $\zeta_n$, $\mu_n$, $\alpha_n$, and $\balpha_n$ are non-negative. Fulfillment of these requirements allows for a trivial choice that satisfies the safeguarding condition \eqref{eq:bound_on_deviations}, namely $u_{n+1}=v_{n+1}=0$, which results in a novel momentum-type forward--backward scheme. Additional requirements on some of these parameters, that are needed for the convergence analysis, are discussed in \cref{sec:convergence-analysis}.

\cref{alg:main} can be viewed as an extension of the algorithm in \cite{nofob-inc}. The key difference arises from the inclusion of additional monotonicity and cocoercivity inequalities (interpolation conditions) in our analysis compared to the analysis of \cite{nofob-inc}. In contrast to the analysis in \cite{nofob-inc}, we utilize inequalities not only between the last iteration points and a solution but also between points generated during the last two iterations of our algorithm. 
This approach provides our algorithm with an additional degree of freedom, embodied by the parameter $\mu_n$, that stems from the degree to which these extra interpolation conditions are incorporated into the analysis.
This addition yields momentum-like terms in the updates, a less restrictive safeguarding condition, and the potential to derive convergence rate estimates for several involved quantities up to $O(\tfrac{1}{n^2})$. Such rates are not achievable in \cite{nofob-inc} as setting $\mu_n$ to zero reverts our algorithm to that of \cite{nofob-inc}.

When $\gamma_n=\gamma>0$ for every $n\in\nat$, the deviation vectors $u_n$ and $v_n$ can be chosen so that $y_{n}=z_n$. In this case, Step~\ref{alg:main:FB-step} of \cref{alg:main} simplifies to a FB step of the form
\begin{align*}
    p_n = \lp(M+\gamma A)^{-1}\circ(M-\gamma C)\rp y_n.
\end{align*}
It is widely recognized that given appropriate selections of $\gamma>0$, $M\in\posop{\PrimS}$, $A$, and $C$, this FB step can reduce to iterations of well-known algorithms. These include the Chambolle--Pock algorithm~\cite{chambolle2011first}, the Condat--V{\~u} method~\cite{condat2013primal,vu2013splitting}, the Douglas--Rachford method~\cite{Lions_1979}, the \Krasnoselsky--Mann iteration~\cite[Section~5.2]{bauschke2017convex}, and the proximal gradient method. Consequently, \cref{alg:main} can be applied to all these special cases.

\subsection{Preview of special cases}
\label{sec:special_case_preview}



This section previews some special cases of \cref{alg:main}, which we will explore in depth in \cref{sec:special_cases,sec:numerical-experiments}.
Specifically, we consider cases where the sequence \seqq{\lambda_n} is non-decreasing and, for all $n\in\nat$, $\gamma_n=\gamma>0$,
\begin{equation*}
    v_n = \tfrac{\prt{2-\gamma\hbeta}(\lambda_n+\mu_n)}{\htheta_n}u_n,
\end{equation*}
and $u_n$ is parallel to the expression in the first norm in the $\elll{n}$ expression in \eqref{eq:ell_n} that contributes to the upper bound in the safeguarding condition in \eqref{eq:bound_on_deviations}. This yields $z_n=y_n$ and as demonstrated in Section~\ref{sec:special_cases}, with a particular choice of $\mu_n$, 
\cref{alg:main} becomes
\begin{equation*}
\begin{aligned}
		     y_n&=x_n+\tfrac{\lambda_n-\lambda_0}{\lambda_n}(y_{n-1}-x_n)+u_n\\
    p_n &= \prt{M+\gamma A}^{-1}\prt{M y_n - \gamma C y_n}\\
		    x_{n+1} &= x_n + \lambda_n(p_n - y_n) + (\lambda_n-\lambda_0)(y_{n-1}-p_{n-1})\\
      u_{n+1}&=\kappa_n\tfrac{4-\gamma\hbeta-2\lambda_0}{2}(p_n-x_n+\tfrac{\lambda_n-\lambda_0}{\lambda_n}(x_n-p_{n-1}) - \tfrac{2-\gamma\hbeta-2\lambda_0}{4-\gamma\hbeta-2\lambda_0} u_n),
\end{aligned} 
\end{equation*}
that is initialized with $y_{-1}=p_{-1}=x_0$ and $u_0=0$.
This algorithm, under the condition $\kappa_n^2\leq\zeta_{n+1}$, satisfies the safeguarding condition. As we will see later, $\zeta_{n+1}\in[0,1]$ can be set arbitrarily (though stronger convergence conclusions can be drawn if $\zeta_n<1$ for all $n\in\nat$), indicating that as long as $\kappa_n\in[-1,1]$, this new algorithm satisfies the safeguarding condition by design.

In \cref{sec:special_cases}, we present two special cases of this iteration that we numerically evaluate in \cref{sec:numerical-experiments}. The first special case involves setting $\lambda_0=1$ and, for all $n\in\nat$, $\lambda_n=\lambda_0$ and $\kappa_n=\kappa\in(-1,1)$. As shown 
in \cref{sec:special_cases}, the resulting algorithm can be written as
\begin{equation}
\begin{aligned}
    p_n &= \lp\prt{M+\gamma A}^{-1}\circ\prt{M - \gamma C }\rp\lp p_{n-1}+u_n-u_{n-1}\rp\\
      u_{n+1}&=\kappa\lp\tfrac{2-\gamma\hbeta}{2}(p_n-p_{n-1}+u_{n-1}) + \tfrac{\gamma\hbeta}{2} u_n\rp
\end{aligned}
\label{eq:alg-constant-kappa-preview}
\end{equation}
that is initialized with $u_{-1}=u_0=0$.
This algorithm converges weakly to a solution of the inclusion problem and setting $\kappa=0$ gives $u_n=0$ for all $n\in\nat$ and the algorithm reduces to the standard FB method.

The second special case is obtained by letting $\kappa_{n}=\tfrac{\lambda_{n+1}-\lambda_0}{\lambda_{n+1}}\in[0,1)$ 
resulting, as shown in \cref{sec:special_cases}, in the algorithm: 
\begin{equation}
    \begin{aligned}
    p_n &= \prt{M+\gamma A}^{-1}\prt{M y_n - \gamma C y_n}\\
     y_{n+1} 
     &=y_{n} + \lp\tfrac{\lambda_0\lambda_n}{\lambda_{n+1}}+\tfrac{\lambda_{n+1}-\lambda_0}{\lambda_{n+1}}\tfrac{4-\gamma\hbeta-2\lambda_0}{2}\rp(p_n - y_n)  \\
     &\quad+\tfrac{\lambda_n-\lambda_0}{\lambda_{n+1}}\lp(y_n- y_{n-1})+\tfrac{4-\gamma\hbeta}{2}(y_{n-1} - p_{n-1})\rp,
\end{aligned}
\label{eq:alg-kappa-equals-alpha-preview}
\end{equation}
that is initialized with $y_{-1}=p_{-1}$. We will pay particular attention to the choice $\lambda_n=\left(1-\tfrac{\gamma\hbeta}{4}\right)^e(1+n)^e$ with $e\in[0,1]$. The choice $e=0$ gives the standard FB method and, as shown in \cref{sec:accelerated_ppm_Halpern}, the choice $e=1$ has the Halpern iteration analyzed in \cite{lieder2021convergence} and the accelerated proximal point method in \cite{kim2021accelerated} as special cases. By choosing an $e$ value between the allowed extremes, we can interpolate between these methods. We show that $\|p_n-y_n\|_M^2$ with this choice of \seqq{\lambda_n} converges as $\mathcal{O}\lp\tfrac{1}{n^{2e}}\rp$ (and, if $\hbeta>\beta$, as $\mathcal{O}\lp\tfrac{1}{n}\rp$ for all $e\in[0,1]$)
implying that the convergence rate can be tuned by selecting $e$. The requirements we will pose on the parameters in \cref{sec:convergence-analysis} to guarantee convergence state that $\seqq{\lambda_n}$ can grow at most linearly, meaning values of $e>1$ are not viable and the best possible rate is $\mathcal{O}\lp\tfrac{1}{n^2}\rp$ obtained by letting $e=1$.
As shown in \cref{sec:special_cases}, this case recovers the exact $\mathcal{O}\lp\tfrac{1}{n^2}\rp$ rate results of the Halpern iteration in \cite{lieder2021convergence} and the accelerated proximal point method in \cite{kim2021accelerated}. 

In \cref{sec:numerical-experiments}, we present numerical experiments on a simple skew-symmetric monotone inclusion problem, originating from the problem $\max_{y\in\reals}\min_{x\in\reals}xy$. We find that both algorithms \eqref{eq:alg-constant-kappa-preview} and \eqref{eq:alg-kappa-equals-alpha-preview} can significantly outperform the standard FB method and the Halpern iteration when $\kappa$ and $e$ are appropriately chosen.

\section{Convergence analysis}\label{sec:convergence-analysis}

In this section, we conduct a Lyapunov-based convergence analysis for \cref{alg:main}. In \cref{thm:LP-identity}, we define a {\emph{Lyapunov function}}, $V_n$, based on the iterates generated by \cref{alg:main}, and present an identity that establishes a relation between $V_{n+1}$ and $V_n$. In \cref{thm:LP-ineq}, we introduce additional assumptions and derive a {\emph{Lyapunov inequality}} that serves as the main tool for the convergence and convergence rate analysis in \cref{thm:convergence}. 

The proof of our first theorem is lengthy and only based on algebraic manipulations and is therefore deferred to \cref{sec:omitted-proofs}. The equality in the proof is validated with symbolic calculations in
\begin{center}
\url{https://github.com/sbanert/incorporating-history-and-deviations}.
\end{center}

\begin{theorem}\label{thm:LP-identity}
Suppose that \cref{assum:monotone_inclusion} holds. Let $\xs$ be an arbitrary point in $\zer{A+C}$ and $V_0=\norm{x_0-\xs}_M^2$, and based on the iterates generated by \cref{alg:main}, for all $n\in\nat$, let
\begin{equation}\label{eq:LP-function-def}
    V_{n+1} \defeq \norm{x_{n+1}-\xs}_M^2 + 2\lambda_{n+1}\gamma_{n+1}\alpha_{n+1}\phi_n + \el{n},
\end{equation}
where
\begin{equation}\label{eq:phi-def}
    \phi_n \defeq \inpr{\tfrac{z_n-p_n}{\gamma_n}}{p_n-\xs}_M + \tfrac{\hbeta}{4}\norm{y_n-p_n}_M^2,
\end{equation}
and $\el{n}$ given by \eqref{eq:ell_n}. Then,
\begin{equation*}
    \begin{aligned}
        V_{n+1}  + 2\gamma_n(\lambda_n &- \balpha_{n+1}\lambda_{n+1})\phi_n + \el{n-1} \\
        &= V_n + (\lambda_{n}+\mu_{n})\lp\frac{\ttheta_{n}}{\htheta_{n}}\norm{u_{n}}_{M}^2 + \frac{\htheta_{n}}{\theta_{n}}\norm{v_{n}}_{M}^2\rp
    \end{aligned}
\end{equation*}
holds for all $n\in\nat$.
\end{theorem}


The identity relation presented in \cref{thm:LP-identity} provides meaningful insights only when we can guarantee the non-negativity of all its constituent terms. Under this condition, one could immediately infer, for instance, that \seqq{V_n} is non-increasing. The non-negativity of these terms is contingent on the selection of the parameter sequences \seqq{\zeta_n}, \seqq{\gamma_n}, \seqq{\mu_n}, and \seqq{\lambda_n}. To reduce the degrees of freedom and facilitate a clearer exposition, we constrain ourselves to a non-decreasing $\seqq{\lambda_n}$, and
\begin{align}
\mu_n=\frac{1}{\lambda_0}\lambda_n^2-\lambda_n
\label{eq:mu_def}
\end{align}
for all $n\in\mathbb{N}$. This implies $\mu_n\geq 0$ and offers a slightly less general algorithm, yet it encompasses all special cases in \cref{sec:special_cases}. We next state our assumptions on the parameter sequences.

\begin{assumption}\label{assum:parameters}
 Assume that $\varepsilon>0$, $\epsilon_0,\epsilon_1\geq 0$, $\lambda_0>0$, and that, for all $n\in\nat$, $\mu_n=\frac{1}{\lambda_0}\lambda_n^2-\lambda_n$ and the following hold: 
\renewcommand\theenumi\labelenumi
\renewcommand{\labelenumi}{{(\roman{enumi})}}
\begin{enumerate}[ref=\Cref{assum:parameters}~\theenumi]
\item $0\leq\zeta_n\leq 1-\epsilon_0$; \label{itm:assump-param-zeta}

\item $\varepsilon\leq\gamma_n\hbeta\leq 4-2\lambda_0-\varepsilon$;
\label{itm:assump-param-lambda-gamma}

\item $\lambda_{n+1}\geq\lambda_{n}$ and $\gamma_n\lambda_n-\gamma_{n-1}\lambda_{n-1}\leq \gamma_{n}\lambda_0-\epsilon_1$; 
\label{itm:assump-param-lambda}

\item $\hbeta\geq\beta$. \label{itm:assump-param-hbeta}

\end{enumerate}
\end{assumption}

\begin{remark}
\label{rem:Assumption-ii-remark}
    \labelcref{itm:assump-param-zeta} gives an upper bound for $\zeta_n$ to be less than or equal to 1. The variable, $\zeta_n$, multiplies $\elll{n-1}$ in the right-hand side of the safeguarding condition \eqref{eq:bound_on_deviations}, effectively contributing to limit the size of the ball from which the deviations $u_n$ and $v_n$ are selected. A consistent choice is $\zeta_n=1-\epsilon_0$. 
        \labelcref{itm:assump-param-lambda-gamma} sets requirements on the relation between the initial relaxation parameter $\lambda_0$ and the step size parameter $\gamma_n$. An alterative expression of the upper bound is given by
    \begin{align}
        4-\gamma_n\hbeta-2\lambda_0\geq\varepsilon,
        \label{eq:theta_help_bound}
    \end{align}
    implying that
    \begin{align}
        \gamma_n\hbeta\lambda_0\leq (4-\varepsilon)\lambda_0-2\lambda_0^2=-(\sqrt{2}\lambda_0-\sqrt{2})^2+2-\varepsilon\lambda_0\leq 2-\varepsilon\lambda_0.
        \label{eq:htheta_help_bound}
    \end{align}
    This inequality will be used to bound certain algorithm parameters. Note that, similarly to in \cite{latafat2017asymmetric,Giselsson2019NonlinearFS-journal}, we can allow for a $\gamma_n>\tfrac{2}{\beta}$ with the trade-off of using a relaxation parameter $\lambda_0<1$.
    \labelcref{itm:assump-param-lambda} states that $\seqq{\lambda_n}$ is non-decreasing, resulting in $\mu_n\geq 0$, and enforces a linear growth upper bound since $\gamma_n$ is both positive and upper bounded. We will later see that our algorithm can converge as $O(\tfrac{1}{\lambda_n^2})$, with this upper bound resulting in a best possible convergence rate of $O(\tfrac{1}{n^2})$. 
    Finally, \labelcref{itm:assump-param-hbeta} sets requirements on $\hbeta$. While the choice $\hbeta=\beta$ always works, selecting $\hbeta>\beta$ guarantees convergence of certain parameter sequences.
    Note that when $\lambda_0=1$ and $\hbeta=\beta$, it follows from \labelcref{itm:assump-param-lambda-gamma} that $\gamma_n\leq\tfrac{2-\varepsilon}{\beta}$, aligning with the conventional step size upper bound for forward--backward splitting.
\end{remark}

Our convergence analysis requires that specific parameter sequences are non-negative or, in certain cases, are lower bounded by a positive number. Before showing that \cref{assum:parameters} ensures this, we provide expressions for the following sequences defined in \cref{alg:main} in terms of $\lambda_n$, $\gamma_n$, and $\hbeta$:
\begin{equation}
\begin{aligned}
\alpha_n&=\tfrac{\mu_n}{\lambda_n+\mu_n}=\tfrac{\lambda_n-\lambda_0}{\lambda_n},\\
\balpha_n&=\tfrac{\gamma_n\mu_n}{\gamma_{n-1}(\lambda_n+\mu_n)}=\tfrac{\gamma_n}{\gamma_{n-1}}\tfrac{\lambda_n-\lambda_0}{\lambda_n},\\
\theta_n&=(4-\gamma_n\hbeta)(\lambda_n+\mu_n)-2\lambda_n^2=\tfrac{4-\gamma_n\hbeta-2\lambda_0}{\lambda_0}\lambda_n^2,\\
\htheta_n&=2\lambda_n+2\mu_n-\gamma_n\hbeta\lambda_n^2=\tfrac{2-\lambda_0\gamma_n\hbeta}{\lambda_0}\lambda_n^2,\\
\btheta_n&=\lambda_n+\mu_n-\lambda_n^2=\tfrac{1-\lambda_0}{\lambda_0}\lambda_n^2,\\
\ttheta_n&=(\lambda_n+\mu_n)\gamma_n\hbeta = \tfrac{\gamma_n\hbeta}{\lambda_0}\lambda_n^2.
\end{aligned}
\label{eq:parameters_for_quadratic_mu}
\end{equation}
Note that the final four quantities are quadratic in $\lambda_n$.

\begin{proposition}\label{rem:lower-bound-on-coefficients}
Consider the quantities defined in \cref{alg:main} and suppose that \cref{assum:parameters} holds. The parameter sequences \seqq{\theta_n}, \seqq{\htheta_n}, \seqq{\ttheta_n}, \seqq{\frac{\ttheta_n}{\htheta_n}}, \seqq{\frac{\htheta_n}{\theta_n}}, and \seqq{\frac{\theta_n}{\lambda_n^2}} are lower bounded by a positive constant  and $\seqq{\alpha_n}$ and $\seqq{\gamma_n(\lambda_n-\balpha_{n+1}\lambda_{n+1})-\epsilon_1}$ are non-negative. 
\end{proposition}
\begin{proof}
Let us first consider $\theta_n$, $\htheta_n$, and $\ttheta_n$. Then \cref{assum:parameters}, \eqref{eq:parameters_for_quadratic_mu}, \eqref{eq:theta_help_bound}, and \eqref{eq:htheta_help_bound} immediately imply that $\theta_n$, $\htheta_n$, and $\ttheta_n$ are lower bounded by a positive constant. Moreover, since $2-\lambda_0\gamma_n\hbeta\geq\varepsilon\lambda_0>0$ by \eqref{eq:htheta_help_bound}, we have
\begin{align*}
    \frac{\ttheta_n}{\htheta_n}=\frac{\gamma_n\hbeta}{2-\lambda_0\gamma_n\hbeta}\geq\frac{\gamma_n\hbeta}{2}\geq\frac{\varepsilon}{2}> 0
\end{align*}
and since from\eqref{eq:theta_help_bound}, $4-\gamma_n\hbeta-2\lambda_0\geq\varepsilon>0$, we have
\begin{align*}
\frac{\htheta_n}{\theta_n}=\frac{2-\lambda_0\gamma_n\hbeta}{4-\gamma_n\hbeta-2\lambda_0}\geq\frac{2-\lambda_0\gamma_n\hbeta}{4}\geq \frac{\lambda_0\varepsilon}{4}>0
\end{align*}
and
\begin{align*}
    \frac{\theta_n}{\lambda_n^2}=\frac{4-\gamma_n\hbeta-2\lambda_0}{\lambda_0}\geq \frac{\varepsilon}{\lambda_0}>0.
\end{align*}
That $\alpha_n\geq 0$ follows trivially from nonegativity of $\mu_n$ and that $\lambda_n>0$. Finally, 
\begin{align*}
    \gamma_n(\lambda_n-\balpha_{n+1}\lambda_{n+1}) 
    &= \gamma_n\lambda_n-\gamma_{n+1}\frac{\lambda_{n+1}-\lambda_0}{\lambda_{n+1}}\lambda_{n+1}\\
     &=\gamma_n\lambda_n-\gamma_{n+1}\lambda_{n+1}+\gamma_{n+1}\lambda_0\geq\epsilon_1
\end{align*}
by \labelcref{itm:assump-param-lambda}.
\end{proof}

In the following result, we introduce a so-called Lyapunov inequality that serves as the foundation of our main convergence results.

\begin{theorem}\label{thm:LP-ineq}
Suppose that \cref{assum:monotone_inclusion} and  \cref{assum:parameters} hold. Let $\xs$ be an arbitrary point in $\zer{A+C}$, and the sequences \seqq{\el{n}}, \seqq{V_n}, and \seqq{\phi_n}  be constructed in terms of the iterates obtained from \cref{alg:main}, as per \eqref{eq:ell_n}, \eqref{eq:LP-function-def}, and \eqref{eq:phi-def} respectively.
Then, for all $n\in\nat$,
\renewcommand\theenumi\labelenumi
\renewcommand{\labelenumi}{{(\roman{enumi})}}
\begin{enumerate}[noitemsep, ref=\Cref{thm:LP-ineq}~\theenumi]
    \item  the safeguarding upper bound 
    \begin{equation*}
                \zeta_{n+1}\elll{n} \geq \tfrac{\zeta_{n+1}\theta_n}{2}\norm{p_n-x_n+\alpha_n(x_n-p_{n-1}) +\tfrac{\gamma_n\hbeta\lambda_n^2}{\htheta_n}u_n - \tfrac{2\btheta_n}{\theta_n}v_n}_M^2\geq 0;
            \end{equation*}
   \label{itm:thm:LP-ineq-elln}

\item the term safeguarded by $\zeta_{n+1}\elll{n}$,
\begin{align*}
\left(\lambda_{n+1}+\mu_{n+1}\right)\lp\frac{\ttheta_{n+1}}{\htheta_{n+1}}\norm{u_{n+1}}_{M}^2 + \frac{\htheta_{n+1}}{\theta_{n+1}}\norm{v_{n+1}}_{M}^2\rp\geq 0;
\end{align*}
\label{itm:thm:safeguarded-by-elln}
    
\item $\phi_n\geq 0$, more specifically, if $\beta>0$:
\begin{equation*}
    \phi_n\geq\tfrac{\hbeta}{4}\norm{\tfrac{2}{\hbeta}(Cy_n-C\xs)+M(p_n-y_n)}_{M^{-1}}^2+\tfrac{\hbeta-\beta}{\hbeta\beta}\norm{Cy_n-C\xs}_{M^{-1}}^2\geq 0,
\end{equation*}
and if $\beta=0$:
\begin{equation*}
    \phi_n\geq\tfrac{\hbeta}{4}\norm{p_n-y_n}_M^2\geq 0;
\end{equation*}
\label{itm:thm:LP-ineq-phin}
    
    \item the Lyapunov function $V_n\geq 0$; 
    \label{itm:thm:LP-ineq-Vn}

    \item the following Lyapunov inequality holds \label{itm:thm:LP-ineq-LP-ineq}
        \begin{equation*}
            V_{n+1}+2\gamma_n(\lambda_n-\balpha_{n+1}\lambda_{n+1})\phi_n + (1-\zeta_n)\elll{n-1} \leq V_n.
        \end{equation*}
\end{enumerate}
\end{theorem}

\begin{proof}

\labelcref{itm:thm:LP-ineq-elln}. In view of \eqref{eq:bound_on_deviations} and since, by \cref{assum:parameters,rem:lower-bound-on-coefficients}, $2\mu_{n}\gamma_{n}\geq 0$, $\theta_n>0$, and $\zeta_{n+1}\geq 0$, the statement reduces to showing that 
\begin{align*}
\hvarphi_n \defeq\inpr{\tfrac{z_{n}-p_{n}}{\gamma_{n}}-\tfrac{z_{n-1}-p_{n-1}}{\gamma_{n-1}}}{p_{n}-p_{n-1}}_M+\tfrac{\hbeta}{4}\norm{p_{n}-y_{n}-(p_{n-1}-y_{n-1})}_M^2\geq 0.
\end{align*}
It follows from Step~\ref{alg:main:FB-step} of \cref{alg:main} that
\begin{equation}\label{eq:pn_resolvent}
    \frac{Mz_n-Mp_n}{\gamma_n}-Cy_n\in{Ap_n}.
\end{equation}
Combined with montonicity of $A$, this gives
\begin{align}
    0&\leq\inpr{\tfrac{Mz_n-Mp_n}{\gamma_n}-Cy_n-\tfrac{Mz_{n-1}-Mp_{n-1}}{\gamma_{n-1}}+Cy_{n-1}}{p_n-p_{n-1}}.\label{eq:monotonicity-consequtive-pn}
\end{align}
If $\beta=0$,  $C$ is constant, implying that the right-hand side reduces to the first term defining $\hvarphi_n$. Therefore, $\hvarphi_n\geq 0$ since it is constructed by adding two non-negative terms. It remains to show $\hvarphi_n\geq 0$ when $\beta>0$. From $\tfrac{1}{\beta}$-cocoercivity of $C$ w.r.t. $\norm{\cdot}_M$, we have
\begin{align}
    0&\leq\inpr{Cy_n-Cy_{n-1}}{y_n-y_{n-1}}-\tfrac{1}{\beta}\norm{Cy_n-Cy_{n-1}}_{M^{-1}}^2.\label{eq:cocoercivity-consecutive-pn}
\end{align}
Adding \eqref{eq:monotonicity-consequtive-pn} and \eqref{eq:cocoercivity-consecutive-pn} to form $\varphi_n$ gives
\begin{align}
    \varphi_n &\defeq\inpr{\tfrac{Mz_n-Mp_n}{\gamma_n}-Cy_n-\tfrac{Mz_{n-1}-Mp_{n-1}}{\gamma_{n-1}}+Cy_{n-1}}{p_n-p_{n-1}}\\
    &\quad+\inpr{Cy_n-Cy_{n-1}}{y_n-y_{n-1}}-\tfrac{1}{\beta}\norm{Cy_n-Cy_{n-1}}_{M^{-1}}^2\\
    &=\inpr{\tfrac{z_n-p_n}{\gamma_n}-\tfrac{z_{n-1}-p_{n-1}}{\gamma_{n-1}}}{p_n-p_{n-1}}_M-\tfrac{1}{\hbeta}\norm{Cy_n-Cy_{n-1}}_{M^{-1}}^2 \nonumber\\
    &\quad+\inpr{Cy_n-Cy_{n-1}}{y_n-y_{n-1}-(p_n-p_{n-1})}\nonumber\\
    &=\inpr{\tfrac{z_n-p_n}{\gamma_n}-\tfrac{z_{n-1}-p_{n-1}}{\gamma_{n-1}}}{p_n-p_{n-1}}_M + \tfrac{\hbeta}{4}\norm{p_n-p_{n-1}-y_n+y_{n-1}}_M^2\nonumber\\
    &\quad - \tfrac{\hbeta}{4}\norm{\tfrac{2}{\hbeta}(Cy_n-Cy_{n-1})+M(p_n-p_{n-1}-y_n+y_{n-1})}_{M^{-1}}^2\nonumber\\
    &=\hvarphi_n - \tfrac{\hbeta}{4}\norm{\tfrac{2}{\hbeta}(Cy_n-Cy_{n-1})+M(p_n-p_{n-1}-y_n+y_{n-1})}_{M^{-1}}^2,\label{eq:varphi-epsilon}
\end{align}
where we have used that
\begin{align}
    \inpr{s}{t}-\tfrac{1}{\delta}\norm{s}_{M^{-1}}^2=\tfrac{\delta}{4}\norm{t}_M^2-\tfrac{\delta}{4}\norm{\tfrac{2}{\delta}s-Mt}_{M^{-1}}^2
    \label{eq:inner_prod_norm_eq}
\end{align}
holds for all $t,s\in\PrimS$. Since $\varphi_n\geq 0$ by construction and $\hbeta\geq\beta> 0$, this implies that $\hvarphi_n\geq 0$.

\labelcref{itm:thm:safeguarded-by-elln}. This follows by \cref{assum:parameters} and \cref{rem:lower-bound-on-coefficients} that imply strict positiveness of $\lambda_{n+1}+\mu_{n+1}$, $\frac{\ttheta_{n+1}}{\htheta_{n+1}}$, and $\frac{\htheta_{n+1}}{\theta_{n+1}}$.

\labelcref{itm:thm:LP-ineq-phin}. Recall that
\begin{equation}
    \phi_n = \inpr{\tfrac{z_n-p_n}{\gamma_n}}{p_n-\xs}_M + \tfrac{\hbeta}{4}\norm{y_n-p_n}_M^2,
\end{equation}
as defined in \eqref{eq:phi-def}. Since $\xs\in\zer{A+C}$, we have
 $-C\xs\in{A\xs}$,
which combined with \eqref{eq:pn_resolvent} and montonicity of $A$ gives
\begin{align}
    0&\leq\inpr{\tfrac{Mz_n-Mp_n}{\gamma_n}-Cy_n+C\xs}{p_n-\xs}. \label{eq:monotonicity-pn-xs}
\end{align}
If $\beta=0$, $C$ is constant and the right hand side reduces to $\phi_n-\tfrac{\hbeta}{4}\|y_n-p_n\|_M^2$, which is non-negative for all $n\in\nat$ by \eqref{eq:monotonicity-pn-xs}. Let us now consider $\beta>0$. From $\tfrac{1}{\beta}$-cocoercivity of $C$ w.r.t. $\norm{\cdot}_M$, we have
\begin{align}
    0&\leq\inpr{Cy_n-C\xs}{y_n-\xs}-\tfrac{1}{\beta}\norm{Cy_n-C\xs}_{M^{-1}}^2. \label{eq:cocoercivity-pn-xs}
\end{align}
Construct $\wphi_n$ by adding \eqref{eq:monotonicity-pn-xs} and \eqref{eq:cocoercivity-pn-xs} to get 
\begin{align*}
    \wphi_n&\defeq \inpr{\tfrac{Mz_n-Mp_n}{\gamma_n}-Cy_n+C\xs}{p_n-\xs}\nonumber\\
    &\quad+ \inpr{Cy_n-C\xs}{y_n-\xs}-\tfrac{1}{\beta}\norm{Cy_n-C\xs}_{M^{-1}}^2\nonumber\\
    &= \inpr{\tfrac{z_n-p_n}{\gamma_n}}{p_n-\xs}_M + \inpr{Cy_n-C\xs}{y_n-p_n}-\tfrac{1}{\beta}\norm{Cy_n-C\xs}_{M^{-1}}^2, \\
    &= \inpr{\tfrac{z_n-p_n}{\gamma_n}}{p_n-\xs}_M + \inpr{Cy_n-C\xs}{y_n-p_n}-\tfrac{1}{\hbeta}\norm{Cy_n-C\xs}_{M^{-1}}^2\nonumber\\
    &\qquad+\left(\tfrac{1}{\hbeta}-\tfrac{1}{\beta}\right)\norm{Cy_n-C\xs}_{M^{-1}}^2\nonumber\\
    &= \inpr{\tfrac{z_n-p_n}{\gamma_n}}{p_n-\xs}_M + \tfrac{\hbeta}{4}\norm{y_n-p_n}_M^2 \nonumber\\
    &\qquad- \tfrac{\hbeta}{4}\norm{\tfrac{2}{\hbeta}(Cy_n-C\xs)+M(p_n-y_n)}_{M^{-1}}^2-\tfrac{\hbeta-\beta}{\hbeta\beta}\norm{Cy_n-C\xs}_{M^{-1}}^2\nonumber\\
    &= \phi_n - \tfrac{\hbeta}{4}\norm{\tfrac{2}{\hbeta}(Cy_n-C\xs)+M(p_n-y_n)}_{M^{-1}}^2-\tfrac{\hbeta-\beta}{\hbeta\beta}\norm{Cy_n-C\xs}_{M^{-1}}^2,
\end{align*}
where \eqref{eq:inner_prod_norm_eq} is used in the next to last equality. The result therefore follows since $\hbeta\geq\beta>0$ and $\wphi_n\geq0$ by construction.

\labelcref{itm:thm:LP-ineq-Vn}. 
Since $\el{n}\geq0$ by \labelcref{itm:thm:LP-ineq-elln}, $\phi_n\geq 0$ by \labelcref{itm:thm:LP-ineq-phin}, and the coefficients in front of $\phi_n$ in the definition of $V_n$ in \eqref{eq:LP-function-def} are non-negative by \cref{assum:parameters} and \cref{rem:lower-bound-on-coefficients}, we conclude that $V_n\geq0$.

\labelcref{itm:thm:LP-ineq-LP-ineq}.  By \cref{thm:LP-identity}, we have
\begin{align*}
    V_{n+1} + \el{n-1} &+ 2\gamma_n(\lambda_n-\balpha_{n+1}\lambda_{n+1})\phi_n\\
    &= V_n + (\lambda_{n}+\mu_{n})\lp\tfrac{\ttheta_{n}}{\htheta_{n}}\norm{u_{n}}_{M}^2 + \tfrac{\htheta_{n}}{\theta_{n}}\norm{v_{n}}_{M}^2\rp.
\end{align*}
Using this equality and \eqref{eq:bound_on_deviations} gives
\begin{align*}
    V_{n+1} &+ 2\gamma_n(\lambda_n-\balpha_{n+1}\lambda_{n+1})\phi_n +\elll{n-1}\leq V_n + \zeta_n\elll{n-1}.
\end{align*}
Moving $\zeta_n\el{n-1}$ to the other side  gives the desired result. This concludes the proof.
\end{proof}

This result demonstrates the feasibility of selecting $u_{n+1}$ and $v_{n+1}$ that meet the safeguarding condition. The obvious selection of $u_{n+1}=v_{n+1}=0$ is always viable, but we will provide in \cref{sec:special_cases} a nontrivial choice that consistently satisfies the condition and can enhance convergence. Furthermore, \cref{thm:LP-ineq} introduces a valuable Lyapunov inequality that will underpin our conclusions on convergence. Before stating these convergence results, we show boundedness of certain coefficient sequences.

\begin{lemma}\label{lem:boundedness-of-coefficients}
Consider the quantities defined in \cref{alg:main} and suppose that \cref{assum:parameters} holds. The sequences 
\seqq{\tfrac{\ttheta_n}{\htheta_n}}, 
\seqq{\tfrac{\lambda_n^2}{\htheta_n}}, 
\seqq{\tfrac{\btheta_n}{\theta_n}}, as well as
\seqq{\tfrac{(2-\gamma_n\hbeta)(\lambda_n+\mu_n)}{\theta_n}}
are bounded.
\end{lemma}
\begin{proof}
We have $\frac{\ttheta_n}{\htheta_n}\geq 0$ by \cref{rem:lower-bound-on-coefficients} and
\begin{align*}
    \frac{\ttheta_n}{\htheta_n}&=\frac{\gamma_n\hbeta}{2-\lambda_0\gamma_n\hbeta}\leq\frac{\gamma_n\hbeta}{\lambda_0\varepsilon}\leq\frac{4-2\lambda_0-\varepsilon}{\lambda_0\varepsilon}\leq\frac{4}{\lambda_0\varepsilon}
\end{align*}
by \eqref{eq:htheta_help_bound}. 
Further, $\tfrac{\lambda_n^2}{\htheta_n}\geq 0$ by \cref{rem:lower-bound-on-coefficients} and
\begin{align*}
\frac{\lambda_n^2}{\htheta_n}=\frac{\lambda_0}{(2-\lambda_0\gamma_n\hbeta)}\leq\frac{1}{\varepsilon}
\end{align*}
by \eqref{eq:htheta_help_bound}. Further, by \eqref{eq:theta_help_bound},
\begin{align*}
\left|\frac{\btheta_n}{\theta_n}\right| =\frac{|1-\lambda_0|}{4-\gamma_n\hbeta-2\lambda_0}\leq\frac{|1-\lambda_0|}{\varepsilon}
\end{align*}
and, since $\gamma_n\hbeta\in(0,4)$ by \labelcref{itm:assump-param-lambda-gamma},
\begin{align*}
\left|\frac{(2-\gamma_n\hbeta)(\lambda_n+\mu_n)}{\theta_n}\right|&\leq\frac{2(\lambda_n+\mu_n)}{\theta_n}=\frac{2}{4-\gamma_n\hbeta-2\lambda_0}\leq \frac{2}{\varepsilon}.
\end{align*}
This completes the proof.


\end{proof}

\begin{theorem}\label{thm:convergence}
    Suppose that \cref{assum:monotone_inclusion,assum:parameters} hold. Let $\xs$ be an arbitrary point in $\zer{A+C}$, and the sequences \seqq{\el{n}}, \seqq{V_n}, and \seqq{\phi_n} be constructed in terms of the iterates obtained from \cref{alg:main}, as per \eqref{eq:ell_n}, \eqref{eq:LP-function-def}, and \eqref{eq:phi-def} respectively. Then the following hold:
\renewcommand\theenumi\labelenumi
\renewcommand{\labelenumi}{{(\roman{enumi})}}
\begin{enumerate}[noitemsep, ref=\Cref{thm:convergence}~\theenumi]
    \item the sequence \seqq{V_n} is convergent and $\elll{n}\leq V_{n+1}\leq\|x_0-\xs\|_M^2$;\label{itm:thm:Vn_convergent_elln_bounded}
    \item if $\seqq{\lambda_n}$ increasing and $\lambda_n\to\infty$ as $n\to\infty$, then
    \begin{align*}
        \norm{p_n-x_n+\alpha_n(x_n-p_{n-1}) +\tfrac{\gamma_n\hbeta\lambda_n^2}{\htheta_n}u_n - \tfrac{2\btheta_n}{\theta_n}v_n}_M^2\leq\frac{2\lambda_0\|x_0-\xs\|_M^2}{(4-\gamma_n\hbeta-2\lambda_0)\lambda_n^2};
    \end{align*}\label{itm:thm:norm_convergence}
    \item if $\epsilon_0>0$, then \seqq{\elll{n}} is summable;\label{itm:thm:elln_summable}
    \item if $\epsilon_0>0$, then $\lambda_n u_n\to 0$, $\lambda_nv_n\to 0$, and $x_{n+1}-x_n\to 0$ as $n\to\infty$;\label{itm:thm:xn_residual}
    \item if $\epsilon_1>0$, then \seqq{\phi_n} is summable;\label{itm:thm:phin_summable}
    \item if $\epsilon_1>0$ and $\hbeta>\beta$, then $\|y_n-p_n\|_M^2$ is summable;\label{itm:thm:yn-pn_summable}
    \item if $\epsilon_0,\epsilon_1>0$, $\seqq{\lambda_n}$ is bounded, and $p_n-x_n\to 0$, $y_n-p_n\to 0$, and $z_n-p_n\to 0$ as $n\to\infty$, then $p_n\weakto\xs$;\label{itm:thm:sequence_convergence}
    \item if $\epsilon_0,\epsilon_1>0$ and $\seqq{\lambda_n}$ is constant, then $p_n\weakto\xs$. \label{itm:thm:constant_lambda}
\end{enumerate}
\end{theorem}
\begin{proof}
We base our convergence results on 
\begin{align}
    V_{n+1} &+ 2\gamma_n(\lambda_n-\balpha_{n+1}\lambda_{n+1})\phi_n +\prt{1-\zeta_n}\elll{n-1}\leq V_n,\label{eq:Lyap_ineq}
\end{align}
from \cref{thm:LP-ineq}.

\labelcref{itm:thm:Vn_convergent_elln_bounded}. Recall that the sequences  \seqq{\elll{n}}, \seqq{V_n}, and \seqq{\phi_n} are non-negative by \cref{thm:LP-ineq}. Additionally, by \ref{itm:assump-param-zeta} and \cref{rem:lower-bound-on-coefficients} respectively, the quantities $1-\zeta_n$ and $\gamma_n(\lambda_n-\balpha_{n+1}\lambda_{n+1})$ are non-negative  for all $n\in\nat$; and thus, the quantity $2\gamma_n(\lambda_n-\balpha_{n+1}\lambda_{n+1})\phi_n +\prt{1-\zeta_n}\elll{n-1}$ is non-negative for all $n\in\nat$. Therefore, by \cite[Lemma 5.31]{bauschke2017convex} the sequence 
\seqq{V_n} converges and $V_{n+1}\leq V_n$ for all $n\in\mathbb{N}$ and since $\lambda_{n+1}\gamma_{n+1}\alpha_{n+1}\geq 0$ by \cref{assum:parameters} and \cref{rem:lower-bound-on-coefficients},
\begin{align}
    \elll{n}\leq V_{n+1} \leq V_n\leq\ldots\leq V_0=\|x_0-\xs\|_M^2.
    \label{eq:elln_lessthan_V0}
\end{align}

\labelcref{itm:thm:norm_convergence}. \labelcref{itm:thm:LP-ineq-elln} states that
\begin{align*}
    \frac{\theta_n}{2}\norm{p_n-x_n+\alpha_n(x_n-p_{n-1}) +\tfrac{\gamma_n\hbeta\lambda_n^2}{\htheta_n}u_n - \tfrac{2\btheta_n}{\theta_n}v_n}_M^2\leq\elll{n},
\end{align*}
where $\elll{n}\leq \|x_0-\xs\|_M^2$ by \labelcref{itm:thm:Vn_convergent_elln_bounded}. Inserting the definition of $\theta_n$ and rearranging gives the result.

\labelcref{itm:thm:elln_summable}. That $\epsilon_0>0$ implies that $1-\zeta_n>\epsilon_0>0$ by \labelcref{itm:assump-param-zeta} and a telescope summation of \eqref{eq:Lyap_ineq} gives summability of $\seqq{\elll{n}}$.

\labelcref{itm:thm:xn_residual}.
To show that $\seqq{\lambda_nu_n}$ and $\seqq{\lambda_nv_n}$ converge to 0, we note that due to \eqref{eq:bound_on_deviations}, the summability of \seqq{\elll{n}} implies summability of 
\begin{equation}\label{eq:deviations_summability}
    \seqq{\tfrac{\lambda_{n+1}+\mu_{n+1}}{\lambda_{n+1}^2}\lp\tfrac{\ttheta_{n+1}}{\htheta_{n+1}}\norm{\lambda_{n+1}u_{n+1}}_{M}^2+\tfrac{\htheta_{n+1}}{\theta_{n+1}}\norm{\lambda_{n+1}v_{n+1}}_{M}^2\rp}.
\end{equation}
Hence, as, for all $n\in\nat$, by \cref{rem:lower-bound-on-coefficients} and \cref{assum:parameters} the coefficients in the expression above are strictly positive, the sequences \seqq{\lambda_nu_n} and \seqq{\lambda_nv_n} must be convergent to zero.

Next, we show convergence to zero of \seqq{x_{n+1}-x_n}. Since \seqq{\elll{n}} is summable, \labelcref{itm:thm:LP-ineq-elln} implies that
\begin{align*}
    \prt{\tfrac{\theta_n}{2}\norm{p_n-x_n+\alpha_n(x_n-p_{n-1}) +\tfrac{\gamma_n\hbeta\lambda_n^2}{\htheta_n}u_n - \tfrac{2\btheta_n}{\theta_n}v_n}_M^2}_{n\in\nat}
\end{align*}
is summable. Using \cref{lemma:ell_n-identical-expressions} to replace the expression inside the norm above by \labelcref{itm:lemma2-iii} and taking the factor $\tfrac{1}{\lambda_n}$ out of the norm, we get
\begin{align*}
    \prt{\tfrac{\theta_n}{2\lambda_n^2}\norm{x_{n+1}-x_n+\tfrac{\ttheta_n}{\htheta_n}\lambda_nu_n+\tfrac{(2-\gamma_n\hbeta)(\lambda_n+\mu_n)}{\theta_n}\lambda_nv_n}_M^2}_{n\in\nat},
\end{align*}
which is a summable sequence too. Since, by \cref{rem:lower-bound-on-coefficients}, $\tfrac{\theta_n}{2\lambda_n^2}$ is lower bounded by a positive constant and the coefficients multiplying $\lambda_nu_n$ and $\lambda_nv_n$ are bounded by \cref{lem:boundedness-of-coefficients}, we conclude, since $\lambda_nu_n\to 0$ and $\lambda_nv_n\to 0$ as $n\to\infty$, that $x_{n+1}-x_n\to 0$ as $n\to\infty$.

\labelcref{itm:thm:phin_summable}. \cref{rem:lower-bound-on-coefficients} implies that $2\gamma_n(\lambda_n-\balpha_{n+1}\lambda_{n+1})\geq 2\epsilon_1>0$ and a telescope summation of \eqref{eq:Lyap_ineq} gives summability of $\seqq{\phi_n}$.

\labelcref{itm:thm:yn-pn_summable}. Let $\beta=0$. Then \labelcref{itm:thm:LP-ineq-phin} immediately gives the result due to summability of $\seqq{\phi_n}$. Let $\beta>0$. Then \labelcref{itm:thm:LP-ineq-phin} and $\hbeta>\beta$ imply that 
\begin{align*}
    \norm{\tfrac{2}{\hbeta}(Cy_n-C\xs)+M(p_n-y_n)}_{M^{-1}}^2\qquad{\hbox{and}}\qquad\norm{Cy_n-C\xs}_{M^{-1}}^2
\end{align*}
are summable. Since
\begin{align*}
    \|p_n-y_n\|_M^2&=\|\tfrac{2}{\hbeta}(Cy_n-C\xs)-\tfrac{2}{\hbeta}(Cy_n-C\xs)+M(p_n-y_n)\|_{M^{-1}}^2\\
    &\leq 2\norm{\tfrac{2}{\hbeta}Cy_n-C\xs+M(p_n-y_n)}_{M^{-1}}^2+2\norm{Cy_n-C\xs}_{M^{-1}}^2,
\end{align*}
we conclude that \seqq{\|p_n-y_n\|_M^2} is summable.

\labelcref{itm:thm:sequence_convergence}. We first show that $\|x_n-\xs\|_M^2$ converges. From \labelcref{itm:thm:Vn_convergent_elln_bounded}, we know that 
\begin{equation*}
    V_{n+1} \defeq \norm{x_{n+1}-\xs}_M^2 + 2\lambda_{n+1}\gamma_{n+1}\alpha_{n+1}\phi_n + \el{n},
\end{equation*}
converges. Since $\seqq{\lambda_n}$ is bounded so is $\lambda_{n+1}\gamma_{n+1}\alpha_{n+1}$ and by \labelcref{itm:thm:elln_summable,itm:thm:phin_summable} we conclude that $2\lambda_{n+1}\gamma_{n+1}\alpha_{n+1}\phi_n + \el{n}\to 0$ as $n\to\infty$. This implies that $\norm{x_{n+1}-\xs}_M^2$ converges. 

Now, since $\|p_n-x_n\|_M^2\to 0$ as $n\to\infty$ and $\|x_n-\xs\|\leq D$ for all $n\in\mathbb{N}$ and some $D\in(0,\infty)$, we conclude that
\begin{align*}
    \left|\|p_n-\xs\|_M^2-\|x_n-\xs\|_M^2\right|&=\left|\|p_n-x_n\|_M^2+2\langle p_n-x_n,x_n-\xs\rangle\right|\\
    &\leq\|p_n-x_n\|_M^2+2\left|\langle p_n-x_n,x_n-\xs\rangle\right|\\
    &\leq\|p_n-x_n\|_M^2+2\|p_n-x_n\|\|x_n-\xs\|\\
    &\leq\|p_n-x_n\|_M^2+2\|p_n-x_n\|D\to 0
\end{align*}
as $n\to\infty$. Therefore also $\|p_n-\xs\|^2$ converges and $p_n$ has weakly convergent subsequences. Let $(p_{n_k})_{k\in\nat}$ be one such subsequence with weak limit point $\bar{x}$ and construct corresponding subsequences $(y_{n_k})_{k\in\nat}$,  $(z_{n_k})_{k\in\nat}$, and $(\gamma_{n_k})_{k\in\nat}$. Now, Step~\ref{alg:main:FB-step} in \cref{alg:main} can equivalently be written as
\begin{align*}
    Mp_{n_k}+\gamma_{n_k}Ap_{n_k}\ni Mz_{n_k}-\gamma_{n_k}Cy_{n_k},
\end{align*}
which is equivalent to that
\begin{align*}
    Cp_{n_k}+Ap_{n_k}\ni \frac{1}{\gamma_{n_k}}M(z_{n_k}-p_{n_k})+(Cp_{n_k}-Cy_{n_k}).
\end{align*}
The right hand side converges to 0 as $k\to\infty$ since $z_{n_k}-p_{n_k}\to 0$ and $p_{n_k}-y_{n_k}\to 0$ as $k\to\infty$ and due to Lipschitz continuity of $C$, the uniform positive lower bound on $\gamma_n$ in \labelcref{itm:assump-param-lambda-gamma}, and boundedness of $M\in\posop{\PrimS}$. By weak-strong closedness of the maximal monotone operator $(A+C)$ (which is maximally monotone since $C$ has full domain) the limit point satisfies $0\in (A+C)\bar{x}$ by \cite[Proposition~20.38]{bauschke2017convex}. The weak convergence result now follows from \cite[Lemma~2.47]{bauschke2017convex}.

\labelcref{itm:thm:constant_lambda}. In view of \labelcref{itm:thm:sequence_convergence}, it is enough to show that $p_n-x_n\to 0$, $y_n-p_n\to 0$, and $z_n-p_n\to 0$ as $n\to\infty$. Since $\seqq{\lambda_n}$ is constant, $\mu_n=\frac{1}{\lambda_0}\lambda_n^2-\lambda_n=0$ and $\alpha_n=0$. Summability of $\elll{n}$ therefore implies through \labelcref{itm:thm:LP-ineq-elln} that
\begin{align*}
    \prt{\tfrac{\theta_n}{2}\norm{p_n-x_n +\tfrac{\gamma_n\hbeta\lambda_n^2}{\htheta_n}u_n - \tfrac{2\btheta_n}{\theta_n}v_n}_M^2}_{n\in\nat}
\end{align*}
is summable. Since $\theta_n$ is lower bounded by a positive constant due to \cref{rem:lower-bound-on-coefficients} and the coefficients in front of $u_n$ and $v_n$ are bounded due to \cref{lem:boundedness-of-coefficients}, we conclude, since $u_n\to 0$ and $v_n\to 0$ by \labelcref{itm:thm:xn_residual} and \cref{assum:parameters}, that $p_n-x_n\to 0$ as $n\to\infty$. From the $x_n$ update,
\begin{align*}
    x_{k+1}=x_k+\lambda_n(p_n-z_n),
\end{align*}
\labelcref{itm:thm:xn_residual}, and since $\seqq{\lambda_n}$ is constant, we conclude that $p_n-z_n\to 0$ as $n\to\infty$. Finally, from the $y_n$ update,
\begin{align*}
    y_n=x_n+u_n,
\end{align*}
and since $u_n\to 0$, we conclude that $y_n-x_n\to 0$, which implies that $y_n-p_n\to 0$ as $n\to\infty$. This concludes the proof.
\end{proof}


We could derive convergence properties for other quantities involved, yet we limit our discussion to these results as they are sufficient for our needs for the special cases. Notably, the conclusion in \labelcref{itm:thm:constant_lambda} aligns with a similar result presented in the authors' previous work \cite{nofob-inc}. This is due to $\mu_n=0$, causing our algorithm to reduce to the one presented in that work.

\section{A special case}
\label{sec:special_cases}

The safeguarding condition in \eqref{eq:bound_on_deviations} typically requires the evaluation of four norms and a scalar product. However, if the vectors inside the norms are parallel, the number of norm evaluations is reduced. This section introduces an algorithm wherein we choose $u_{n}$ and $v_{n}$ to ensure $y_n=z_n$ for all $n\in\mathbb{N}$ and such that the safeguarding condition reduces to a scalar condition that is readily verified offline. The algorithm we propose is as follows:
\begin{equation}
\begin{aligned}
		     y_n&=x_n+\tfrac{\lambda_n-\lambda_0}{\lambda_n}(y_{n-1}-x_n)+u_n\\
    p_n &= \prt{M+\gamma A}^{-1}\prt{M y_n - \gamma C y_n}\\
		    x_{n+1} &= x_n + \lambda_n(p_n - y_n) + (\lambda_n-\lambda_0)(y_{n-1}-p_{n-1})\\
      u_{n+1}&=\kappa_n\tfrac{4-\gamma\hbeta-2\lambda_0}{2}\lp p_n-x_n+\tfrac{\lambda_n-\lambda_0}{\lambda_n}(x_n-p_{n-1}) - \tfrac{2-\gamma\hbeta-2\lambda_0}{4-\gamma\hbeta-2\lambda_0} u_n\rp,
\end{aligned} 
\label{eq:alg-parallel-safeguarding}
\end{equation}
where $y_{-1}=p_{-1}=x_0$, $u_0=0$, and a constant step size $\gamma>0$ is used. With a constant step size, \labelcref{itm:assump-param-lambda} reduces to 
\begin{align}
   \lambda_n\leq \lambda_{n+1}\leq\lambda_n+\lambda_0-\frac{\epsilon_1}{\gamma}=(2+n)\lambda_0-(n+1)\frac{\epsilon_1}{\gamma},
   \label{eq:lambda-bounds}
\end{align}
which gives an increasing \seqq{\lambda_n} sequence that grows at most linearly in $n$. 

Prior to presenting the convergence results for this algorithm, we specify a particular form for the sequence $\seqq{\lambda_n}$. This form separates the growth in $n$ from the selection of $\lambda_0$.
\begin{assumption}
Let $\lambda_0>0$. Assume that $f:{\mathrm{dom}} f\to\reals$, with $\mathrm{int\,dom} f\supseteq \{x\in\reals:x\geq 0\}$, is differentiable (on the interior of its domain), concave, and non-decreasing, and satisfies $f(0)=1$ and $f^\prime(0)\in[0,1]$. Let, for all $n\in\nat$,
\begin{align*}
    \lambda_n = f(n)\lambda_0.
\end{align*}  
\label{assum:lambda-sequence}
\end{assumption}
\begin{proposition}
    Suppose that \cref{assum:lambda-sequence} holds, then \eqref{eq:lambda-bounds} and \labelcref{itm:assump-param-lambda} hold with $\epsilon_1=0$. Suppose in addition that $f^\prime(0)<1$, then there exists $\epsilon_1>0$ such that \eqref{eq:lambda-bounds} and \labelcref{itm:assump-param-lambda} hold.
\label{prop:lambda-assumption-holds}
\end{proposition}
\begin{proof}
That $f$ is non-decreasing trivially implies $\lambda_{n}\leq\lambda_{n+1}$. Concavity implies $f^{\prime}(x)\leq f^{\prime}(0)$ for all $x\geq 0$. Therefore
    \begin{align*}
        f(n) =1+\int_{0}^{n}f^\prime(x) dx\leq1+\int_{0}^{n}f^\prime(0) dx\leq 1+n
    \end{align*}
    and $\lambda_{n+1}\leq (2+n)\lambda_0$. Let $a:=f^{\prime}(0)\in[0,1)$, then
    \begin{align*}
        f(n) =1+\int_{0}^{n}f^\prime(x) dx\leq 1+\int_{0}^{n}f^\prime(0) dx=1+an=1+n-(1-a)n
    \end{align*}
    and with $\epsilon_1=(1-a)\gamma\lambda_0>0$, we get $\lambda_{n+1}\leq (2+n)\lambda_0-(n+1)\tfrac{\epsilon_1}{\gamma}$, as desired.
\end{proof}
\begin{example}
 Examples of functions $f$ that satisfy \cref{assum:lambda-sequence} for which an $\epsilon_1>0$ exists include functions that, for all $n\in\nat$, satisfy $f(n)=(1+n)^e$ with $e\in[0,1)$,  $f(n)=\tfrac{\log(n+2)}{\log(2)}$, and $f(n)=1$. The choice $f(n)=(1+n)$ requires that $\epsilon_1=0$.
 \label{ex:lambda-sequence-examples}
\end{example}

 We will use this construction of $\seqq{\lambda_n}$ throughout this section and specialize \cref{assum:parameters} as follows.
\begin{assumption}\label{assum:parameters-sc}
 Assume that $\varepsilon>0$, $\epsilon_0\geq 0$, $\lambda_0>0$, and that, for all $n\in\nat$, $\mu_n=\frac{1}{\lambda_0}\lambda_n^2-\lambda_n$ and the following hold: 
\renewcommand\theenumi\labelenumi
\renewcommand{\labelenumi}{{(\roman{enumi})}}
\begin{enumerate}[ref=\Cref{assum:parameters-sc}~\theenumi]
\item $0\leq\kappa_n^2\leq 1-\epsilon_0$; \label{itm:assump-param-zeta-sc}

\item $\varepsilon\leq\gamma\hbeta\leq 4-2\lambda_0-\varepsilon$;
\label{itm:assump-param-lambda-gamma-sc}

\item \seqq{\lambda_n} is given by \cref{assum:lambda-sequence}; 
\label{itm:assump-param-lambda-sc}

\item $\hbeta\geq\beta$. \label{itm:assump-param-hbeta-sc}

\end{enumerate}
\end{assumption}

The differences to \cref{assum:parameters} are that $\zeta_{n}$ in \labelcref{itm:assump-param-zeta} has been replaced with $\kappa_n^2$ in \labelcref{itm:assump-param-zeta-sc} and that
\labelcref{itm:assump-param-lambda} has been replaced with \labelcref{itm:assump-param-lambda-sc}. If, e.g., $\kappa_n^2\leq\zeta_{n+1}$, \cref{prop:lambda-assumption-holds} implies that  \cref{assum:parameters} holds if \cref{assum:parameters-sc} does.

We are ready to state our convergence results for the algorithm in \eqref{eq:alg-parallel-safeguarding}.
\begin{proposition}\label{prop:convergence-sc}
    Suppose that \cref{assum:monotone_inclusion,assum:parameters-sc} hold. 
    Then the following hold for \eqref{eq:alg-parallel-safeguarding}:
    \renewcommand\theenumi\labelenumi
    \renewcommand{\labelenumi}{{(\roman{enumi})}}
    \begin{enumerate}[noitemsep, ref=\Cref{prop:convergence-sc}~\theenumi]
    \item if $f(n)\to\infty$ as $n\to\infty$, then
    \begin{align*}
        \norm{p_n-x_n+\tfrac{\lambda_n-\lambda_0}{\lambda_n}(x_n-p_{n-1}) - \tfrac{2-\gamma\hbeta-2\lambda_0}{4-\gamma\hbeta-2\lambda_0}u_n}_M^2\leq\frac{2\|y_0-\xs\|_M^2}{(4-\gamma\hbeta-2\lambda_0)\lambda_0f(n)^2};
    \end{align*}\label{itm:prop:norm_convergence-sc}
    \item if $\hbeta>\beta$ and $f^{\prime}(0)<1$, then \seqq{\|p_n-y_n\|^2} is summable; \label{itm:prop:yn-pn_summable-sc}
    \item if $\seqq{\kappa_n^2}$ is upper bounded by a constant less than 1 and $f^{\prime}(0)=0$ (i.e., \seqq{\lambda_n} is constant), then $p_n\weakto\xs\in\zer{A+C}$. \label{itm:prop:constant_lambda-sc}
\end{enumerate}
\end{proposition}
\begin{proof}
    We first show that the algorithm is a special case of \cref{alg:main}. First note that 
    $\gamma_n=\gamma$ implies that $\alpha_n=\balpha_n$ for all $n\in\mathbb{N}$. Let 
    \begin{align*}
    v_n&=\frac{2-\gamma_n\hbeta}{2-\lambda_0\gamma_n\hbeta}u_n
\end{align*}
for all $n\in\mathbb{N}$, which implies that
\begin{align}
    \frac{\btheta_n\gamma_n\hbeta}{\htheta_n}u_n+v_n=\frac{(1-\lambda_0)\gamma_n\hbeta}{2-\lambda_0\gamma_n\hbeta}u_n+\frac{2-\gamma_n\hbeta}{2-\lambda_0\gamma_n\hbeta}u_n=u_n
    \label{eq:un+vn_gives_un}
\end{align}
since $2-\lambda_0\gamma_n\hbeta>0$ by \eqref{eq:htheta_help_bound}. Let us show by induction that this implies $y_n=z_n$ for all $n\in\mathbb{N}$ in \cref{alg:main}.  Since $y_{-1}=z_{-1}=p_{-1}=x_0$ and $u_0=0$, we get $y_0=z_0$. Now, assume that $y_k=z_k$ for all $k\in\{-1,\ldots,n\}$, then, since $\alpha_n=\balpha_n$ and due to \eqref{eq:un+vn_gives_un},
\begin{align*}
        z_{n+1} &=y_n-x_n+\alpha_n(p_{n-1}-x_n)+\balpha_n(y_{n-1}-p_{n-1})+\tfrac{\btheta_n\gamma_n\hbeta}{\htheta_n}u_n+v_n\\
        &=y_n-x_n+\alpha_n(p_{n-1}-x_n)+u_n
        =y_{n+1}.
\end{align*}
Therefore, the $z_{n+1}$ update of \cref{alg:main} can be removed and all $z_n$ instances replaced by $y_n$ in \eqref{eq:alg-parallel-safeguarding}. Moreover, the $y_n$ and $x_n$ updates of \eqref{eq:alg-parallel-safeguarding} are obtained from the corresponding sequences in \cref{alg:main} by inserting $\alpha_n=\tfrac{\lambda_n-\lambda_0}{\lambda_n}$.

It remains to show that the $u_{n+1}$ update satisfies the safeguarding condition. We use \labelcref{itm:thm:LP-ineq-elln}, $\alpha_n=\tfrac{\lambda_n-\lambda_0}{\lambda_n}$, and the equality
\begin{equation}
\begin{aligned}
\frac{\gamma_n\hbeta\lambda_n^2}{\htheta_n}u_n - \frac{2\btheta_n}{\theta_n}v_n&=\left(\frac{\gamma_n\hbeta\lambda_n^2}{\htheta_n} - \frac{2\btheta_n}{\theta_n}\frac{2-\gamma_n\hbeta}{2-\lambda_0\gamma_n\hbeta}\right)u_n\\
&=\left(\frac{\lambda_0\gamma_n\hbeta}{2-\lambda_0\gamma_n\hbeta} - \frac{2(1-\lambda_0)}{4-\gamma_n\hbeta-2\lambda_0}\frac{2-\gamma_n\hbeta}{2-\lambda_0\gamma_n\hbeta}\right)u_n\\
&=\frac{\gamma_n\hbeta\lambda_0(4-\gamma_n\hbeta-2\lambda_0)-2(1-\lambda_0)(2-\gamma_n\hbeta)}{(2-\lambda_0\gamma_n\hbeta)(4-\gamma_n\hbeta-2\lambda_0)}u_n \\
&=\frac{\gamma_n\hbeta\lambda_0(2-\gamma_n\hbeta-2\lambda_0)-2(2-\gamma_n\hbeta-2\lambda_0)}{(2-\lambda_0\gamma_n\hbeta)(4-\gamma_n\hbeta-2\lambda_0)}u_n \\
&=\frac{(\gamma_n\hbeta\lambda_0-2)(2-\gamma_n\hbeta-2\lambda_0)}{(2-\lambda_0\gamma_n\hbeta)(4-\gamma_n\hbeta-2\lambda_0)}u_n \\
&=-\frac{(2-\gamma_n\hbeta-2\lambda_0)}{(4-\gamma_n\hbeta-2\lambda_0)}u_n,
\end{aligned}
\label{eq:un_vn_coefficient_equality}
\end{equation}
to conclude that
\begin{equation*}
            \begin{aligned}
                \elll{n} 
                &\geq\tfrac{\theta_n}{2}\norm{p_n-x_n+\alpha_n(x_n-p_{n-1}) +\tfrac{\gamma_n\hbeta\lambda_n^2}{\htheta_n}u_n - \tfrac{2\btheta_n}{\theta_n}v_n}_M^2\\
                &=\tfrac{\theta_n}{2}\norm{p_n-x_n+\tfrac{\lambda_n-\lambda_0}{\lambda_n}(x_n-p_{n-1}) -\tfrac{(2-\gamma_n\hbeta-2\lambda_0)}{(4-\gamma_n\hbeta-2\lambda_0)}u_n}_M^2.
            \end{aligned}
            \end{equation*}
Now, since $\lambda_{n+1}+\mu_{n+1}=\tfrac{\lambda_n^2}{\lambda_0}$, we conclude that if 
\begin{align*}
      &\tfrac{\lambda_n^2}{\lambda_0}\lp\tfrac{\ttheta_{n+1}}{\htheta_{n+1}}\norm{u_{n+1}}_{M}^2 + \tfrac{\htheta_{n+1}}{\theta_{n+1}}\norm{v_{n+1}}_{M}^2\rp\\
      &\qquad\qquad\qquad\leq \zeta_{n+1}\tfrac{\theta_n}{2}\norm{p_n-x_n+\tfrac{\lambda_n-\lambda_0}{\lambda_n}(x_n-p_{n-1}) -\tfrac{(2-\gamma_n\hbeta-2\lambda_0)}{(4-\gamma_n\hbeta-2\lambda_0)}u_n}_M^2,
\end{align*}
the safeguarding condition in \cref{alg:main} is satisfied. The vectors $u_{n+1}$ and $v_{n+1}$ are scalars times the quantity inside this norm. Therefore, the safeguarding condition reduces to the scalar condition
\begin{align*}
    \frac{\lambda_n^2}{\lambda_0}\kappa_n^2\frac{(4-\gamma\hbeta-2\lambda_0)^2}{4}\lp\frac{\ttheta_{n+1}}{\htheta_{n+1}} + \frac{\htheta_{n+1}}{\theta_{n+1}}\frac{(2-\gamma_n\hbeta)^2}{(2-\lambda_0\gamma_n\hbeta)^2}\rp \leq \zeta_{n+1}\frac{\theta_n}{2}.
\end{align*}
Inserting the quantities in \eqref{eq:parameters_for_quadratic_mu} and $\mu_n=\tfrac{1}{\lambda_0}\lambda_n^2-\lambda_n$ and multiplying by $\tfrac{2}{\theta_n}>0$ gives
\begin{align*}
    \kappa_n^2\frac{(4-\gamma\hbeta-2\lambda_0)}{2}\lp\frac{\gamma_n\hbeta}{2-\lambda_0\gamma_n\hbeta} + \frac{2-\lambda_0\gamma_n\hbeta}{(4-\gamma_n\hbeta-2\lambda_0)}\frac{(2-\gamma_n\hbeta)^2}{(2-\lambda_0\gamma_n\hbeta)^2}\rp \leq \zeta_{n+1}.
\end{align*}
The left-hand side satisfies
\begin{align*}
    &\kappa_n^2\frac{(4-\gamma\hbeta-2\lambda_0)}{2}\lp\frac{\gamma_n\hbeta}{2-\lambda_0\gamma_n\hbeta} + \frac{2-\lambda_0\gamma_n\hbeta}{(4-\gamma_n\hbeta-2\lambda_0)}\frac{(2-\gamma_n\hbeta)^2}{(2-\lambda_0\gamma_n\hbeta)^2}\rp\\
     &=\frac{(4-\gamma\hbeta-2\lambda_0)}{2}\frac{\kappa_n^2}{2-\lambda_0\gamma_n\hbeta}\frac{\lp\gamma_n\hbeta(4-\gamma_n\hbeta-2\lambda_0)+(2-\gamma_n\hbeta)^2\rp}{(4-\gamma_n\hbeta-2\lambda_0)}\\
      &=\frac{\kappa_n^2}{2(2-\lambda_0\gamma_n\hbeta)}\lp\gamma_n\hbeta(4-\gamma_n\hbeta-2\lambda_0)+4-4\gamma_n\hbeta +(\gamma_n\hbeta)^2\rp\\
      &=\frac{\kappa_n^2}{2(2-\lambda_0\gamma_n\hbeta)}\lp 4-2\gamma_n\hbeta\lambda_0\rp\\
      &=\kappa_n^2,
\end{align*}
leading to the safeguarding condition
\begin{align*}
    \kappa_n^2\leq\zeta_{n+1}
\end{align*}
which is satisfied by letting $\zeta_{n+1}=\kappa_n^2$. 

Using \cref{prop:lambda-assumption-holds} and the choice $\zeta_{n+1}=\kappa_n^2$, we conclude that \cref{assum:parameters} holds since \cref{assum:parameters-sc} does and we can use \cref{thm:convergence} to prove convergence.

That \labelcref{itm:prop:norm_convergence-sc} holds follows from \labelcref{itm:thm:norm_convergence} since $y_0=x_0$ and by updating the norm expression using \eqref{eq:un_vn_coefficient_equality}.

\labelcref{itm:prop:yn-pn_summable-sc} follows from \labelcref{itm:thm:yn-pn_summable} due to \cref{prop:lambda-assumption-holds} that ensures $\epsilon_1>0$.

\labelcref{itm:prop:constant_lambda-sc} follows from \labelcref{itm:thm:constant_lambda} due to \cref{prop:lambda-assumption-holds} and that $\kappa_n^2=\zeta_{n+1}$ is upper bounded by a constant less than 1, which implies $\epsilon_0>0$.

\end{proof}
\begin{remark}
The algorithm produces points that satisfy
    \begin{align*}
        (A+C)p_n\ni \tfrac{1}{\gamma}M(y_n-p_n)-C(y_n-p_n)
    \end{align*}
    and \labelcref{itm:prop:yn-pn_summable-sc} implies, if $f^\prime(0)<1$ and $\hbeta>\beta$, that $M(y_n-p_n)-C(y_n-p_n)\to 0$ as $n\to\infty$ due to boundedness of $M$ and Lipschitz continuity of $C$. Although $\seqq{p_n}$ may not converge, it satisfies the monotone inclusion \eqref{eq:monotone_inclusion} in the limit. If in addition $\|p_n-\xs\|_M$ converges, we can conclude that $p_n\weakto\xs\in\zer{A+C}$. 
\end{remark}

A special case of \eqref{eq:alg-parallel-safeguarding} that we will evaluate numerically in \cref{sec:numerical-experiments} is found by letting $\lambda_0=1$, and, for all $n\in\nat$, $\lambda_n=\lambda_0$, and $\kappa_n=\kappa\in(-1,1)$. Then \eqref{eq:alg-parallel-safeguarding} reduces to
\begin{equation*}
\begin{aligned}
		     y_n&=x_n+u_n\\
    p_n &= \prt{M+\gamma A}^{-1}\prt{M y_n - \gamma C y_n}\\
		    x_{n+1} &= x_n + p_n - y_n\\
      u_{n+1}&=\kappa\lp\tfrac{2-\gamma\hbeta}{2}\lp p_n-x_n\rp + \tfrac{\gamma\hbeta}{2} u_n\rp,
\end{aligned} 
\end{equation*}
which, since $x_{n+1}=p_n-u_n$, can be written as
\begin{equation}
\begin{aligned}
    p_n &= \lp\prt{M+\gamma A}^{-1}\circ\prt{M - \gamma C }\rp\lp p_{n-1}+u_n-u_{n-1}\rp\\
      u_{n+1}&=\kappa\lp\tfrac{2-\gamma\hbeta}{2}(p_n-p_{n-1}+u_{n-1}) + \tfrac{\gamma\hbeta}{2} u_n\rp.
\end{aligned}
\label{eq:alg-parallel-safeguarding-specific-lambda0-fixed-kappa}
\end{equation}
This algorithm is previewed in \cref{sec:special_case_preview}.
Since $\kappa_n=\kappa\in(-1,1)$ for all $n\in\nat$ and since $\seqq{\lambda_n}$ is constant, \labelcref{itm:prop:constant_lambda-sc} ensures that this algorithm produces a $p_n$-sequence that converges weakly to a solution.


\subsection{Alternative formulation}

We can eliminate the $x_n$ sequence in \eqref{eq:alg-parallel-safeguarding} and express the algorithm solely in terms of $y_n$, $p_n$, and $u_n$. The algorithm becomes
\begin{equation}
    \begin{aligned}
    p_n &= \prt{M+\gamma A}^{-1}\prt{M y_n - \gamma C y_n}\\
     y_{n+1} 
     &=y_{n} + \tfrac{\lambda_0\lambda_n}{\lambda_{n+1}}(p_n - y_n) +u_{n+1}-\tfrac{\lambda_n}{\lambda_{n+1}}u_n \\
     &\quad+\tfrac{\lambda_n-\lambda_0}{\lambda_{n+1}}\left((y_n- y_{n-1})+\lambda_0 (y_{n-1} - p_{n-1})\right)\\
           u_{n+1}
          &=\kappa_n\lp\tfrac{(4-\gamma\hbeta-2\lambda_0)}{2}\lp p_n-y_n-\tfrac{\lambda_n-\lambda_0}{\lambda_n}(p_{n-1}-y_{n-1})\rp +  u_n\rp
\end{aligned}  
\label{eq:alg-parallel-safeguarding-no-xn}
\end{equation}
with $y_{-1}=p_{-1}=x_0$ and $u_0=0$.
\begin{proposition}
The algorithms in \eqref{eq:alg-parallel-safeguarding} and \eqref{eq:alg-parallel-safeguarding-no-xn} produce the same $\seqq{y_n}$ and $\seqq{p_n}$ sequences, provided $p_{-1}=y_{-1}=y_0=x_0$.
\label{prop:algs-equal-with-without-xn}
\end{proposition}
\begin{proof}
    We remove the $x_n$ sequence from \eqref{eq:alg-parallel-safeguarding} by inserting
    \begin{align*}
        x_n=\tfrac{\lambda_n}{\lambda_0}\left(y_n-\tfrac{\lambda_n-\lambda_0}{\lambda_n} y_{n-1}-u_n\right),
    \end{align*}
    that comes from the $y_n$ update, into the $x_{n+1}$ update. This gives $x_{n+1}$ update
\begin{align*}
		    \tfrac{\lambda_{n+1}}{\lambda_0}\left(y_{n+1}-\tfrac{\lambda_{n+1}-\lambda_0}{\lambda_{n+1}} y_{n}-u_{n+1}\right)
      &= \tfrac{\lambda_n}{\lambda_0}\left(y_n-\tfrac{\lambda_n-\lambda_0}{\lambda_n} y_{n-1}-u_n\right)\\
      &\quad + \lambda_n(p_n - y_n) + (\lambda_n-\lambda_0)(y_{n-1}-p_{n-1}).
\end{align*}  
Multiplying by $\tfrac{\lambda_0}{\lambda_{n+1}}$ gives
\begin{align*}
		    y_{n+1} 
      &= \tfrac{\lambda_{n+1}-\lambda_0}{\lambda_{n+1}}y_{n}+u_{n+1}+\tfrac{\lambda_n}{\lambda_{n+1}}\left(y_n-\tfrac{\lambda_n-\lambda_0}{\lambda_n} y_{n-1}-u_n\right)\\
      &\quad + \tfrac{\lambda_0}{\lambda_{n+1}}\lp\lambda_n(p_n - y_n) + (\lambda_n-\lambda_0)(y_{n-1}-p_{n-1})\rp\\
      &= y_n+u_{n+1}+\tfrac{\lambda_n}{\lambda_{n+1}}\left(-\tfrac{\lambda_0}{\lambda_{n}}y_{n}+y_n-\tfrac{\lambda_n-\lambda_0}{\lambda_n} y_{n-1}-u_n\right)\\
      &\quad + \tfrac{\lambda_0}{\lambda_{n+1}}\lp\lambda_n(p_n - y_n) + (\lambda_n-\lambda_0)(y_{n-1}-p_{n-1})\rp\\
      &= y_n+\tfrac{\lambda_n-\lambda_0}{\lambda_{n+1}}\left(y_n- y_{n-1}\right)+u_{n+1}-\tfrac{\lambda_n}{\lambda_{n+1}}u_n\\
      &\quad + \tfrac{\lambda_0}{\lambda_{n+1}}\lp\lambda_n(p_n - y_n) + (\lambda_n-\lambda_0)(y_{n-1}-p_{n-1})\rp\\
      &= y_n+ \tfrac{\lambda_0\lambda_n}{\lambda_{n+1}}\lp p_n - y_n\rp+u_{n+1}-\tfrac{\lambda_n}{\lambda_{n+1}}u_n\\
      &\quad +\tfrac{\lambda_n-\lambda_0}{\lambda_{n+1}}\lp\left(y_n- y_{n-1}\right) + \lambda_0(y_{n-1}-p_{n-1})\rp.
\end{align*}  
The $u_{n+1}$ update becomes
\begin{align*}
          u_{n+1}&=\kappa_n\tfrac{(4-\gamma\hbeta-2\lambda_0)}{2}\lp p_n-x_n+\tfrac{\lambda_n-\lambda_0}{\lambda_n}(x_n-p_{n-1}) - \tfrac{2-\gamma\hbeta-2\lambda_0}{4-\gamma\hbeta-2\lambda_0} u_n\rp\\
          &=\kappa_n\tfrac{(4-\gamma\hbeta-2\lambda_0)}{2}\lp p_n-\tfrac{\lambda_0}{\lambda_n}x_n-\tfrac{\lambda_n-\lambda_0}{\lambda_n}p_{n-1} - \tfrac{2-\gamma\hbeta-2\lambda_0}{4-\gamma\hbeta-2\lambda_0} u_n\rp\\
          &=\kappa_n\tfrac{(4-\gamma\hbeta-2\lambda_0)}{2}\lp p_n-\lp y_n-\tfrac{\lambda_n-\lambda_0}{\lambda_n} y_{n-1}-u_n\rp-\tfrac{\lambda_n-\lambda_0}{\lambda_n}p_{n-1} - \tfrac{2-\gamma\hbeta-2\lambda_0}{4-\gamma\hbeta-2\lambda_0} u_n\rp\\
          &=\kappa_n\tfrac{(4-\gamma\hbeta-2\lambda_0)}{2}\lp p_n-y_n-\tfrac{\lambda_n-\lambda_0}{\lambda_n}(p_{n-1}-y_{n-1}) + \tfrac{2}{4-\gamma\hbeta-2\lambda_0} u_n\rp\\
          &=\kappa_n\lp\tfrac{(4-\gamma\hbeta-2\lambda_0)}{2}\lp p_n-y_n-\tfrac{\lambda_n-\lambda_0}{\lambda_n}(p_{n-1}-y_{n-1})\rp +  u_n\rp.
\end{align*}
This concludes the proof.
\end{proof}

\subsection{Fixed-point residual convergence rate}

The convergent quantity in \labelcref{itm:prop:norm_convergence-sc} may be hard to interpret. In this section, we propose a special case of \eqref{eq:alg-parallel-safeguarding} and \eqref{eq:alg-parallel-safeguarding-no-xn} such that this quantity is the fixed-point residual, $p_n-y_n$, for the forward--backward mapping. This is achieved by letting $\kappa_n=\tfrac{\lambda_{n+1}-\lambda_0}{\lambda_{n+1}}$, which implies that 
\begin{align*}
    u_{n+1} = \tfrac{\lambda_{n+1}-\lambda_0}{\lambda_{n+1}}\tfrac{(4-\gamma\hbeta-2\lambda_0)}{2}\lp p_n-y_n\rp
\end{align*}
and that the algorithm becomes
\begin{equation}
    \begin{aligned}
    p_n &= \prt{M+\gamma A}^{-1}\prt{M y_n - \gamma C y_n}\\
     y_{n+1} 
     &=y_{n} + \lp\tfrac{\lambda_0\lambda_n}{\lambda_{n+1}}+\tfrac{\lambda_{n+1}-\lambda_0}{\lambda_{n+1}}\tfrac{(4-\gamma\hbeta-2\lambda_0)}{2}\rp(p_n - y_n)  \\
     &\quad+\tfrac{\lambda_n-\lambda_0}{\lambda_{n+1}}\lp(y_n- y_{n-1})+\tfrac{(4-\gamma\hbeta)}{2}(y_{n-1} - p_{n-1})\rp.
\end{aligned}  
\label{eq:alg-parallel-safeguarding-kappa-equals-alpha}
\end{equation}
This algorithm converges as per the following result.
\begin{proposition}\label{prop:convergence-tunable-FB}
    Suppose that \cref{assum:monotone_inclusion,assum:parameters-sc} hold.
    Then the following hold for \eqref{eq:alg-parallel-safeguarding-kappa-equals-alpha}:
    \begin{enumerate}[noitemsep, ref=\Cref{prop:convergence-sc}~\theenumi]
    \item if $f(n)\to\infty$ as $n\to\infty$, then
    \begin{align*}
        \norm{\tfrac{1}{\gamma}(p_n-y_n)}_M^2\leq\frac{2\|y_0-\xs\|_M^2}{\gamma^2(4-\gamma\hbeta-2\lambda_0)\lambda_0f(n)^2};
    \end{align*}
    \label{itm:thm:norm_convergence-tunable-FB}
    \item if $\hbeta>\beta$ and $f^{\prime}(0)<1$, then \seqq{\|p_n-y_n\|^2} is summable; \label{itm:prop:yn-pn_summable-tunable-FB}
    \item if \seqq{\lambda_n} is constant, the algorithm reduces to relaxed forward--backward splitting and $p_n\weakto\xs\in\zer{A+C}$. \label{itm:prop:constant_lambda-tunable-FB}
\end{enumerate}
\end{proposition}
\begin{proof}
The claims follow from \cref{prop:convergence-sc,prop:algs-equal-with-without-xn} by showing that, for all $n\in\mathbb{N}$: the choice $\kappa_n=\tfrac{\lambda_{n+1}-\lambda_0}{\lambda_{n+1}}$ in \eqref{eq:alg-parallel-safeguarding-no-xn} gives \eqref{eq:alg-parallel-safeguarding-kappa-equals-alpha}; by noting that $\kappa_n=\tfrac{\lambda_{n+1}-\lambda_0}{\lambda_{n+1}}$ satisfies \labelcref{itm:assump-param-zeta-sc} with $\epsilon_0>0$ if $f$ (and consequently \seqq{\lambda_n}) is bounded and with $\epsilon_0=0$ if $f$ (and consequently \seqq{\lambda_n}) is unbounded; and by showing that the expression inside the norm in \labelcref{itm:prop:norm_convergence-sc} satisfies
\begin{align}
p_n-x_n+\tfrac{\lambda_n-\lambda_0}{\lambda_n}(x_n-p_{n-1}) - \tfrac{2-\gamma\beta-2\lambda_0}{4-\gamma\beta-2\lambda_0}u_n = p_n-y_n.
\label{eq:residual_equality}
\end{align}


We will first show that the $u_{n+1}$ update in \eqref{eq:alg-parallel-safeguarding-no-xn} with $\kappa_n=\tfrac{\lambda_{n+1}-\lambda_0}{\lambda_{n+1}}$, i.e.,
\begin{align}
    u_{n+1}
          &=\tfrac{\lambda_{n+1}-\lambda_0}{\lambda_{n+1}}\lp\tfrac{(4-\gamma\hbeta-2\lambda_0)}{2}\lp p_n-y_n-\tfrac{\lambda_n-\lambda_0}{\lambda_n}(p_{n-1}-y_{n-1})\rp +  u_n\rp
          \label{eq:un+1-update}
\end{align}
implies that $u_n=\tfrac{\lambda_n-\lambda_0}{\lambda_n}\tfrac{(4-\gamma\hbeta-2\lambda_0)}{2}(p_{n-1}-y_{n-1})$ for all $n\in\mathbb{N}$. For $n=0$, since $p_{-1}=y_{-1}=u_0=0$, we get
\begin{align*}
    u_{1} = \tfrac{\lambda_1-\lambda_0}{\lambda_1}\tfrac{(4-\gamma\hbeta-2\lambda_0)}{2}(p_{0}-y_{0}).
\end{align*}
For $n\geq 1$, we use induction. Assume that $u_n=\tfrac{\lambda_n-\lambda_0}{\lambda_n}\tfrac{(4-\gamma\hbeta-2\lambda_0)}{2}(p_{n-1}-y_{n-1})$, then
\begin{align}
    u_{n+1}&=\tfrac{\lambda_{n+1}-\lambda_0}{\lambda_{n+1}}\tfrac{(4-\gamma\hbeta-2\lambda_0)}{2}\lp p_n-y_n\rp
     \label{eq:un+1-formula}
\end{align}
since the last two terms in \eqref{eq:un+1-update} cancel, which is what we wanted to show.

The $y_{n+1}$ update in \eqref{eq:alg-parallel-safeguarding-no-xn} with $u_{n+1}$ defined in \eqref{eq:un+1-formula} inserted satisfies
\begin{align*}
         y_{n+1}&=y_{n} + \tfrac{\lambda_0\lambda_n}{\lambda_{n+1}}(p_n - y_n) +u_{n+1}-\tfrac{\lambda_n}{\lambda_{n+1}}u_n \\
     &\quad+\tfrac{\lambda_n-\lambda_0}{\lambda_{n+1}}\left((y_n- y_{n-1})+\lambda_0 (y_{n-1} - p_{n-1})\right)\\
     &=y_{n} + \tfrac{\lambda_0\lambda_n}{\lambda_{n+1}}(p_n - y_n) +\tfrac{\lambda_{n+1}-\lambda_0}{\lambda_{n+1}}\tfrac{(4-\gamma\hbeta-2\lambda_0)}{2}\lp p_n-y_n\rp\\
     &\quad-\tfrac{\lambda_{n}-\lambda_0}{\lambda_{n+1}}\tfrac{(4-\gamma\hbeta-2\lambda_0)}{2}\lp p_{n-1}-y_{n-1}\rp \\
     &\quad+\tfrac{\lambda_n-\lambda_0}{\lambda_{n+1}}\left((y_n- y_{n-1})+\lambda_0 (y_{n-1} - p_{n-1})\right)\\
     &=y_{n} + \lp\tfrac{\lambda_0\lambda_n}{\lambda_{n+1}}+\tfrac{\lambda_{n+1}-\lambda_0}{\lambda_{n+1}}\tfrac{(4-\gamma\hbeta-2\lambda_0)}{2}\rp(p_n - y_n) \\
     &\quad+\tfrac{\lambda_n-\lambda_0}{\lambda_{n+1}}\left(y_n- y_{n-1}\right)
     +\tfrac{\lambda_n-\lambda_0}{\lambda_{n+1}}\left(\lambda_0 +\tfrac{(4-\gamma\hbeta-2\lambda_0)}{2}\right)(y_{n-1} - p_{n-1})\\
     &=y_{n} + \lp\tfrac{\lambda_0\lambda_n}{\lambda_{n+1}}+\tfrac{\lambda_{n+1}-\lambda_0}{\lambda_{n+1}}\tfrac{(4-\gamma\hbeta-2\lambda_0)}{2}\rp(p_n - y_n) \\
     &\quad+\tfrac{\lambda_n-\lambda_0}{\lambda_{n+1}}\left(\left(y_n- y_{n-1}\right)
     +\tfrac{(4-\gamma\hbeta)}{2}(y_{n-1} - p_{n-1})\right),
           \end{align*}
           which equals the $y_{n+1}$ update in \eqref{eq:alg-parallel-safeguarding-kappa-equals-alpha}.

           Finally, using the $y_n$ update equation in \eqref{eq:alg-parallel-safeguarding}, i.e, $\tfrac{\lambda_0}{\lambda_n}x_n=y_n-\tfrac{\lambda_n-\lambda_0}{\lambda_n}y_{n-1}-u_n$ and the $u_{n+1}$ definition in \eqref{eq:un+1-formula}, we conclude that
\begin{equation*}
\begin{aligned}
&p_n-x_n+\tfrac{\lambda_n-\lambda_0}{\lambda_n}(x_n-p_{n-1}) - \tfrac{2-\gamma\beta-2\lambda_0}{4-\gamma\beta-2\lambda_0}u_n\\
&\qquad\qquad=p_n-y_n+\tfrac{\lambda_n-\lambda_0}{\lambda_n}(y_{n-1}-p_{n-1}) + \left(1-\tfrac{2-\gamma\beta-2\lambda_0}{4-\gamma\beta-2\lambda_0}\right) u_n\\
     &\qquad\qquad=p_n-y_n+\tfrac{\lambda_n-\lambda_0}{\lambda_n}(y_{n-1}-p_{n-1}) + \tfrac{2}{4-\gamma\beta-2\lambda_0} u_n, \\
     &\qquad\qquad=p_n-y_n.
\end{aligned}
\end{equation*}
           This completes the proof.
\end{proof}


One of the special cases previewed in \cref{sec:special_case_preview} is obtained from \eqref{eq:alg-parallel-safeguarding-kappa-equals-alpha} by letting $\lambda_n=\left(1-\tfrac{\gamma\hbeta}{4}\right)^e(1+n)^{e}$ for all $n\in\nat$. The resulting algorithm is numerically evaluated in \cref{sec:numerical-experiments} and enjoys the following convergence properties.
\begin{corollary}\label{prop:convergence-tunable-FB-exp-e}
    Suppose that \cref{assum:monotone_inclusion,assum:parameters-sc} hold 
    with $\lambda_0=\left(1-\tfrac{\gamma\hbeta}{4}\right)^e$ 
    and  
     let $f(n)=(1+n)^{e}$ with $e\in[0,1]$. Then the following hold for \eqref{eq:alg-parallel-safeguarding-kappa-equals-alpha}:
    \renewcommand\theenumi\labelenumi
    \renewcommand{\labelenumi}{{(\roman{enumi})}}
    \begin{enumerate}[noitemsep, ref=\Cref{prop:convergence-tunable-FB-exp-e}~\theenumi]
    \item if $e>0$, then
       \begin{align*}
        \norm{\tfrac{1}{\gamma}(p_n-y_n)}_M^2\leq\frac{2\|y_0-\xs\|_M^2}{\gamma^2(4-\gamma\hbeta-2\lambda_0)\lambda_0(1+n)^{2e}};
    \end{align*}
    \label{itm:cor:exp-e:fp-resid_convergence}
    \item if $\hbeta>\beta$ and $e<1$, then \seqq{\|p_n-y_n\|^2} is summable;
    \label{itm:cor:exp-e:fp-resid_summable}
    \item if $e=0$, the algorithm reduces to forward--backward splitting and $p_n\weakto\xs\in\zer{A+C}$.
    \label{itm:cor:exp-e:sequence_convergence}
\end{enumerate}
\end{corollary}
\labelcref{itm:cor:exp-e:fp-resid_convergence} states that \seqq{\|p_n-y_n\|^2} converges as $\mathcal{O}\lp\tfrac{1}{n^{2e}}\rp$. When $\hbeta>\beta$ and $e<1$. \labelcref{itm:cor:exp-e:fp-resid_summable} gives that \seqq{\|p_n-y_n\|^2} converges as $\mathcal{O}\lp\tfrac{1}{n}\rp$ due to its summability. This gives a combined $\mathcal{O}\lp\min\lp\tfrac{1}{n},\tfrac{1}{n^{2e}}\rp\rp$ convergence rate and implies tunability of the convergence rate by selecting $e\in[0,1]$. Letting $e=1$ implies that our algorithm, as we will see in \cref{sec:accelerated_ppm_Halpern}, reduces to the accelerated proximal point method and the Halpern iteration that converge as $O\lp\tfrac{1}{n^2}\rp$.

\subsubsection{Accelerated proximal point method and Halpern iteration}
\label{sec:accelerated_ppm_Halpern}

Letting $f(n)=1+n$ and $\lambda_0=\left(1-\tfrac{\gamma\hbeta}{4}\right)$ to get $\lambda_n=\left(1-\tfrac{\gamma\hbeta}{4}\right)(1+n)$, we get that the $y_{n+1}$ update of \eqref{eq:alg-parallel-safeguarding-kappa-equals-alpha} satisfies
\begin{align*}
         y_{n+1}
     &=y_{n} + \lp\tfrac{\lambda_0\lambda_n}{\lambda_{n+1}}+\tfrac{\lambda_{n+1}-\lambda_0}{\lambda_{n+1}}\tfrac{(4-\gamma\hbeta-2\lambda_0)}{2}\rp(p_n - y_n) \\
     &\quad+\tfrac{\lambda_n-\lambda_0}{\lambda_{n+1}}\left(\left(y_n- y_{n-1}\right)
     +\tfrac{(4-\gamma\hbeta)}{2}(y_{n-1} - p_{n-1})\right)\\
     &=y_{n} + \left(1-\tfrac{\gamma\hbeta}{4}\right)\lp\tfrac{n+1}{n+2}+\tfrac{n+1}{n+2}\rp(p_n - y_n) \\
     &\quad+\tfrac{n}{n+2}\left(\left(y_n- y_{n-1}\right)
     +\tfrac{(4-\gamma\hbeta)}{2}(y_{n-1} - p_{n-1})\right)\\
      &=y_{n} + \tfrac{4-\gamma\hbeta}{2}\tfrac{n+1}{n+2}(p_n - y_n) \\
     &\quad+\tfrac{n}{n+2}\left(\left(y_n- y_{n-1}\right)
     +\tfrac{(4-\gamma\hbeta)}{2}(y_{n-1} - p_{n-1})\right)\\
          &= \tfrac{\gamma\hbeta(1+n)}{4+2n}y_n + \tfrac{n(2-\gamma\hbeta)}{4+2n}y_{n-1} + \tfrac{(1+n)(4-\gamma\hbeta)}{4+2n}p_n - \tfrac{n(4-\gamma\hbeta)}{4+2n}p_{n-1}
           \end{align*}
and algorithm \eqref{eq:alg-parallel-safeguarding-kappa-equals-alpha} becomes
\begin{equation}
\begin{aligned}
    p_n &= \prt{M+\gamma A}^{-1}\prt{M - \gamma C}y_n,\\
    y_{n+1} &= \tfrac{\gamma\hbeta(1+n)}{4+2n}y_n + \tfrac{n(2-\gamma\hbeta)}{4+2n}y_{n-1} + \tfrac{(1+n)(4-\gamma\hbeta)}{4+2n}p_n - \tfrac{n(4-\gamma\hbeta)}{4+2n}p_{n-1}.
\end{aligned}
\label{eq:accelerated_fb}
\end{equation}
From \cref{prop:convergence-tunable-FB-exp-e}, we conclude since $e=1$ and by letting $\beta=\hbeta$ that this algorithm converges as
\begin{align*}
        \norm{\tfrac{1}{\gamma}(p_n-y_n)}_M^2\leq\frac{16\|y_0-\xs\|_M^2}{\gamma^2\left(4-\gamma\beta\right)^2(1+n)^{2}}.
    \end{align*}
By letting $C=0$ and consequently $\beta=0$, we arrive at the accelerated proximal point method in \cite{kim2021accelerated} and the $O(\tfrac{1}{n^2})$ convergence rate results found in \cite{kim2021accelerated} is recovered by \cref{prop:convergence-tunable-FB-exp-e}.

If we let $A=0$, $\hbeta=\beta$, and $\gamma\beta=2$, the forward--backward mapping in \eqref{eq:accelerated_fb} satisfies
\begin{align*}
    p_n=(M-\tfrac{2}{\beta} C)y_n
\end{align*}
where $(M-\tfrac{2}{\beta} C):=N$ is nonexpansive in the $\|\cdot\|_M$ norm. This implies that the algorithm aims at solving the nonexpansive fixed-point equation $y=Ny$ by iterating
\begin{align*}
    p_n&=Ny_n\\
     y_{n+1} 
     &= \tfrac{1+n}{2+n}y_n + \tfrac{1+n}{2+n}p_n - \tfrac{n}{2+n}p_{n-1},
\end{align*}
which is the Halpern iteration studied in \cite{lieder2021convergence}. This is seen by recursively inserting $y_n$ into the $y_{n+1}$ update to get
\begin{align*}
    y_{n+1} = \tfrac{1}{n+2}y_{0} + \tfrac{n+1}{n+2}Ny_n,
\end{align*}
which is the formulation used in \cite{lieder2021convergence}.
From \cref{prop:convergence-tunable-FB-exp-e}, we conclude that this iteration converges as
\begin{align*}
        \norm{p_n-y_n}_M^2\leq\frac{4\|y_0-\xs\|_M^2}{(1+n)^{2}},
    \end{align*}
    which recovers the convergence result in \cite{lieder2021convergence}. Interestingly, although the convergence rate is optimized by this choice of $\lambda_n$, it does not perform very well in practice. Other choices of $\seqq{\lambda_n}$ with slower rate guarantees can give significantly better practical performance as demonstrated in \cref{sec:numerical-experiments}.

\section{Numerical examples}
\label{sec:numerical-experiments}

In this section, we apply our proposed algorithms on the problem $0\in Az$, where $z=(x,y)$ and
\begin{align*}
    Az=\begin{bmatrix}0& -1\\1 & 0\end{bmatrix}z=(-y,x)
\end{align*}
for all $z\in\reals^2$. The operator $A:\reals^2\to\reals^2$ is skew-symmetric and the monotone inclusion problem $0\in Az$ can be interpreted as an optimality condition for the minimax problem 
\begin{align*}
    \max_{y\in\reals}\min_{x\in\reals} xy
\end{align*}
with unique solution $x=y=0$. We will in particular evaluate the algorithm in \eqref{eq:alg-parallel-safeguarding-specific-lambda0-fixed-kappa} with different choices of $\kappa\in(-1,1)$ and the algorithm in \eqref{eq:alg-parallel-safeguarding-kappa-equals-alpha} with  
\begin{align}
    \lambda_n=\lp 1-\tfrac{\gamma\hbeta}{4}\rp^e(1+n)^e
    \label{eq:lambda_n-exp-e}
\end{align} 
for all $n\in\nat$ and $e\in[0,1]$. According to \cref{prop:convergence-sc} and \cref{prop:convergence-tunable-FB}, \eqref{eq:alg-parallel-safeguarding-specific-lambda0-fixed-kappa} with $\kappa\in(-1,1)$ and \eqref{eq:alg-parallel-safeguarding-kappa-equals-alpha} with $\lambda_n$ in \eqref{eq:lambda_n-exp-e} and $e=0$ (corresponding to the standard FB method) converge weakly to a solution of the inclusion problem. As per \cref{prop:convergence-tunable-FB-exp-e}, \eqref{eq:alg-parallel-safeguarding-kappa-equals-alpha} with $\lambda_n$ in \eqref{eq:lambda_n-exp-e} and $e\in(0,1]$ converges in squared norm of the fixed point residual as $O\lp\tfrac{1}{n^{2e}}\rp$ and when $e<1$ and $\hbeta>\beta$, it does so with a rate of $O\lp\tfrac{1}{n}\rp$. 

For all our experiments, parameters $\gamma=0.1$ and $\hbeta=0.001$ (which is feasible since $C=0$ and therefore $\beta=0$) are used, and starting points  $y_{-1}=p_{-1}=y_0=(3,3)$ and $x_{-1}=p_{-1}=x_0=(3,3)$ for \eqref{eq:alg-parallel-safeguarding-kappa-equals-alpha} and \eqref{eq:alg-parallel-safeguarding-specific-lambda0-fixed-kappa} respectively. 

In \cref{fig:numerics_FB_tunable_rate,tab:numerics_FB_tunable_rate}
we report numerical results for the algorithm in \eqref{eq:alg-parallel-safeguarding-kappa-equals-alpha} with $\lambda_n$ in \eqref{eq:lambda_n-exp-e} and $e\in\{0,0.1,\ldots,1\}$ and $M=\Id$. The choice $e=0$ gives standard forward--backward splitting and $e=1$ gives the accelerated proximal point method in \cite{kim2021accelerated}. The other choices of $e$ gives rise to new algorithms. The figure shows that the distance to the unique solution behaves over-damped for small values of $e$ and under-damped for large values of $e$. There is a sweet spot in the middle that has the right level of damping and performs significantly better than the previously known special cases with $e=0$ and $e=1$, at least for this problem.
\begin{figure}
\centering
\begin{tikzpicture}
    \node at (0,0) {\includegraphics[width=9cm]{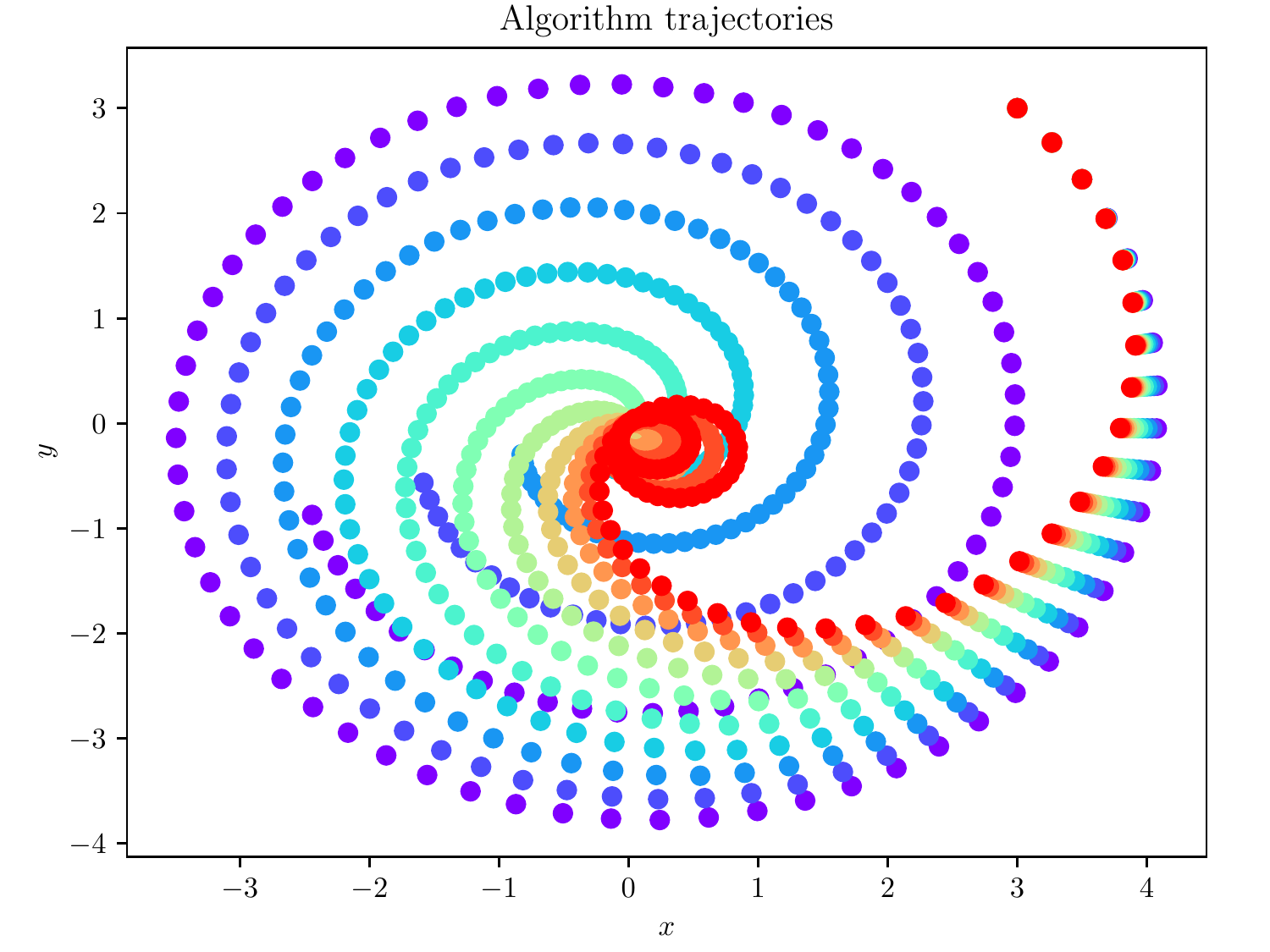}};
    \node at (0,-6.8) {\includegraphics[width=9cm]{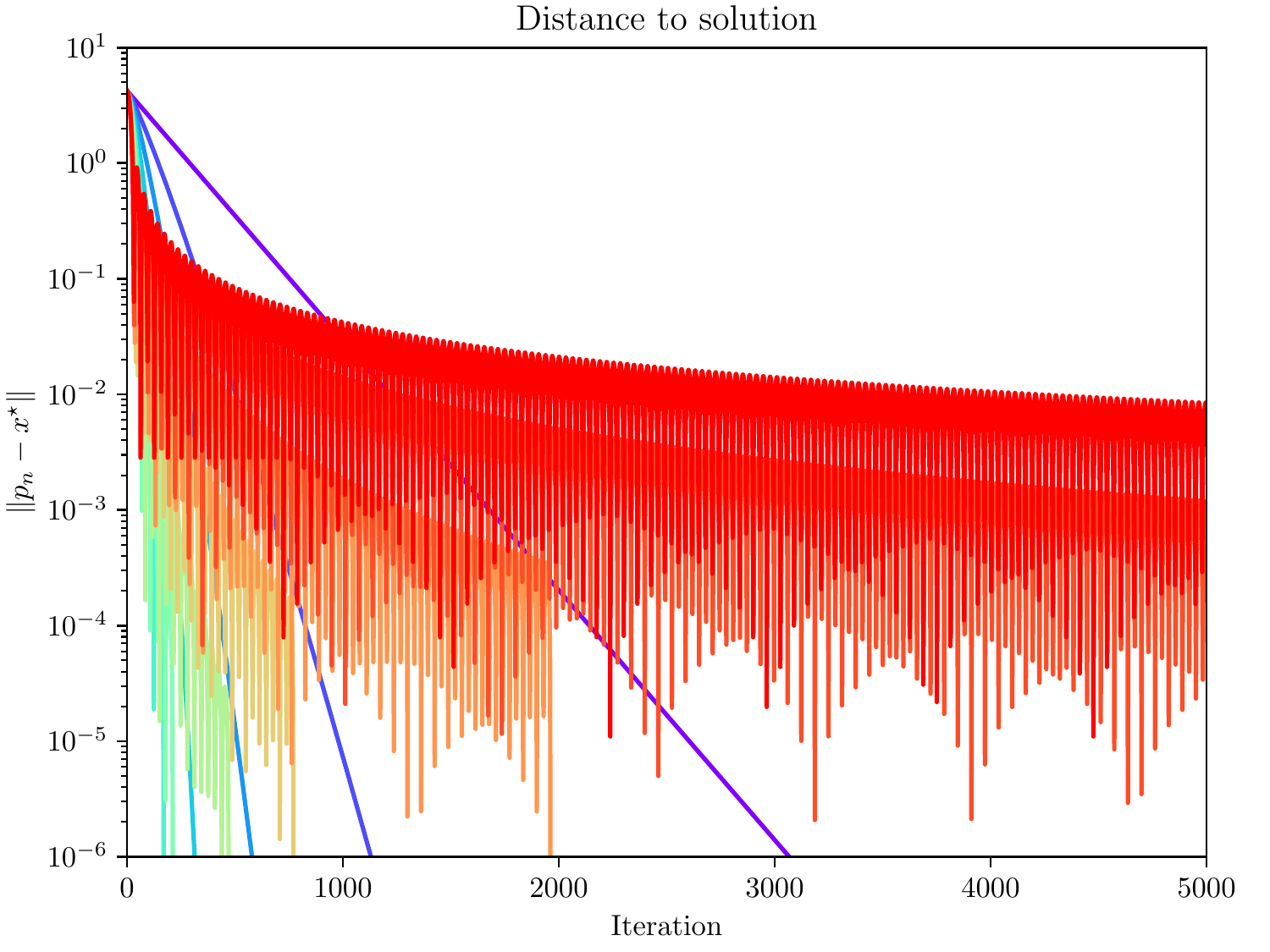}};

    \begin{scope}[yshift=-0.25cm,xshift=4.5cm]
       \xdef\perrow{4}
    \foreach \r in {1,...,11}{
                \xdef\temp{\noexpand\draw[color=matplotlibcolor\r of11,ultra thick] (0,-0.5*\r)--(0.5,-0.5*\r);}
                \temp
                \node[right] at (0.5,-0.5*\r) {$e=\;$\pgfmathparse{(\r-1)/10}\pgfmathprintnumber[fixed,precision=2]{\pgfmathresult}};
    }
 
    \draw (-0.2,-0.2) rectangle (1.8,-5.8);
    \end{scope}

 
\end{tikzpicture}
\caption{We evaluate the algorithm in \eqref{eq:alg-parallel-safeguarding-kappa-equals-alpha} with $\lambda_n$ in \eqref{eq:lambda_n-exp-e} with the different choices of $e$ specified in the legend to the right. The upper figure shows the 100 first $p_n$-iterates for the different $e$ and the lower figure shows the distance to the unique solution. The algorithm with $e=0$ is standard forward--backward splitting (here, only the backward part is used) and the algorithm with $e=1$ is the accelerated proximal point in \cite{kim2021accelerated}. For small $e$, we get a strictly decreasing distance to solution, while large $e$ gives an oscillatory behavior. Our new algorithms with $e$ chosen around the middle of the allowed range strike a good balance and achieves superior performance compared to the previously known methods.}
\label{fig:numerics_FB_tunable_rate}
\end{figure}
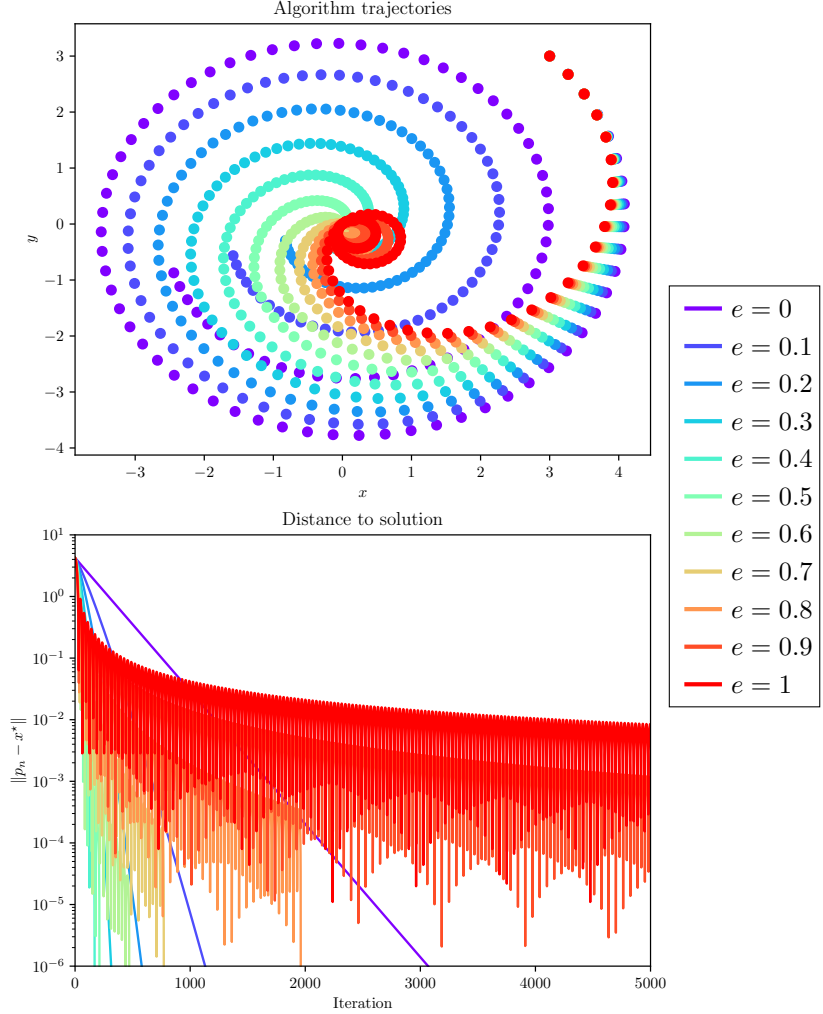

\begin{table}
	\begin{center}
\caption{We report the number of iterations to reach accuracy $\|p_n-\xs\|\leq 10^{-6}$ for the algorithm in \eqref{eq:alg-parallel-safeguarding-kappa-equals-alpha} with $\lambda_n$ in \eqref{eq:lambda_n-exp-e} and different choices of $e$. The theory predicts a $O\lp 1/n^{2e}\rp$ convergence rate for the squared fixed point residual norm. Although $e=1$ (which gives the accelerated proximal point method in \cite{kim2021accelerated}) gives the best theoretical rate, it performs the worst. Standard forward--backward splitting is obtained with $e=0$. We find that values of $e$ in between can perform considerably better than these previously known algorithms.}
 \label{tab:numerics_FB_tunable_rate}
		{\renewcommand{\arraystretch}{1.3}
			\begin{tabular}{cc|cc|cc|cc}
\specialrule{2pt}{1pt}{1pt}
\textbf{$e$} & \textbf{\# iter.} &\textbf{$e$} & \textbf{\# iter.}&\textbf{$e$} & \textbf{\# iter.}&\textbf{$e$} & \textbf{\# iter.} \\ 
\hline
0.0 & 3068 & 
0.1 & 1131 & 
0.2 & 580 &
0.3 & 314 \\ 
0.4 & 170 & 
0.5 & 212 & 
0.6 & 471 & 
0.7 & 771 \\ 
0.8 & 1961 & 
0.9 & 10625 & 
1.0 & 21213167 \\ 
\specialrule{2pt}{1pt}{1pt}
\end{tabular}
		}
\end{center}
\end{table}

In \cref{fig:numerics_fixed-kappa-full-range,tab:numerics_fixed-kappa-full-range}, we report numerical results for the algorithm in \eqref{eq:alg-parallel-safeguarding-specific-lambda0-fixed-kappa} with $\kappa\in\{-0.9,-0.8,\ldots,0.9\}$ and $M=\Id$. The theory predicts sequence convergence towards a solution of the problem for all $\kappa\in(-1,1)$. 
The choice $\kappa=0$ gives rise to standard forward--backward splitting and all other values of $\kappa$ define new algorithms. The figure reveals that the performance is best for $\kappa\in[0.8,0.9]$, significantly better than standard FB splitting with $\kappa=0$.

\begin{figure}
\centering
\begin{tikzpicture}
    \node at (0,0) {\includegraphics[width=9cm]{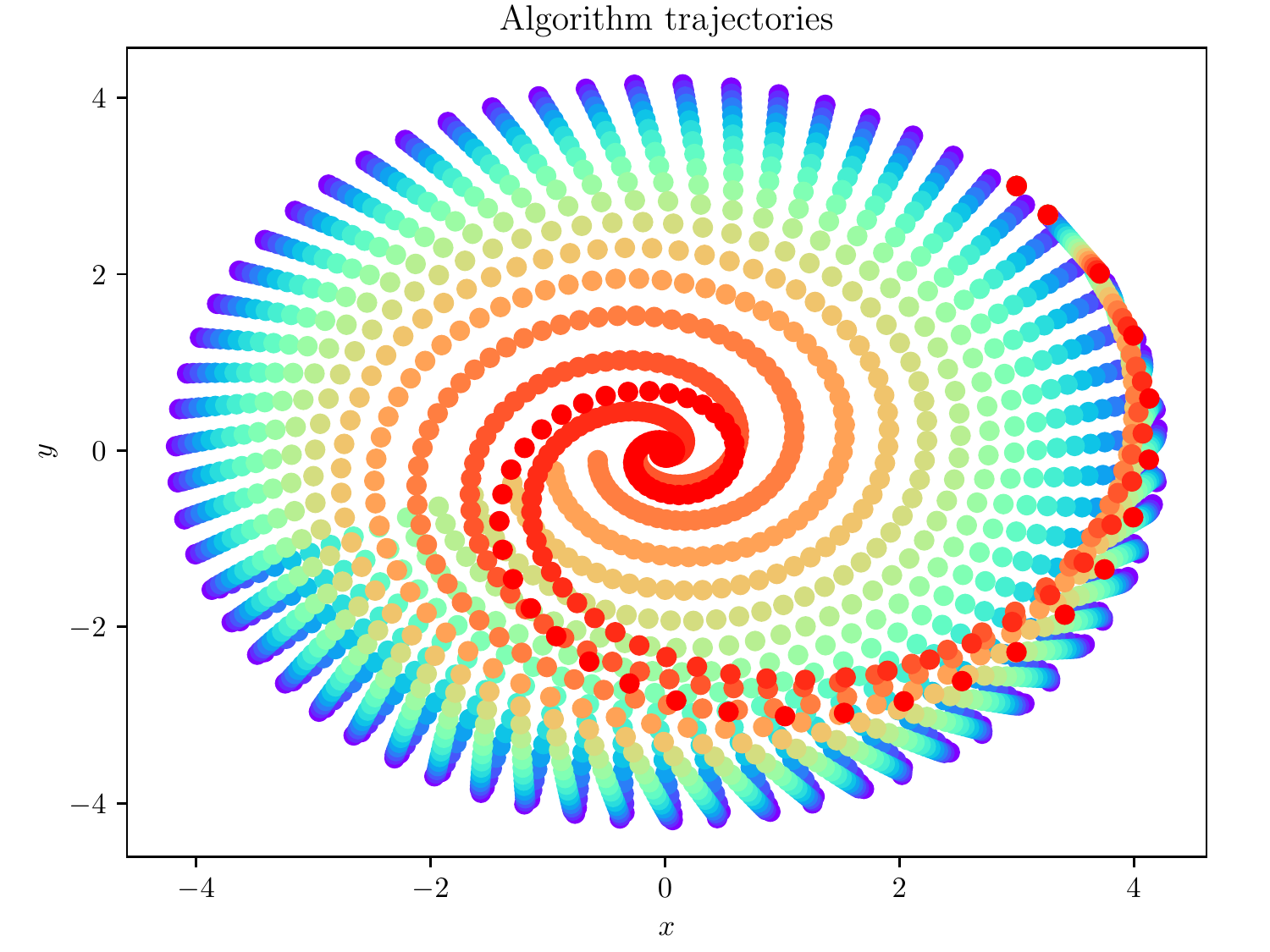}};
    \node at (0,-6.8) {\includegraphics[width=9cm]{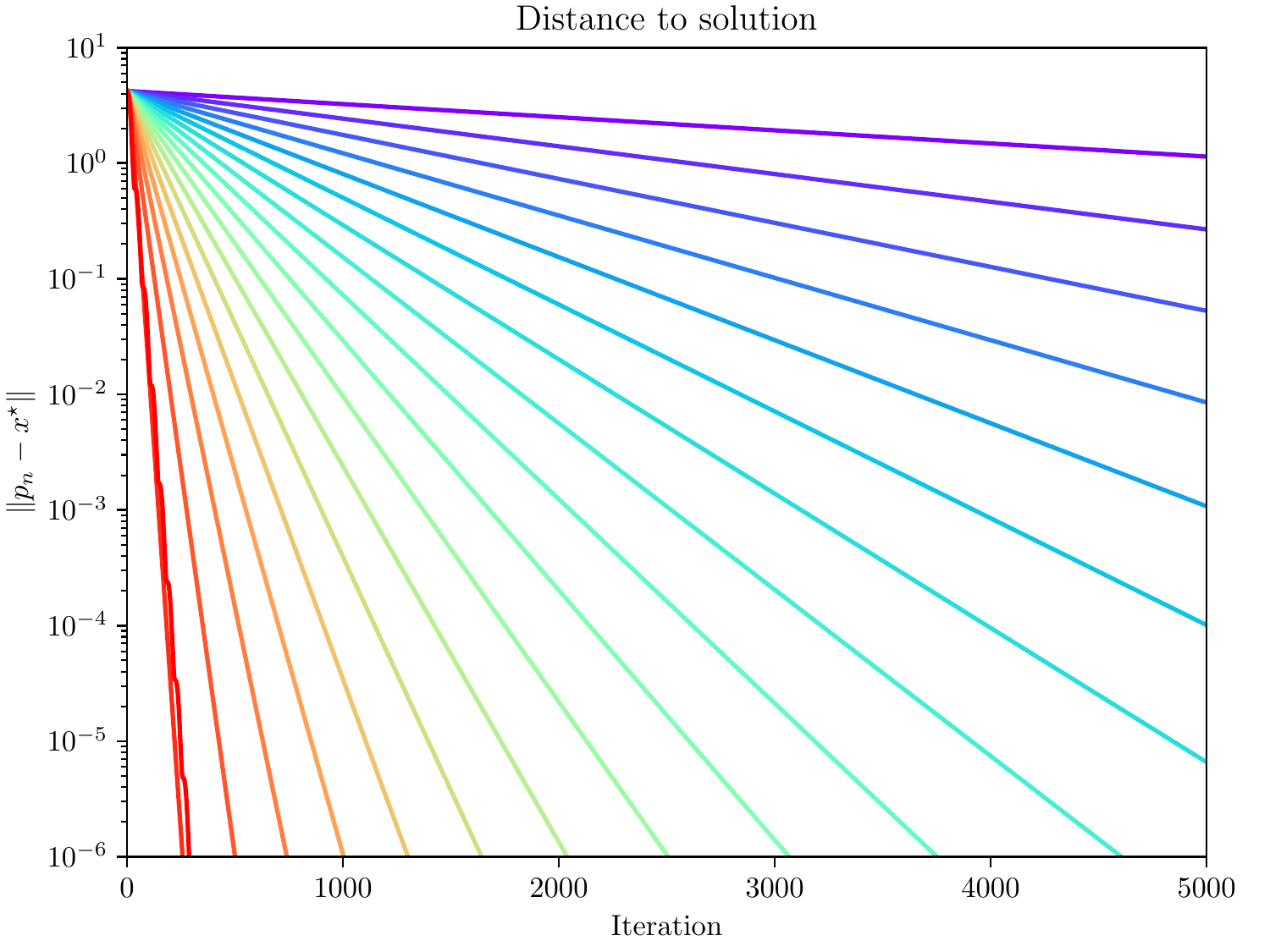}};

\begin{scope}[yshift=1.75cm,xshift=4.5cm]
    \xdef\perrow{5}
    \foreach \r in {1,...,19}{
                \xdef\temp{\noexpand\draw[color=matplotlibcolor\r of19,ultra thick] (0,-0.5*\r)--(0.5,-0.5*\r);}
                \temp
                \node[right] at (0.5,-0.5*\r) {$\kappa=\;$\pgfmathparse{(\r-10)/10}\pgfmathprintnumber[fixed,precision=2]{\pgfmathresult}};
}
    \draw (-0.2,-0.2) rectangle (2.2,-9.8);
    \end{scope}

\end{tikzpicture}
\caption{We evaluate the algorithm in \eqref{eq:alg-parallel-safeguarding-specific-lambda0-fixed-kappa} with $\kappa\in\{-0.9,-0.8,\ldots,0.9\}$ as specified in the legend to the right. The upper figure shows the 100 first $p_n$-iterates for the different $\kappa$ and the lower figure shows the distance to the unique solution. The performance is best for $\kappa\in\{0.8,0.9\}$ and many choices of $\kappa\in(0,1)$ outperform the standard forward--backward splitting method that is obtained by letting $\kappa=0$.}
\label{fig:numerics_fixed-kappa-full-range}
\end{figure}

\begin{table}
	\begin{center}
\caption{We report the number of iterations to reach accuracy $\|p_n-\xs\|\leq 10^{-6}$ for the algorithm in \eqref{eq:alg-parallel-safeguarding-specific-lambda0-fixed-kappa} with $\kappa\in\{-0.9,-0.8,\ldots,0.9\}$. The theory predicts sequence for all these values of $\kappa$. The choices $\kappa=0.8$ and $\kappa=0.9$ significantly outperform standard forward--backward splitting that is obtained by letting $\kappa=0$.}
\label{tab:numerics_fixed-kappa-full-range}
		{\renewcommand{\arraystretch}{1.3}
			\begin{tabular}{cc|cc|cc|cc}
\specialrule{2pt}{1pt}{1pt}
\textbf{$\kappa$} & \textbf{\# iter.} &\textbf{$\kappa$} & \textbf{\# iter.}&\textbf{$\kappa$} & \textbf{\# iter.}&\textbf{$\kappa$} & \textbf{\# iter.} \\ 
\hline
-0.9 & 58350 &
-0.8 & 27653 &
-0.7 & 17414 &
-0.6 & 12292 \\
-0.5 & 9219 &
-0.4 & 7170 &
-0.3 & 5706 &
-0.2 & 4607 \\
-0.1 & 3752 &
0.0 & 3068 &
0.1 & 2507 &
0.2 & 2040 \\
0.3 & 1643 &
0.4 & 1302 &
0.5 & 1005 &
0.6 & 741 \\
0.7 & 501 &
0.8 & 258 &
0.9 & 288 \\
\specialrule{2pt}{1pt}{1pt}
\end{tabular}}
\end{center}
\end{table}

In \cref{fig:numerics_fixed-kappa-best_range,tab:numerics_fixed-kappa-best_range}, we provide numerical results over a finer grid of the best performing $\kappa$. We set the range to $\kappa\in[0.8,0.9]$ and use a spacing of 0.02. We see that for $\kappa=0.8$ and $\kappa=0.82$, the distance to solution is non-oscillatory, while it oscillates for greater values of $\kappa$. All these choices of $\kappa$ perform very well.

\begin{figure}
\centering
\begin{tikzpicture}
    \node at (0,0) {\includegraphics[width=9cm]{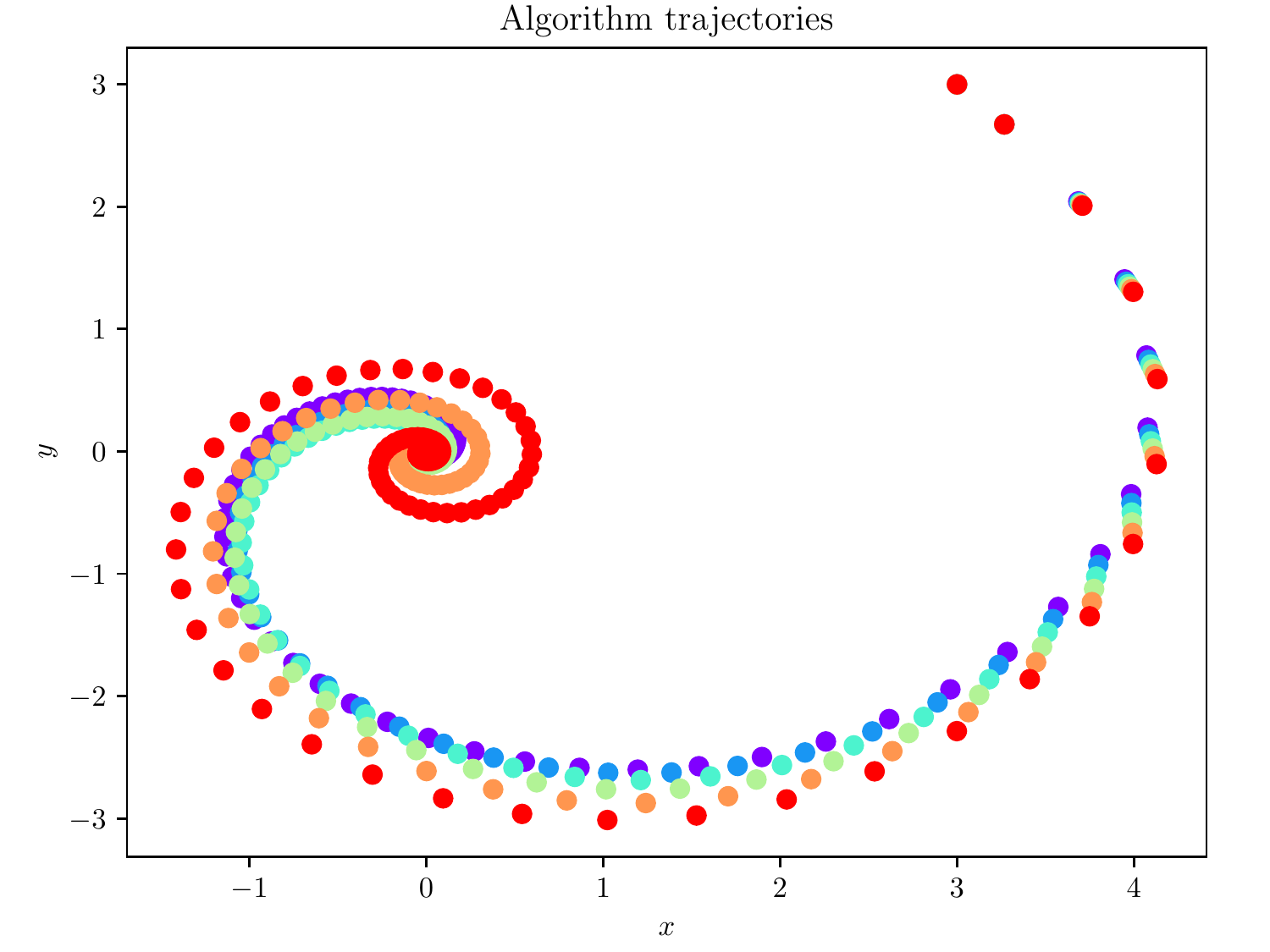}};
    \node at (0,-6.8) {\includegraphics[width=9cm]{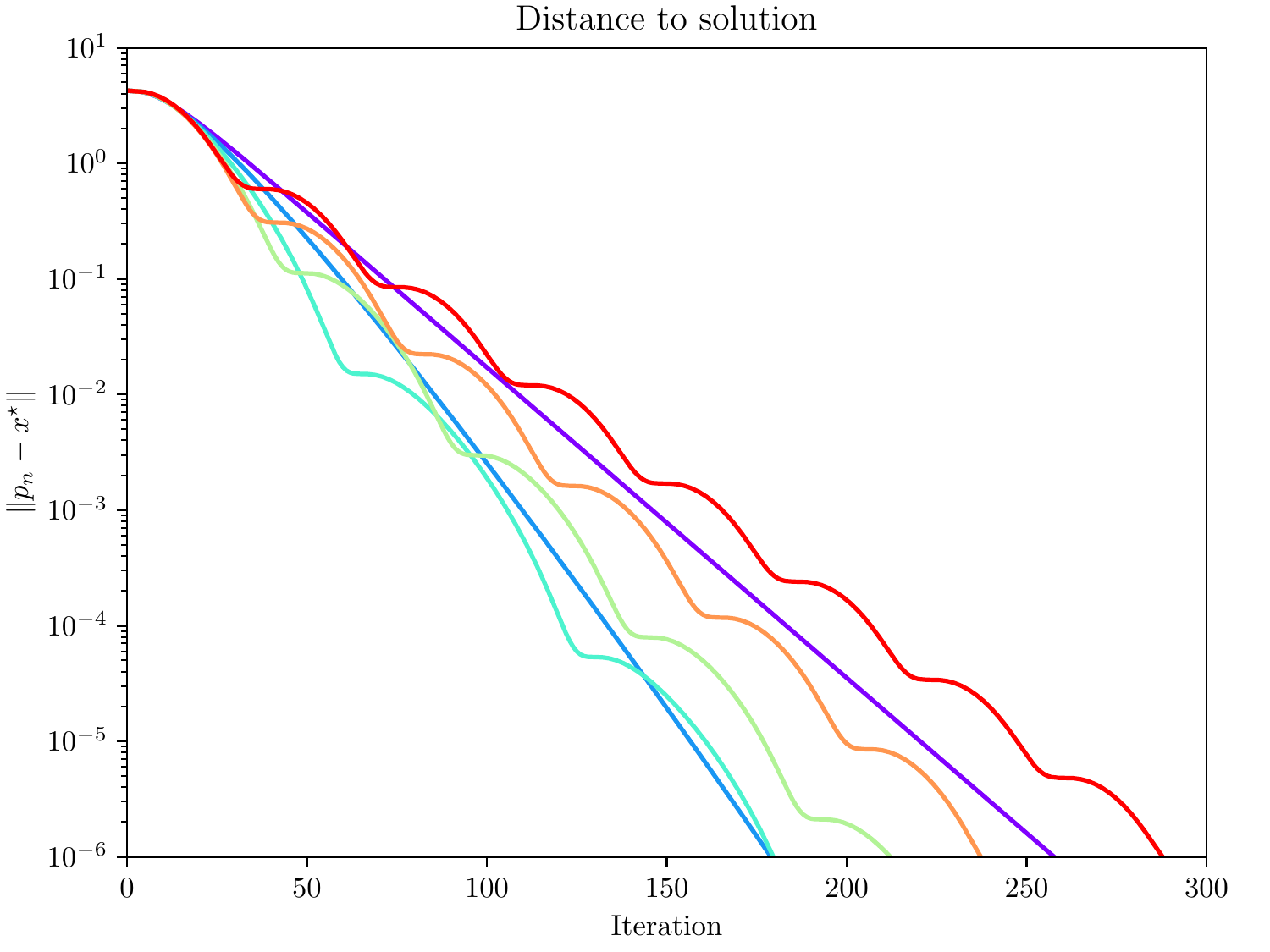}};

   \begin{scope}[yshift=-1.5cm,xshift=4.5cm]
    \xdef\perrow{3}
    \foreach \r in {1,...,6}{
                \xdef\temp{\noexpand\draw[color=matplotlibcolor\r of6,ultra thick] (0,-0.5*\r)--(0.5,-0.5*\r);}
                \temp
                \node[right] at (0.5,-0.5*\r) {$\kappa=\;$\pgfmathparse{(\r-1+40))/50}\pgfmathprintnumber[fixed,precision=2]{\pgfmathresult}};
}

    \draw (-0.2,-0.2) rectangle (2.1,-3.3);
    \end{scope}
    
\end{tikzpicture}
\caption{We evaluate the algorithm in \eqref{eq:alg-parallel-safeguarding-specific-lambda0-fixed-kappa} with $\kappa\in\{0.8,0.82,\ldots,0.9\}$ as specified in the legend to the right. The upper figure shows the 100 first $p_n$-iterates for the different $\kappa$ and the lower figure shows the distance to the unique solution. All these choices perform well and we go from a non-oscillatory behavior to an oscillatory behaviour within this range of $\kappa$.}
\label{fig:numerics_fixed-kappa-best_range}
\end{figure}

\begin{table}
	\begin{center}
\caption{We report the number of iterations to reach accuracy $\|p_n-\xs\|\leq 10^{-6}$ for the algorithm in \eqref{eq:alg-parallel-safeguarding-specific-lambda0-fixed-kappa} with $\kappa\in\{0.8,0.82,\ldots,0.9\}$. The theory predicts sequence convergence with all these $\kappa$. All choices of $\kappa$ perform well.}
 \label{tab:numerics_fixed-kappa-best_range}
		{\renewcommand{\arraystretch}{1.3}
			\begin{tabular}{cc|cc|cc}
\specialrule{2pt}{1pt}{1pt}
\textbf{$\kappa$} & \textbf{\# iter.} &\textbf{$\kappa$} & \textbf{\# iter.}&\textbf{$\kappa$} & \textbf{\# iter.} \\ 
\hline
0.80 & 258 &
0.82 & 179 &
0.84 & 180 \\
0.86 & 213 &
0.88 & 238 &
0.90 & 288 \\
\specialrule{2pt}{1pt}{1pt}
\end{tabular}}
\end{center}
\end{table}

\newpage

\section{Deferred results and proofs}\label{sec:omitted-proofs}
In what follows, we present some results that have been used in the previous sections along with the proof of \cref{thm:LP-identity} that was deferred to this section. Prior to that, we define the auxiliary parameter
\begin{align}\label{itm:def-param-theta-prime}
    \ptheta_n\defeq\prt{2-\gamma_n\hbeta}\mu_n+2\balpha_n\btheta_n
\end{align}
which frequently appears throughout this section.

We begin by establishing some identities between the  parameters defined in \cref{alg:main}. These identities are used several times in the proof of \cref{thm:LP-identity}.

\begin{proposition}\label{prop:parameters}
Consider the auxiliary parameters defined in step~\ref{alg:main:auxiliary-parameters} of \cref{alg:main}. Then, for all $n\in\nat$, the following identities hold
\renewcommand\theenumi\labelenumi
\renewcommand{\labelenumi}{{(\roman{enumi})}}
\begin{enumerate}[noitemsep, ref=\Cref{prop:parameters}~\theenumi]
    \item $\theta_n=\prt{2-\gamma_n\hbeta}\btheta_n+\htheta_n$;\label{itm:prop-theta-i}
    \item $\theta_n=2\btheta_n+\prt{2-\gamma_n\hbeta}(\lambda_n+\mu_n)$;\label{itm:prop-theta-ii}
    \item $\lambda_n^2\theta_n=\htheta_{n}(\lambda_{n}+\mu_{n})-2\btheta_n^2$.\label{itm:prop-theta-iii}
\end{enumerate}
\end{proposition}

\begin{proof}
For \labelcref{itm:prop-theta-i}, from definition of $\btheta_n$ and $\htheta_n$, we have 
\begin{align*}
    \prt{2-\gamma_n\hbeta}\btheta_n+\htheta_n
    &= \prt{2-\gamma_n\hbeta}\prt{\lambda_n+\mu_n-\lambda_n^2}+\prt{2\lambda_n+2\mu_n-\gamma_n\hbeta\lambda_n^2}\\
    &= \prt{2-\gamma_n\hbeta}\prt{\lambda_n+\mu_n}-2\lambda_n^2+\gamma_n\hbeta\lambda_n^2+\prt{2\lambda_n+2\mu_n-\gamma_n\hbeta\lambda_n^2}\\
    &= \prt{2-\gamma_n\hbeta}\prt{\lambda_n+\mu_n}-2\lambda_n^2+2\prt{\lambda_n+\mu_n}\\
    &= \prt{4-\gamma_n\hbeta}\prt{\lambda_n+\mu_n}-2\lambda_n^2\\
    &=\theta_n,
\end{align*}
which holds by definition of $\theta_n$ in \cref{alg:main}. For \labelcref{itm:prop-theta-ii} we have
\begin{align*}
    2\btheta_n+\prt{2-\gamma_n\hbeta}(\lambda_n+\mu_n)
    &= 2\prt{\lambda_n+\mu_n-\lambda_n^2}+\prt{2-\gamma_n\hbeta}(\lambda_n+\mu_n)\\
    &= -2\lambda_n^2+\prt{4-\gamma_n\hbeta}(\lambda_n+\mu_n)\\
    &=\theta_n.
\end{align*}
For \labelcref{itm:prop-theta-iii}, after moving all terms to the left-hand side of the equality we get
\begin{align*}
    \lambda_n^2\theta_n+2\btheta_n^2-\htheta_{n}(\lambda_{n}+\mu_{n})
    &= \lambda_n^2\prt{\prt{2-\gamma_n\hbeta}\btheta_n+\htheta_n}+2\btheta_n^2-\htheta_{n}(\lambda_{n}+\mu_{n})\\
    &= \btheta_n\prt{\lambda_n^2\prt{2-\gamma_n\hbeta}+2\btheta_n}-\htheta_n(\lambda_n+\mu_n-\lambda_n^2)\\
    &= \btheta_n\prt{2\lambda_n+2\mu_n-\gamma_n\hbeta\lambda_n^2}-\htheta_n\btheta_n = \btheta_n\htheta_n-\htheta_n\btheta_n,
\end{align*}
where in the first equality $\theta_n$ is substituted using \labelcref{itm:prop-theta-i} and in the second and the third equalities, definitions of $\btheta_n$ and $\htheta_n$ are used, respectively. 
\end{proof}

The next lemma provides alternative expressions for the term inside the first norm in \eqref{eq:ell_n}.

\begin{lemma}\label{lemma:ell_n-identical-expressions}
Suppose that \cref{assum:monotone_inclusion} holds and consider the sequences generated by \cref{alg:main}. Then, for all $n\in\nat$, the following expressions represent the same vector
\renewcommand\theenumi\labelenumi
\renewcommand{\labelenumi}{{(\roman{enumi})}}
\begin{enumerate}[ref=\Cref{lemma:ell_n-identical-expressions}~\theenumi]
    \item $p_n - (1-\alpha_n)x_n-\alpha_np_{n-1} +\tfrac{\gamma_n\hbeta\lambda_n^2}{\htheta_n}u_n - \tfrac{2\btheta_n}{\theta_n}v_n$;\label{itm:lemma2-i}
    \item $p_n-\tfrac{2\btheta_n}{\theta_n}z_n+\tfrac{\ttheta_n}{\theta_n}y_n-\tfrac{2\lambda_n}{\theta_n}x_n-\tfrac{\ptheta_n}{\theta_n}p_{n-1}+\tfrac{2\btheta_n\balpha_n}{\theta_n}z_{n-1}-\tfrac{\ttheta_n\alpha_n}{\theta_n}y_{n-1}$;\label{itm:lemma2-ii}
    \item $\tfrac{1}{\lambda_n}\prt{x_{n+1}-x_n}+\tfrac{\ttheta_n}{\htheta_n}u_n+\tfrac{(2-\gamma_n\hbeta)(\lambda_n+\mu_n)}{\theta_n}v_n$.\label{itm:lemma2-iii}
\end{enumerate}
\end{lemma}

\begin{proof} We, first, show that \labelcref{itm:lemma2-ii} represents the same vector as \labelcref{itm:lemma2-i}:
\begin{align*}
    p_n&-\tfrac{2\btheta_n}{\theta_n}z_n+\tfrac{\ttheta_n}{\theta_n}y_n-\tfrac{2\lambda_n}{\theta_n}x_n-\tfrac{\ptheta_n}{\theta_n}p_{n-1}+\tfrac{2\btheta_n\balpha_n}{\theta_n}z_{n-1}-\tfrac{\ttheta_n\alpha_n}{\theta_n}y_{n-1}\\
    &= p_n - \tfrac{2\btheta_n}{\theta_n}(z_n-\balpha_nz_{n-1}) + \tfrac{\ttheta_n}{\theta_n}(y_n-\alpha_ny_{n-1})-\tfrac{2\lambda_n}{\theta_n}x_n -\tfrac{\ptheta_n}{\theta_n}p_{n-1}\\
    &= p_n - \tfrac{2\btheta_n}{\theta_n}\prt{(1-\alpha_n)x_n + (\alpha_n-\balpha_n)p_{n-1}  + \tfrac{\btheta_n\gamma_n\hbeta}{\htheta_n}u_n + v_n}\\
    &\qquad+ \tfrac{\ttheta_n}{\theta_n}((1-\alpha_n)x_n + u_n)-\tfrac{2\lambda_n}{\theta_n}x_n -\tfrac{\ptheta_n}{\theta_n}p_{n-1}\\
    &= p_n - \tfrac{2(1-\alpha_n)\btheta_n-(1-\alpha_n)\ttheta_n+2\lambda_n}{\theta_n}x_n - \tfrac{\ptheta_n+2\btheta_n(\alpha_n-\balpha_n)}{\theta_n}p_{n-1}\\
    &\qquad+ \tfrac{\htheta\ttheta_n-2\btheta_n^2\gamma_n\hbeta}{\theta_n\htheta_n}u_n-\tfrac{2\btheta_n}{\theta_n}v_n\\
    &=p_n - (1-\alpha_n)x_n-\alpha_np_{n-1} +\tfrac{\gamma_n\hbeta\lambda_n^2}{\htheta_n}u_n - \tfrac{2\btheta_n}{\theta_n}v_n
\end{align*}
where the coefficients of the last equality are found as follows. The numerator of the coefficient of $x_n$ reads
\begin{equation}\label{eq:proof-lemma1-x_n-coefficient}
\begin{aligned}
    2(1&-\alpha_n)\btheta_n-(1-\alpha_n)\ttheta+2\lambda_n\\
    &= (1-\alpha_n)\prt{\theta_n-\prt{2-\gamma_n\hbeta}(\lambda_n+\mu_n)}\\
    &\qquad-(1-\alpha_n)\gamma_n\hbeta(\lambda_n+\mu_n)+2\lambda_n \\
    &= (1-\alpha_n)\theta_n-2(1-\alpha_n)(\lambda_n+\mu_n)+2\lambda_n\\
    &=(1-\alpha_n)\theta_n-2\tfrac{\lambda_n}{\lambda_n+\mu_n}(\lambda_n+\mu_n)+2\lambda_n = (1-\alpha_n)\theta_n
\end{aligned}
\end{equation}
where in the first equality, $\btheta_n$ is substituted from \labelcref{itm:prop-theta-ii}, and $\ttheta_n$ and $\alpha_n$ are substituted by their definitions in \cref{alg:main}. The numerator of the coefficient of $p_{n-1}$ is
\begin{equation}\label{eq:proof-lemma1-p_(n-1)-coefficient}
\begin{aligned}
    \ptheta_n+2\btheta_n(\alpha_n-\balpha_n)& = \prt{2-\gamma_n\hbeta}\mu_n+2\balpha_n\btheta_n+2\btheta_n(\alpha_n-\balpha_n)\\
    &= \prt{2-\gamma_n\hbeta}\mu_n+2\btheta_n\alpha_n\\
    &= \prt{2-\gamma_n\hbeta}\alpha_n(\lambda_n+\mu_n)+2\btheta_n\alpha_n= \alpha_n\theta_n
\end{aligned}
\end{equation}
where in the first equality \labelcref{itm:def-param-theta-prime} is used, the third equality is obtained using the definition of $\alpha_n$, and \labelcref{itm:prop-theta-ii} is utilized in the last equality. For the numerator of $u_n$ we get
\begin{equation}\label{eq:proof-lemma1-u_n-coefficient}
\begin{aligned}
    \htheta\ttheta_n-2\btheta_n^2\gamma_n\hbeta &= \htheta\gamma_n\hbeta(\lambda_n+\mu_n)-2\btheta_n^2\gamma_n\hbeta\\
    &= \gamma_n\hbeta\prt{\htheta_n(\lambda_n+\mu_n)-2\btheta_n^2} = \gamma_n\hbeta\lambda_n^2\theta_n
\end{aligned}    
\end{equation}
where the first equality is obtained by substitution of the definition of $\ttheta_n$ from \cref{alg:main}, and in the last equality \labelcref{itm:prop-theta-iii} is used.

Now, we show that \labelcref{itm:lemma2-ii} and \labelcref{itm:lemma2-iii} represent the same vector. Starting from \labelcref{itm:lemma2-ii}, we have
\begin{align*}
    p_n&-\tfrac{2\btheta_n}{\theta_n}z_n+\tfrac{\ttheta_n}{\theta_n}y_n-\tfrac{2\lambda_n}{\theta_n}x_n-\tfrac{\ptheta_n}{\theta_n}p_{n-1}+\tfrac{2\btheta_n\balpha_n}{\theta_n}z_{n-1}-\tfrac{\ttheta_n\alpha_n}{\theta_n}y_{n-1}\\
    &= p_n - \tfrac{2\btheta_n}{\theta_n}(z_n-\balpha_nz_{n-1}) + \tfrac{\ttheta_n}{\theta_n}(y_n-\alpha_ny_{n-1})-\tfrac{2\lambda_n}{\theta_n}x_n -\tfrac{\ptheta_n}{\theta_n}p_{n-1}\\
    &= \tfrac{1}{\lambda_n}(x_{n+1}-x_{n}) + z_n + \balpha_n(p_{n-1}-z_{n-1}) - \tfrac{2\btheta_n}{\theta_n}(z_n-\balpha_nz_{n-1})  \\ 
    &\qquad+ \tfrac{\ttheta_n}{\theta_n}(y_n-\alpha_ny_{n-1})-\tfrac{2\lambda_n}{\theta_n}x_n -\tfrac{\ptheta_n}{\theta_n}p_{n-1}\\
    &= \tfrac{1}{\lambda_n}(x_{n+1}-x_{n})+\tfrac{\theta_n-2\btheta_n}{\theta_n}\prt{z_n-\balpha_nz_{n-1}} + \tfrac{\ttheta_n}{\theta_n}(y_n-\alpha_ny_{n-1})-\tfrac{2\lambda_n}{\theta_n}x_n\\ 
    &\qquad + \tfrac{\balpha_n\theta_n-\ptheta_n}{\theta_n}p_{n-1}\\
    &= \tfrac{1}{\lambda_n}(x_{n+1}-x_{n})+\tfrac{\theta_n-2\btheta_n}{\theta_n}\prt{(1-\alpha_n)x_n + (\alpha_n-\balpha_n)p_{n-1}  + \tfrac{\btheta_n\gamma_n\hbeta}{\htheta_n}u_n + v_n}\\
    &\qquad+ \tfrac{\ttheta_n}{\theta_n}\prt{(1-\alpha_n)x_n + u_n}-\tfrac{2\lambda_n}{\theta_n}x_n + \tfrac{\balpha\theta_n-\ptheta_n}{\theta_n}p_{n-1}\\
    &= \tfrac{1}{\lambda_n}(x_{n+1}-x_{n}) + \tfrac{\prt{\theta_n-2\btheta_n}\btheta_n\gamma_n\hbeta+\ttheta_n\htheta_n}{\theta_n\htheta_n}u_n + \tfrac{\theta_n-2\btheta_n}{\theta_n}v_n\\
    &\qquad+ \tfrac{\prt{\theta_n-2\btheta_n}(1-\alpha_n)+\ttheta_n(1-\alpha_n)-2\lambda_n}{\theta_n}x_n + \tfrac{(\theta_n-2\btheta_n)(\alpha_n-\balpha_n)+\balpha_n\theta_n-\ptheta_n}{\theta_n}p_{n-1}\\
    &=\tfrac{1}{\lambda_n}\prt{x_{n+1}-x_n}+\tfrac{\ttheta_n}{\htheta_n}u_n+\tfrac{(2-\gamma_n\hbeta)(\lambda_n+\mu_n)}{\theta_n}v_n
\end{align*}
In the second equality, the definition of $x_{n+1}$ in step~\ref{alg:main:x} of \cref{alg:main} is used. In the fourth equality, the definition of $y_n$ in step~\ref{alg:main:y} and the definition of $z_n$ in step~\ref{alg:main:z} of \cref{alg:main} are used. In the last equality, the coefficient of $x_n$ is found to be $-\tfrac{1}{\lambda_n}$ by \eqref{eq:proof-lemma1-x_n-coefficient}, the coefficient of $p_{n-1}$ is zero by  \eqref{eq:proof-lemma1-p_(n-1)-coefficient}, the coefficient of $v_n$ is found by \labelcref{itm:prop-theta-ii}, and for the coefficient of $u_n$ we have 
\begin{align*}
    \prt{\theta_n-2\btheta_n}\btheta_n\gamma_n\hbeta+\ttheta_n\htheta_n &= \theta_n\btheta_n\gamma_n\hbeta-2\btheta^2\gamma_n\hbeta+\ttheta_n\htheta_n\\
    &= \theta_n\btheta_n\gamma_n\hbeta + \gamma_n\hbeta\lambda_n^2\theta_n\\
    &= \theta_n\gamma_n\hbeta\prt{\btheta_n+\lambda_n^2}=\theta_n\gamma_n\hbeta\prt{\lambda_n+\mu_n}
\end{align*}
where  the second equality is obtained by \eqref{eq:proof-lemma1-u_n-coefficient}, and in the last equality the definition of $\btheta_n$ is used. This concludes the proof.
\end{proof}

\subsection{Proof of \cref{thm:LP-identity}}

\begin{proof}
We define the following  quantity 
\begin{equation}\label{eq:Lyapunov-equality}
    \begin{aligned}
        \Delta_n \defeq V_{n+1} - V_n &+ 2\gamma_n(\lambda_n-\balpha_{n+1}\lambda_{n+1})\phi_n + \el{n-1}\\
        &- (\lambda_{n}+\mu_{n})\lp\tfrac{\ttheta_{n}}{\htheta_{n}}\norm{u_{n}}_{M}^2 + \tfrac{\htheta_{n}}{\theta_{n}}\norm{v_{n}}_{M}^2\rp 
    \end{aligned}
\end{equation}
and prove the result by showing that, for all $n\in\nat$, it is identical to zero. By substituting $V_{n+1}$ and $V_n$ in \eqref{eq:Lyapunov-equality}, we get
\begin{align*}
    \Delta_n &= \norm{x_{n+1}-\xs}_M^2 + \el{n} + 2\lambda_{n+1}\gamma_{n+1}\alpha_{n+1}\phi_n\\
    &\quad- \norm{x_{n}-\xs}_M^2 - \el{n-1} - 2\lambda_{n}\gamma_{n}\alpha_{n}\phi_{n-1}\\
    &\quad + 2\gamma_n(\lambda_n-\balpha_{n+1}\lambda_{n+1})\phi_n + \el{n-1} -  (\lambda_{n}+\mu_{n})\lp\tfrac{\ttheta_{n}}{\htheta_{n}}\norm{u_{n}}_{M}^2+\tfrac{\htheta_{n}}{\theta_{n}}\norm{v_{n}}_{M}^2\rp\\
    &= \norm{x_{n+1}-\xs}_M^2 - \norm{x_{n}-\xs}_M^2 + \el{n} - 2\lambda_{n}\gamma_{n}\alpha_{n}\phi_{n-1} + 2\gamma_n\lambda_n\phi_n \nonumber\\
    &\quad - (\lambda_{n}+\mu_{n})\lp\tfrac{\ttheta_{n}}{\htheta_{n}}\norm{u_{n}}_{M}^2+\tfrac{\htheta_{n}}{\theta_{n}}\norm{v_{n}}_{M}^2\rp,
\end{align*}
where in the last equality we used $\gamma_n\balpha_{n+1}=\gamma_{n+1}\alpha_{n+1}$. Next, substituting $\el{n}$ from \eqref{eq:ell_n}, and $\phi_{n-1}$ and $\phi_n$ from \eqref{eq:phi-def} on the right-hand side of the last equality above, yields
\begin{align*}
    \Delta_n &= \norm{x_{n+1}-\xs}_M^2 - \norm{x_{n}-\xs}_M^2\\
    &\quad+ \tfrac{1}{2}\theta_n\norm{p_n-x_n+\alpha_n(x_n-p_{n-1}) +\tfrac{\gamma_n\hbeta\lambda_n^2}{\htheta_n}u_n - \tfrac{2\btheta_n}{\theta_n}v_n}_M^2 \nonumber\\
    &\quad+ 2\mu_{n}\gamma_{n}\inpr{\tfrac{z_{n}-p_{n}}{\gamma_{n}}-\tfrac{z_{n-1}-p_{n-1}}{\gamma_{n-1}}}{p_{n}-p_{n-1}}_M\\
    &\quad+ \tfrac{\mu_{n}\gamma_{n}\hbeta}{2}\norm{p_{n}-y_{n}-(p_{n-1}-y_{n-1})}_M^2\\
    &\quad- 2\lambda_{n}\gamma_{n}\alpha_{n}\lp\inpr{\tfrac{z_{n-1}-p_{n-1}}{\gamma_{n-1}}}{p_{n-1}-\xs}_M + \tfrac{\hbeta}{4}\norm{y_{n-1}-p_{n-1}}_M^2\rp\\
    &\quad+ 2\lambda_n\gamma_n\lp\inpr{\tfrac{z_n-p_n}{\gamma_n}}{p_n-\xs}_M + \tfrac{\hbeta}{4}\norm{y_n-p_n}_M^2\rp\\
    &\quad - (\lambda_{n}+\mu_{n})\lp\tfrac{\ttheta_{n}}{\htheta_{n}}\norm{u_{n}}_{M}^2+\tfrac{\htheta_{n}}{\theta_{n}}\norm{v_{n}}_{M}^2\rp\\
    &= \norm{x_{n+1}-\xs}_M^2 - \norm{x_{n}-\xs}_M^2 + 2\lambda_n\gamma_n\inpr{\tfrac{z_n-p_n}{\gamma_n}}{p_n-\xs}_M\\
    &\quad- 2\lambda_{n}\gamma_{n}\alpha_{n}\inpr{\tfrac{z_{n-1}-p_{n-1}}{\gamma_{n-1}}}{p_{n-1}-p_n+p_n-\xs}_M \\
    &\quad+ 2\mu_{n}\gamma_{n}\inpr{\tfrac{z_{n}-p_{n}}{\gamma_{n}}-\tfrac{z_{n-1}-p_{n-1}}{\gamma_{n-1}}}{p_{n}-p_{n-1}}_M\\
    &\quad+ \tfrac{1}{2}\theta_n\norm{p_n-x_n+\alpha_n(x_n-p_{n-1}) +\tfrac{\gamma_n\hbeta\lambda_n^2}{\htheta_n}u_n - \tfrac{2\btheta_n}{\theta_n}v_n}_M^2\\
    &\quad+ \tfrac{\mu_{n}\gamma_{n}\hbeta}{2}\norm{p_{n}-y_{n}-(p_{n-1}-y_{n-1})}_M^2 - \tfrac{\lambda_{n}\gamma_{n}\alpha_{n}\hbeta}{2}\norm{y_{n-1}-p_{n-1}}_M^2\\
    &\quad + \tfrac{\lambda_n\gamma_n\hbeta}{2}\norm{y_n-p_n}_M^2 - (\lambda_{n}+\mu_{n})\lp\tfrac{\ttheta_{n}}{\htheta_{n}}\norm{u_{n}}_{M}^2+\tfrac{\htheta_{n}}{\theta_{n}}\norm{v_{n}}_{M}^2\rp\\
    &= \norm{x_{n+1}-\xs}_M^2 - \norm{x_{n}-\xs}_M^2\\
    &\quad+ 2\inpr{\lambda_n(z_n-p_n)-\balpha_n\lambda_n(z_{n-1}-p_{n-1})}{p_n-\xs}_M\\
    &\quad+ 2\inpr{\mu_{n}(z_{n}-p_{n})+(\balpha_n\lambda_n-\tfrac{\gamma_n}{\gamma_{n-1}}\mu_n)(z_{n-1}-p_{n-1})}{p_{n}-p_{n-1}}_M\\
    &\quad+ \tfrac{1}{2}\theta_n\norm{p_n-x_n+\alpha_n(x_n-p_{n-1}) +\tfrac{\gamma_n\hbeta\lambda_n^2}{\htheta_n}u_n - \tfrac{2\btheta_n}{\theta_n}v_n}_M^2\\
    &\quad+ \tfrac{\mu_{n}\gamma_{n}\hbeta}{2}\norm{p_{n}-y_{n}-(p_{n-1}-y_{n-1})}_M^2 - \tfrac{\lambda_{n}\gamma_{n}\alpha_{n}\hbeta}{2}\norm{y_{n-1}-p_{n-1}}_M^2\\
    &\quad + \tfrac{\lambda_n\gamma_n\hbeta}{2}\norm{y_n-p_n}_M^2 - (\lambda_{n}+\mu_{n})\lp\tfrac{\ttheta_{n}}{\htheta_{n}}\norm{u_{n}}_{M}^2+\tfrac{\htheta_{n}}{\theta_{n}}\norm{v_{n}}_{M}^2\rp.
\end{align*}
We define 
\begin{equation}\label{eq:omega}
    \omega_n\defeq\balpha_n\lambda_n-\tfrac{\gamma_n}{\gamma_{n-1}}\mu_n   
\end{equation}
and substitute it in the last equality above; and also from step~\ref{alg:main:x} of \cref{alg:main}, we replace $\lambda_n(z_n-p_n)-\balpha_n\lambda_n(z_{n-1}-p_{n-1})$ by $x_n-x_{n+1}$. Then, we get
\begin{align*}
    \Delta_n &= \norm{x_{n+1}-\xs}_M^2 - \norm{x_{n}-\xs}_M^2 + 2\inpr{x_n-x_{n+1}}{p_n-\xs}_M\\
    &\quad+ 2\inpr{\mu_{n}(z_{n}-p_{n})+\omega_n(z_{n-1}-p_{n-1})}{p_{n}-p_{n-1}}_M\\
    &\quad+ \tfrac{1}{2}\theta_n\norm{p_n-x_n+\alpha_n(x_n-p_{n-1}) +\tfrac{\gamma_n\hbeta\lambda_n^2}{\htheta_n}u_n - \tfrac{2\btheta_n}{\theta_n}v_n}_M^2\\
    &\quad+ \tfrac{\mu_{n}\gamma_{n}\hbeta}{2}\norm{p_{n}-y_{n}-(p_{n-1}-y_{n-1})}_M^2 - \tfrac{\lambda_{n}\gamma_{n}\alpha_{n}\hbeta}{2}\norm{y_{n-1}-p_{n-1}}_M^2\\
    &\quad + \tfrac{\lambda_n\gamma_n\hbeta}{2}\norm{y_n-p_n}_M^2 - (\lambda_{n}+\mu_{n})\lp\tfrac{\ttheta_{n}}{\htheta_{n}}\norm{u_{n}}_{M}^2+\tfrac{\htheta_{n}}{\theta_{n}}\norm{v_{n}}_{M}^2\rp\\
    &= \norm{x_{n+1}-p_n}_M^2 - \norm{x_{n}-p_n}_M^2\\
    &\quad+ 2\inpr{\mu_{n}(z_{n}-p_{n})+\omega_n(z_{n-1}-p_{n-1})}{p_{n}-p_{n-1}}_M\\
    &\quad+ \tfrac{1}{2}\theta_n\norm{p_n-x_n+\alpha_n(x_n-p_{n-1}) +\tfrac{\gamma_n\hbeta\lambda_n^2}{\htheta_n}u_n - \tfrac{2\btheta_n}{\theta_n}v_n}_M^2\\
    &\quad+ \tfrac{\mu_{n}\gamma_{n}\hbeta}{2}\norm{p_{n}-y_{n}-(p_{n-1}-y_{n-1})}_M^2 - \tfrac{\lambda_{n}\gamma_{n}\alpha_{n}\hbeta}{2}\norm{y_{n-1}-p_{n-1}}_M^2\\
    &\quad + \tfrac{\lambda_n\gamma_n\hbeta}{2}\norm{y_n-p_n}_M^2 - (\lambda_{n}+\mu_{n})\lp\tfrac{\ttheta_{n}}{\htheta_{n}}\norm{u_{n}}_{M}^2+\tfrac{\htheta_{n}}{\theta_{n}}\norm{v_{n}}_{M}^2\rp
\end{align*}
where in the last equality we used the identity $2\inpr{a-b}{c-d}_M+\norm{b-d}_M^2-\norm{a-d}_M^2=\norm{b-c}_M^2-\norm{a-c}_M^2$ for all $a,b,c,d\in\PrimS$. Now, inserting $x_{n+1}$ from step~\ref{alg:main:x} of \cref{alg:main}, yields
\begin{align*}
    \Delta_n &= \norm{x_n-p_n+\lambda_n(p_n - z_n) + \lambda_n\balpha_n(z_{n-1}-p_{n-1})}_M^2 - \norm{x_{n}-p_n}_M^2 \\
    &\quad+ 2\inpr{\mu_{n}(z_{n}-p_{n})+\omega_n(z_{n-1}-p_{n-1})}{p_{n}-p_{n-1}}_M\\
    &\quad+ \tfrac{1}{2}\theta_n\norm{p_n-x_n+\alpha_n(x_n-p_{n-1}) +\tfrac{\gamma_n\hbeta\lambda_n^2}{\htheta_n}u_n - \tfrac{2\btheta_n}{\theta_n}v_n}_M^2\\
    &\quad+ \tfrac{\mu_{n}\gamma_{n}\hbeta}{2}\norm{p_{n}-y_{n}-(p_{n-1}-y_{n-1})}_M^2 - \tfrac{\lambda_{n}\gamma_{n}\alpha_{n}\hbeta}{2}\norm{y_{n-1}-p_{n-1}}_M^2\\
    &\quad + \tfrac{\lambda_n\gamma_n\hbeta}{2}\norm{y_n-p_n}_M^2 - (\lambda_{n}+\mu_{n})\lp\tfrac{\ttheta_{n}}{\htheta_{n}}\norm{u_{n}}_{M}^2+\tfrac{\htheta_{n}}{\theta_{n}}\norm{v_{n}}_{M}^2\rp\\
    &=\norm{\lambda_n(p_n-z_n)+\lambda_n\balpha_n(z_{n-1}-p_{n-1})}_M^2\\
    &\quad+ 2\inpr{x_n-p_n}{\lambda_n(p_n-z_n)+\lambda_n\balpha_n(z_{n-1}-p_{n-1})}_M \\
    &\quad+ 2\inpr{\mu_{n}(z_{n}-p_{n})+\omega_n(z_{n-1}-p_{n-1})}{p_{n}-p_{n-1}}_M\\
    &\quad+ \tfrac{1}{2}\theta_n\norm{p_n-x_n+\alpha_n(x_n-p_{n-1}) +\tfrac{\gamma_n\hbeta\lambda_n^2}{\htheta_n}u_n - \tfrac{2\btheta_n}{\theta_n}v_n}_M^2\\
    &\quad+ \tfrac{\mu_{n}\gamma_{n}\hbeta}{2}\norm{p_{n}-y_{n}-(p_{n-1}-y_{n-1})}_M^2 - \tfrac{\lambda_{n}\gamma_{n}\alpha_{n}\hbeta}{2}\norm{y_{n-1}-p_{n-1}}_M^2\\
    &\quad + \tfrac{\lambda_n\gamma_n\hbeta}{2}\norm{y_n-p_n}_M^2 - (\lambda_{n}+\mu_{n})\lp\tfrac{\ttheta_{n}}{\htheta_{n}}\norm{u_{n}}_{M}^2+\tfrac{\htheta_{n}}{\theta_{n}}\norm{v_{n}}_{M}^2\rp.
\end{align*}
Next, using \cref{lemma:ell_n-identical-expressions} and steps~\ref{alg:main:y}--\ref{alg:main:z} of \cref{alg:main}, we replace the terms including $u_n$ and $v_n$ in terms of the iterates
\begin{align*}
    \Delta_n &= \norm{\lambda_n(p_n-z_n)+\lambda_n\balpha_n(z_{n-1}-p_{n-1})}_M^2\\
    &\quad+ 2\inpr{x_n-p_n}{\lambda_n(p_n-z_n)+\lambda_n\balpha_n(z_{n-1}-p_{n-1})}_M \\
    &\quad+ 2\inpr{\mu_{n}(z_{n}-p_{n})+\omega_n(z_{n-1}-p_{n-1})}{p_{n}-p_{n-1}}_M\\
    &\quad+ \tfrac{\theta_n}{2}\norm{p_n-\tfrac{2\btheta_n}{\theta_n}z_n+\tfrac{\ttheta_n}{\theta_n}y_n-\tfrac{2\lambda_n}{\theta_n}x_n-\tfrac{\ptheta_n}{\theta_n}p_{n-1}+\tfrac{2\btheta_n\balpha_n}{\theta_n}z_{n-1}-\tfrac{\ttheta_n\alpha_n}{\theta_n}y_{n-1}}_M^2\\
    &\quad+ \tfrac{\mu_{n}\gamma_{n}\hbeta}{2}\norm{p_{n}-y_{n}-(p_{n-1}-y_{n-1})}_M^2 - \tfrac{\lambda_{n}\gamma_{n}\alpha_{n}\hbeta}{2}\norm{y_{n-1}-p_{n-1}}_M^2\\
    &\quad + \tfrac{\lambda_n\gamma_n\hbeta}{2}\norm{y_n-p_n}_M^2 - \tfrac{(\lambda_{n}+\mu_{n})\ttheta_{n}}{\htheta_{n}}\norm{y_n-(1-\alpha_n)x_n-\alpha_ny_{n-1}}_{M}^2\\
    &\quad- \tfrac{(\lambda_{n}+\mu_{n})\htheta_{n}}{\theta_{n}}\Big\Vert z_n-\tfrac{\btheta_n\gamma_n\hbeta}{\htheta_n}y_n-\tfrac{\prt{2-\gamma_n\hbeta}\lambda_n}{\htheta_n}x_n\\
    &\hspace{40mm}+(\balpha_n-\alpha_n)p_{n-1}-\balpha_nz_{n-1}+\tfrac{\alpha_n\btheta_n\gamma_n\hbeta}{\htheta_n}y_{n-1}\Big\Vert_{M}^2
\end{align*}
where we used
\begin{align*}
    v_n&=z_n-(1-\alpha_n)x_n+(\balpha_n-\alpha_n)p_{n-1}-\balpha_nz_{n-1}-\hbeta\tfrac{\btheta_n\gamma_n}{\htheta_n}u_n\\
    &=z_n-\tfrac{\btheta_n\gamma_n\hbeta}{\htheta_n}y_n-\tfrac{\prt{2-\gamma_n\hbeta}\lambda_n}{\htheta_n}x_n+(\balpha_n-\alpha_n)p_{n-1}-\balpha_nz_{n-1}+\tfrac{\alpha_n\btheta_n\gamma_n\hbeta}{\htheta_n}y_{n-1}
\end{align*}
which is obtained by substituting $u_n$ from step~\ref{alg:main:y} into step~\ref{alg:main:z} of \cref{alg:main}. Next, we expand the terms on the right-hand side of the last equality above which include $p_n$. This yields
\begin{align*}
    \Delta_n &= \lambda_n^2\norm{p_n}_M^2+2\inpr{p_n}{\lambda_n^2\prt{-z_n+\balpha_nz_{n-1}-\balpha_np_{n-1}}}_M\\
    &\quad+ \lambda_n^2\norm{z_n-\balpha_nz_{n-1}+\balpha_np_{n-1}}_M^2\\
    &\quad- 2\lambda_n\norm{p_n}_M^2+2\inpr{p_n}{\lambda_n\prt{z_n+x_n+\balpha_np_{n-1}-\balpha_nz_{n-1}}}_M\\
    &\quad+ 2\inpr{x_n}{\lambda_n\prt{-z_n-\balpha_np_{n-1}+\balpha_nz_{n-1}}}_M\\
    &\quad- 2\mu_n\norm{p_n}_M^2+2\inpr{p_n}{\mu_{n}z_{n}+(\mu_n-\omega_n)p_{n-1}+\omega_nz_{n-1}}_M\\
    &\quad+ 2\inpr{\mu_{n}z_{n}+\omega_n(z_{n-1}-p_{n-1})}{-p_{n-1}}_M + \tfrac{1}{2}\theta_n\norm{p_n}_M^2\\
    &\quad + 2\inpr{p_n}{\tfrac{\theta_n}{2}\prt{-\tfrac{2\btheta_n}{\theta_n}z_n+\tfrac{\ttheta_n}{\theta_n}y_n-\tfrac{2\lambda_n}{\theta_n}x_n}}_M\\
    &\quad + 2\inpr{p_n}{\tfrac{\theta_n}{2}\prt{-\tfrac{\ptheta_n}{\theta_n}p_{n-1}+\tfrac{2\btheta_n\balpha_n}{\theta_n}z_{n-1}-\tfrac{\ttheta_n\alpha_n}{\theta_n}y_{n-1}}}_M\\
    &\quad+ \tfrac{1}{2}\theta_n\norm{-\tfrac{2\btheta_n}{\theta_n}z_n+\tfrac{\ttheta_n}{\theta_n}y_n-\tfrac{2\lambda_n}{\theta_n}x_n-\tfrac{\ptheta_n}{\theta_n}p_{n-1}+\tfrac{2\btheta_n\balpha_n}{\theta_n}z_{n-1}-\tfrac{\ttheta_n\alpha_n}{\theta_n}y_{n-1}}_M^2\\
    &\quad+ \tfrac{\mu_{n}\gamma_{n}\hbeta}{2}\norm{p_{n}}_M^2 + 2\inpr{p_n}{\tfrac{\mu_{n}\gamma_{n}\hbeta}{2}\prt{-y_{n}-(p_{n-1}-y_{n-1})}}_M\\
    &\quad+ \tfrac{\mu_{n}\gamma_{n}\hbeta}{2}\norm{-y_{n}-(p_{n-1}-y_{n-1})}_M^2 - \tfrac{\lambda_{n}\gamma_{n}\alpha_{n}\hbeta}{2}\norm{y_{n-1}-p_{n-1}}_M^2\\
    &\quad + \tfrac{\lambda_n\gamma_n\hbeta}{2}\norm{p_n}_M^2 - 2\inpr{p_n}{\tfrac{\lambda_n\gamma_n\hbeta}{2}y_n}_M + \tfrac{\lambda_n\gamma_n\hbeta}{2}\norm{y_n}_M^2\\
    &\quad- \tfrac{\ttheta_{n}(\lambda_{n}+\mu_{n})}{\htheta_{n}}\norm{y_n-(1-\alpha_n)x_n-\alpha_ny_{n-1}}_{M}^2\\
    &\quad- \tfrac{\htheta_{n}(\lambda_{n}+\mu_{n})}{\theta_{n}}\Big\Vert z_n-\tfrac{\btheta_n\gamma_n\hbeta}{\htheta_n}y_n-\tfrac{\prt{2-\gamma_n\hbeta}\lambda_n}{\htheta_n}x_n\\
    &\hspace{30mm} +(\balpha_n-\alpha_n)p_{n-1}-\balpha_nz_{n-1}+\tfrac{\alpha_n\btheta_n\gamma_n\hbeta}{\htheta_n}y_{n-1}\Big\Vert_{M}^2\\
    &= \lp\lambda_n^2-\tfrac{\prt{4-\gamma_n\hbeta}\prt{\lambda_n+\mu_n}}{2}+\tfrac{1}{2}\theta_n\rp\norm{p_n}_M^2 \\
    &\quad+ 2\inpr{p_n}{\prt{\lambda_n+\mu_n-\lambda_n^2-\btheta_n}z_n + \prt{\tfrac{\ttheta_n}{2}-\tfrac{\prt{\lambda_n+\mu_n}\gamma_n\hbeta}{2}}y_n}_M\\
    &\quad+ 2\inpr{p_n}{\prt{\lambda_n-\lambda_n}x_n + \prt{(1-\lambda_n)\lambda_n\balpha_n+\mu_n-\omega_n-\tfrac{1}{2}\ptheta_n-\tfrac{\mu_n\gamma_n\hbeta}{2}}p_{n-1}}_M\\
    &\quad+ 2\inpr{p_n}{\prt{\lambda_n^2\balpha_n-\lambda_n\balpha_n+\omega_n+\btheta_n\balpha_n}z_{n-1} + \prt{-\tfrac{1}{2}\ttheta_n\alpha_n+\tfrac{\mu_n\gamma_n\hbeta}{2}}y_{n-1}}_M\\
    &\quad+ \lambda_n^2\norm{z_n-\balpha_nz_{n-1}+\balpha_np_{n-1}}_M^2 + 2\inpr{x_n}{\lambda_n\prt{-z_n-\balpha_np_{n-1}+\balpha_nz_{n-1}}}_M\\
    &\quad+ 2\inpr{\mu_{n}z_{n}+\omega_n(z_{n-1}-p_{n-1})}{-p_{n-1}}_M\\
    &\quad+ \tfrac{1}{2}\theta_n\norm{-\tfrac{2\btheta_n}{\theta_n}z_n+\tfrac{\ttheta_n}{\theta_n}y_n-\tfrac{2\lambda_n}{\theta_n}x_n-\tfrac{\ptheta_n}{\theta_n}p_{n-1}+\tfrac{2\btheta_n\balpha_n}{\theta_n}z_{n-1}-\tfrac{\ttheta_n\alpha_n}{\theta_n}y_{n-1}}_M^2\\
    &\quad+ \tfrac{\mu_{n}\gamma_{n}\hbeta}{2}\norm{-y_{n}-(p_{n-1}-y_{n-1})}_M^2 - \tfrac{\lambda_{n}\gamma_{n}\alpha_{n}\hbeta}{2}\norm{y_{n-1}-p_{n-1}}_M^2\\
    &\quad+ \tfrac{\lambda_n\gamma_n\hbeta}{2}\norm{y_n}_M^2 - \tfrac{\ttheta_{n}(\lambda_{n}+\mu_{n})}{\htheta_{n}}\norm{y_n-(1-\alpha_n)x_n-\alpha_ny_{n-1}}_{M}^2\\
    &\quad- \tfrac{\htheta_{n}(\lambda_{n}+\mu_{n})}{\theta_{n}}\Big\Vert{}z_n-\tfrac{\btheta_n\gamma_n\hbeta}{\htheta_n}y_n-\tfrac{\prt{2-\gamma_n\hbeta}\lambda_n}{\htheta_n}x_n\\
    &\hspace{30mm}+(\balpha_n-\alpha_n)p_{n-1}-\balpha_nz_{n-1}+\tfrac{\alpha_n\btheta_n\gamma_n\hbeta}{\htheta_n}y_{n-1}\Big\Vert_{M}^2.
\end{align*}
All terms involving $p_n$ in this expression are identically zero since their coefficients become zero. This is for most terms straightforward to show by substituting $\theta_n$, $\btheta_n$, $\ttheta_n$, $\ptheta_n$, $\alpha_n$, and $\balpha_n$ defined in \cref{alg:main} into the corresponding coefficients. We show this for two coefficients for which it is less obvious. For the coefficient of $\inpr{p_n}{p_{n-1}}_M$ we have
\begin{align*}
    -\lambda_n^2\balpha_n&+\lambda_n\balpha_n+\mu_n-\omega_n-\tfrac{1}{2}\ptheta_n-\tfrac{\mu_n\gamma_n\hbeta}{2}\\ 
    &= -\lambda_n^2\balpha_n+\lambda_n\balpha_n+\mu_n-\prt{\lambda_n\balpha_n-\tfrac{\gamma_n}{\gamma_{n-1}}\mu_n}-\tfrac{1}{2}\ptheta_n-\tfrac{\mu_n\gamma_n\hbeta}{2}\\
    &= -\lambda_n^2\balpha_n+\tfrac{\gamma_n}{\gamma_{n-1}}\mu_n+\tfrac{\prt{2-\gamma_n\hbeta}\mu_n}{2}-\tfrac{1}{2}\ptheta_n\\
    &= -\lambda_n^2\balpha_n+(\lambda_n+\mu_n)\balpha_n+\tfrac{\prt{2-\gamma_n\hbeta}\mu_n}{2}-\tfrac{1}{2}\ptheta_n\\
    &= \btheta_n\balpha_n+\tfrac{\prt{2-\gamma_n\hbeta}\mu_n}{2}-\tfrac{1}{2}\ptheta_n=\tfrac{1}{2}\ptheta_n-\tfrac{1}{2}\ptheta_n=0,
\end{align*}
where in the first equality $\omega_n$ is substituted from \eqref{eq:omega} and in the third equality the definition of $\balpha_n$ is used. For the coefficient of $\inpr{p_n}{z_{n-1}}_M$
\begin{align*}
    \lambda_n^2\balpha_n&-\lambda_n\balpha_n+\omega_n+\btheta_n\balpha_n\\
    &= \lambda_n^2\balpha_n-\lambda_n\balpha_n+\lambda_n\balpha_n-\tfrac{\gamma_n}{\gamma_{n-1}}\mu_n+\btheta_n\balpha_n\\
    &= \lambda_n^2\balpha_n-\prt{\lambda_n+\mu_n}\balpha_n+\btheta_n\balpha_n= (\btheta_n-\btheta_n)\balpha_n=0.
\end{align*}
Next, for the terms containing $z_n$, we do a similar procedure of expanding, reordering, and recollecting the terms  as we did for $p_n$. This gives
\begin{align*}
    \Delta_n &= \lambda_n^2\norm{z_n-\balpha_nz_{n-1}+\balpha_np_{n-1}}_M^2 + 2\inpr{x_n}{\lambda_n\prt{-z_n-\balpha_np_{n-1}+\balpha_nz_{n-1}}}_M\\
    &\quad+ 2\inpr{\mu_{n}z_{n}+\omega_n(z_{n-1}-p_{n-1})}{-p_{n-1}}_M\\
    &\quad+ \tfrac{1}{2}\theta_n\norm{-\tfrac{2\btheta_n}{\theta_n}z_n+\tfrac{\ttheta_n}{\theta_n}y_n-\tfrac{2\lambda_n}{\theta_n}x_n-\tfrac{\ptheta_n}{\theta_n}p_{n-1}+\tfrac{2\btheta_n\balpha_n}{\theta_n}z_{n-1}-\tfrac{\ttheta_n\alpha_n}{\theta_n}y_{n-1}}_M^2\\
    &\quad+ \tfrac{\mu_{n}\gamma_{n}\hbeta}{2}\norm{-y_{n}-(p_{n-1}-y_{n-1})}_M^2 - \tfrac{\lambda_{n}\gamma_{n}\alpha_{n}\hbeta}{2}\norm{y_{n-1}-p_{n-1}}_M^2\\
    &\quad+ \tfrac{\lambda_n\gamma_n\hbeta}{2}\norm{y_n}_M^2 - \tfrac{\ttheta_{n}(\lambda_{n}+\mu_{n})}{\htheta_{n}}\norm{y_n-(1-\alpha_n)x_n-\alpha_ny_{n-1}}_{M}^2\\
    &\quad- \tfrac{\htheta_{n}(\lambda_{n}+\mu_{n})}{\theta_{n}}\Big\Vert{}z_n-\tfrac{\btheta_n\gamma_n\hbeta}{\htheta_n}y_n-\tfrac{\prt{2-\gamma_n\hbeta}\lambda_n}{\htheta_n}x_n\\
    &\hspace{30mm}+(\balpha_n-\alpha_n)p_{n-1}-\balpha_nz_{n-1}+\tfrac{\alpha_n\btheta_n\gamma_n\hbeta}{\htheta_n}y_{n-1}\Big\Vert_{M}^2\\
    &= \lambda_n^2\norm{z_n}_M^2 + 2\inpr{z_n}{\lambda_n^2\prt{-\balpha_nz_{n-1}+\balpha_np_{n-1}}}_M +  \lambda_n^2\norm{-\balpha_nz_{n-1}+\balpha_np_{n-1}}_M^2\\
    &\quad+ 2\inpr{z_n}{-\lambda_nx_n}_M + 2\inpr{x_n}{\lambda_n\prt{-\balpha_np_{n-1}+\balpha_nz_{n-1}}}_M\\
    &\quad+ 2\inpr{z_{n}}{-\mu_{n}p_{n-1}}_M + 2\inpr{p_{n-1}}{-\omega_n(z_{n-1}-p_{n-1})}_M + \tfrac{2\btheta_n^2}{\theta_n}\norm{z_n}_M^2\\
    &\quad + 2\inpr{z_n}{-\btheta_n\prt{\tfrac{\ttheta_n}{\theta_n}y_n-\tfrac{2\lambda_n}{\theta_n}x_n-\tfrac{\ptheta_n}{\theta_n}p_{n-1}+\tfrac{2\btheta_n\balpha_n}{\theta_n}z_{n-1}-\tfrac{\ttheta_n\alpha_n}{\theta_n}y_{n-1}}}_M \\
    &\quad+ \tfrac{1}{2}\theta_n\norm{\tfrac{\ttheta_n}{\theta_n}y_n-\tfrac{2\lambda_n}{\theta_n}x_n-\tfrac{\ptheta_n}{\theta_n}p_{n-1}+\tfrac{2\btheta_n\balpha_n}{\theta_n}z_{n-1}-\tfrac{\ttheta_n\alpha_n}{\theta_n}y_{n-1}}_M^2\\
    &\quad+ \tfrac{\mu_{n}\gamma_{n}\hbeta}{2}\norm{-y_{n}-(p_{n-1}-y_{n-1})}_M^2 - \tfrac{\lambda_{n}\gamma_{n}\alpha_{n}\hbeta}{2}\norm{y_{n-1}-p_{n-1}}_M^2\\
    &\quad+ \tfrac{\lambda_n\gamma_n\hbeta}{2}\norm{y_n}_M^2 - \tfrac{\ttheta_{n}(\lambda_{n}+\mu_{n})}{\htheta_{n}}\norm{y_n-(1-\alpha_n)x_n-\alpha_ny_{n-1}}_{M}^2\\
    &\quad- \tfrac{\htheta_{n}(\lambda_{n}+\mu_{n})}{\theta_{n}}\norm{z_n}_{M}^2 + 2\inpr{z_n}{-\tfrac{\htheta_{n}(\lambda_{n}+\mu_{n})}{\theta_{n}}\prt{-\tfrac{\btheta_n\gamma_n\hbeta}{\htheta_n}y_n-\tfrac{\prt{2-\gamma_n\hbeta}\lambda_n}{\htheta_n}x_n}}_M\\
    &\quad+ 2\inpr{z_n}{-\tfrac{\htheta_{n}(\lambda_{n}+\mu_{n})}{\theta_{n}}\prt{(\balpha_n-\alpha_n)p_{n-1}-\balpha_nz_{n-1}+\tfrac{\alpha_n\btheta_n\gamma_n\hbeta}{\htheta_n}y_{n-1}}}_M\\
    &\quad- \tfrac{\htheta_{n}(\lambda_{n}+\mu_{n})}{\theta_{n}}\Big\Vert{}-\tfrac{\btheta_n\gamma_n\hbeta}{\htheta_n}y_n-\tfrac{\prt{2-\gamma_n\hbeta}\lambda_n}{\htheta_n}x_n\\
    &\hspace{30mm}+(\balpha_n-\alpha_n)p_{n-1}-\balpha_nz_{n-1}+\tfrac{\alpha_n\btheta_n\gamma_n\hbeta}{\htheta_n}y_{n-1}\Big\Vert_{M}^2\\
    &= \tfrac{\lambda_n^2\theta_n+2\btheta_n^2-\htheta_{n}(\lambda_{n}+\mu_{n})}{\theta_n}\norm{z_n}_M^2 + 2\inpr{z_n}{\tfrac{\btheta_n\prt{\gamma_n\hbeta(\lambda_n+\mu_n)-\ttheta_n}}{\theta_n}y_n}_M\\
    &\quad+ 2\inpr{z_n}{\tfrac{2\btheta_n\lambda_n-\lambda_n\theta_n+\prt{2-\gamma_n\hbeta}(\lambda_n+\mu_n)\lambda_n}{\theta_n}x_n}_M\\
    &\quad+ 2\inpr{z_n}{\tfrac{\lambda_n^2\balpha_n\theta_n-\mu_n\theta_n+\btheta_n\ptheta_n-\htheta_n(\lambda_n+\mu_n)(\balpha_n-\alpha_n)}{\theta_n}p_{n-1}}_M\\
    &\quad+ 2\inpr{z_n}{\tfrac{\balpha_n\prt{\htheta_n(\lambda_n+\mu_n)-2\btheta_n^2-\lambda_n^2\theta_n}}{\theta_n}z_{n-1} + \tfrac{\alpha_n\btheta_n\prt{\ttheta_n-\gamma_n\hbeta(\lambda_n+\mu_n)}}{\theta_n}y_{n-1}}_M\\
    &\quad+ \lambda_n^2\norm{-\balpha_nz_{n-1}+\balpha_np_{n-1}}_M^2 + 2\inpr{x_n}{\lambda_n\prt{-\balpha_np_{n-1}+\balpha_nz_{n-1}}}_M\\
    &\quad+ 2\inpr{p_{n-1}}{-\omega_n(z_{n-1}-p_{n-1})}_M\\
    &\quad+ \tfrac{1}{2}\theta_n\norm{\tfrac{\ttheta_n}{\theta_n}y_n-\tfrac{2\lambda_n}{\theta_n}x_n-\tfrac{\ptheta_n}{\theta_n}p_{n-1}+\tfrac{2\btheta_n\balpha_n}{\theta_n}z_{n-1}-\tfrac{\ttheta_n\alpha_n}{\theta_n}y_{n-1}}_M^2\\
    &\quad+ \tfrac{\mu_{n}\gamma_{n}\hbeta}{2}\norm{-y_{n}-p_{n-1}+y_{n-1}}_M^2 - \tfrac{\lambda_{n}\gamma_{n}\alpha_{n}\hbeta}{2}\norm{y_{n-1}-p_{n-1}}_M^2\\
    &\quad+ \tfrac{\lambda_n\gamma_n\hbeta}{2}\norm{y_n}_M^2 - \tfrac{\ttheta_{n}(\lambda_{n}+\mu_{n})}{\htheta_{n}}\norm{y_n-(1-\alpha_n)x_n-\alpha_ny_{n-1}}_{M}^2\\
    &\quad- \tfrac{\htheta_{n}(\lambda_{n}+\mu_{n})}{\theta_{n}}\Big\Vert{}-\tfrac{\btheta_n\gamma_n\hbeta}{\htheta_n}y_n-\tfrac{\prt{2-\gamma_n\hbeta}\lambda_n}{\htheta_n}x_n\\
    &\hspace{30mm}+(\balpha_n-\alpha_n)p_{n-1}-\balpha_nz_{n-1}+\tfrac{\alpha_n\btheta_n\gamma_n\hbeta}{\htheta_n}y_{n-1}\Big\Vert_{M}^2
\end{align*}
Now, we show that all the coefficients of the terms containing $z_n$ are identical to zero. The coefficients of $\norm{z_n}_M^2$ and $\inpr{z_n}{z_{n-1}}_M$ are zero by \labelcref{itm:prop-theta-iii}. For the coefficient of $\inpr{z_n}{x_n}_M$ we have
\begin{align*}
    2\btheta_n\lambda_n-\lambda_n\theta_n+\prt{2-\gamma_n\hbeta}(\lambda_n+\mu_n)\lambda_n = \lambda_n\prt{2\btheta_n+\prt{2-\gamma_n\hbeta}(\lambda_n+\mu_n)-\theta_n}
\end{align*}
which is identical to zero by \labelcref{itm:prop-theta-ii}. For the coefficient of $\inpr{z_n}{p_{n-1}}_M$ we have
\begin{align*}
    &\lambda_n^2\balpha_n\theta_n-\mu_n\theta_n+\btheta_n\ptheta_n-\htheta_n(\lambda_n+\mu_n)(\balpha_n-\alpha_n)\\
    &\quad= \lambda_n^2\balpha_n\theta_n-\mu_n\theta_n+\btheta_n\prt{\prt{2-\gamma_n\hbeta}\mu_n+2\balpha_n\btheta_n}-\htheta_n(\lambda_n+\mu_n)(\balpha_n-\alpha_n)\\
    &\quad= \balpha_n\prt{\lambda_n^2\theta_n+2\btheta_n^2-(\lambda_n+\mu_n)\htheta_n}-\mu_n\theta_n\\
    &\qquad+\prt{2-\gamma_n\hbeta}\mu_n\btheta_n+(\lambda_n+\mu_n)\alpha_n\htheta_n\\
    &\quad= \prt{-\theta_n+\prt{2-\gamma_n\hbeta}\btheta_n+\htheta_n}\mu_n = 0
\end{align*}
where in the first equality, $\ptheta_n$ is substituted and in the third equality \labelcref{itm:prop-theta-iii} is used and in the last equality \labelcref{itm:prop-theta-i} is used. Therefore, all terms containing $z_n$ can be eliminated from $\Delta_n$ and we are left with 
\begin{align*}
    &\Delta_n =\\
    &\quad \lambda_n^2\norm{-\balpha_nz_{n-1}+\balpha_np_{n-1}}_M^2 + 2\inpr{x_n}{\lambda_n\prt{-\balpha_np_{n-1}+\balpha_nz_{n-1}}}_M\\
    &\quad+ 2\inpr{p_{n-1}}{-\omega_n(z_{n-1}-p_{n-1})}_M\\
    &\quad+ \tfrac{1}{2}\theta_n\norm{\tfrac{\ttheta_n}{\theta_n}y_n-\tfrac{2\lambda_n}{\theta_n}x_n-\tfrac{\ptheta_n}{\theta_n}p_{n-1}+\tfrac{2\btheta_n\balpha_n}{\theta_n}z_{n-1}-\tfrac{\ttheta_n\alpha_n}{\theta_n}y_{n-1}}_M^2\\
    &\quad+ \tfrac{\mu_{n}\gamma_{n}\hbeta}{2}\norm{-y_{n}-p_{n-1}+y_{n-1}}_M^2 - \tfrac{\lambda_{n}\gamma_{n}\alpha_{n}\hbeta}{2}\norm{y_{n-1}-p_{n-1}}_M^2\\
    &\quad+ \tfrac{\lambda_n\gamma_n\hbeta}{2}\norm{y_n}_M^2 - \tfrac{\ttheta_{n}(\lambda_{n}+\mu_{n})}{\htheta_{n}}\norm{y_n-(1-\alpha_n)x_n-\alpha_ny_{n-1}}_{M}^2\\
    &\quad- \tfrac{\htheta_{n}(\lambda_{n}+\mu_{n})}{\theta_{n}}\Big\Vert{}-\tfrac{\btheta_n\gamma_n\hbeta}{\htheta_n}y_n-\tfrac{\prt{2-\gamma_n\hbeta}\lambda_n}{\htheta_n}x_n\\
    &\hspace{30mm}+(\balpha_n-\alpha_n)p_{n-1}-\balpha_nz_{n-1}+\tfrac{\alpha_n\btheta_n\gamma_n\hbeta}{\htheta_n}y_{n-1}\Big\Vert_{M}^2\\
    &= \lambda_n^2\norm{-\balpha_nz_{n-1}+\balpha_np_{n-1}}_M^2 + 2\inpr{x_n}{\lambda_n\prt{-\balpha_np_{n-1}+\balpha_nz_{n-1}}}_M\\
    &\quad+ 2\inpr{p_{n-1}}{-\omega_n(z_{n-1}-p_{n-1})}_M + \tfrac{\ttheta_n^2}{2\theta_n}\norm{y_n}_M^2\\
    &\quad + 2\inpr{y_n}{\tfrac{\ttheta_n}{2}\prt{-\tfrac{2\lambda_n}{\theta_n}x_n-\tfrac{\ptheta_n}{\theta_n}p_{n-1}+\tfrac{2\btheta_n\balpha_n}{\theta_n}z_{n-1}-\tfrac{\ttheta_n\alpha_n}{\theta_n}y_{n-1}}}_M\\
    &\quad+ \tfrac{1}{2}\theta_n\norm{-\tfrac{2\lambda_n}{\theta_n}x_n-\tfrac{\ptheta_n}{\theta_n}p_{n-1}+\tfrac{2\btheta_n\balpha_n}{\theta_n}z_{n-1}-\tfrac{\ttheta_n\alpha_n}{\theta_n}y_{n-1}}_M^2\\
    &\quad+ \tfrac{\mu_{n}\gamma_{n}\hbeta}{2}\norm{y_{n}}_M^2 + 2\inpr{y_n}{\tfrac{\mu_{n}\gamma_{n}\hbeta}{2}\prt{p_{n-1}-y_{n-1}}}_M + \tfrac{\mu_{n}\gamma_{n}\hbeta}{2}\norm{p_{n-1}-y_{n-1}}_M^2 \\
    &\quad+ \tfrac{\lambda_n\gamma_n\hbeta}{2}\norm{y_n}_M^2 - \tfrac{\ttheta_{n}(\lambda_{n}+\mu_{n})}{\htheta_{n}}\norm{y_n}_{M}^2\\
    &\quad- 2\inpr{y_n}{\tfrac{\ttheta_{n}(\lambda_{n}+\mu_{n})}{\htheta_{n}}\prt{(\alpha_n-1)x_n-\alpha_ny_{n-1}}}_M \\
    &\quad- \tfrac{\ttheta_{n}(\lambda_{n}+\mu_{n})}{\htheta_{n}}\norm{(\alpha_n-1)x_n-\alpha_ny_{n-1}}_{M}^2 - \tfrac{\lambda_{n}\gamma_{n}\alpha_{n}\hbeta}{2}\norm{y_{n-1}-p_{n-1}}_M^2\\
    &\quad- \tfrac{\btheta_n^2\gamma_n^2\hbeta^2(\lambda_{n}+\mu_{n})}{\htheta_{n}\theta_n}\norm{y_n}_{M}^2 + 2\inpr{y_n}{\tfrac{\btheta_n\gamma_n\hbeta(\lambda_{n}+\mu_{n})}{\theta_n}\prt{-\tfrac{\prt{2-\gamma_n\hbeta}\lambda_n}{\htheta_n}x_n}}_M\\
    &\quad+ 2\inpr{y_n}{\tfrac{\btheta_n\gamma_n\hbeta(\lambda_{n}+\mu_{n})}{\theta_n}\prt{(\balpha_n-\alpha_n)p_{n-1}-\balpha_nz_{n-1}+\tfrac{\alpha_n\btheta_n\gamma_n\hbeta}{\htheta_n}y_{n-1}}}_M\\
    &\quad- \tfrac{\htheta_{n}(\lambda_{n}+\mu_{n})}{\theta_{n}}\norm{-\tfrac{\prt{2-\gamma_n\hbeta}\lambda_n}{\htheta_n}x_n+(\balpha_n-\alpha_n)p_{n-1}-\balpha_nz_{n-1}+\tfrac{\alpha_n\btheta_n\gamma_n\hbeta}{\htheta_n}y_{n-1}}_{M}^2\\
    &= \prt{\tfrac{\ttheta_n^2}{2\theta_n}+\tfrac{\mu_n\gamma_n\hbeta}{2}+\tfrac{\lambda_n\gamma_n\hbeta}{2}-\tfrac{\ttheta_{n}(\lambda_{n}+\mu_{n})}{\htheta_{n}}-\tfrac{\btheta_n^2\gamma_n^2\hbeta^2(\lambda_{n}+\mu_{n})}{\htheta_{n}\theta_n}}\norm{y_n}_M^2\\
    &\quad+ 2\inpr{y_n}{\prt{-\tfrac{\lambda_n\ttheta_n}{\theta_n}+\tfrac{\ttheta_{n}(\lambda_{n}+\mu_{n})(1-\alpha_n)}{\htheta_{n}}-\tfrac{\btheta_n\lambda_n\gamma_n\hbeta(\lambda_{n}+\mu_{n})\prt{2-\gamma_n\hbeta}}{\theta_n\htheta_{n}}}x_n}_M\\
    &\quad+ 2\inpr{y_n}{\prt{-\tfrac{\ttheta_n\ptheta_n}{2\theta_n}+\tfrac{\mu_{n}\gamma_{n}\hbeta}{2}+\tfrac{\btheta_n\gamma_n\hbeta(\lambda_{n}+\mu_{n})(\balpha_n-\alpha_n)}{\theta_{n}}}p_{n-1}}_M\\
    &\quad+ 2\inpr{y_n}{\prt{\tfrac{\ttheta_n\btheta_n\balpha_n}{\theta_n}-\tfrac{\btheta_n\balpha_n\gamma_n\hbeta(\lambda_{n}+\mu_{n})}{\theta_{n}}}z_{n-1}}_M\\
    &\quad+ 2\inpr{y_n}{\prt{-\tfrac{\ttheta_n^2\alpha_n}{2\theta_n}-\tfrac{\mu_{n}\gamma_{n}\hbeta}{2}+\tfrac{\ttheta_{n}\alpha_n(\lambda_{n}+\mu_{n})}{\htheta_{n}}+\tfrac{\alpha_n\btheta_n^2\gamma_n^2\hbeta^2(\lambda_{n}+\mu_{n})}{\theta_n\htheta_{n}}}y_{n-1}}_M\\
    &\quad+ \lambda_n^2\norm{-\balpha_nz_{n-1}+\balpha_np_{n-1}}_M^2 + 2\inpr{x_n}{\lambda_n\prt{-\balpha_np_{n-1}+\balpha_nz_{n-1}}}_M\\
    &\quad+ \tfrac{\mu_{n}\gamma_{n}\hbeta}{2}\norm{p_{n-1}-y_{n-1}}_M^2 + 2\inpr{p_{n-1}}{-\omega_n(z_{n-1}-p_{n-1})}_M\\
    &\quad+ \tfrac{1}{2}\theta_n\norm{-\tfrac{2\lambda_n}{\theta_n}x_n-\tfrac{\ptheta_n}{\theta_n}p_{n-1}+\tfrac{2\btheta_n\balpha_n}{\theta_n}z_{n-1}-\tfrac{\ttheta_n\alpha_n}{\theta_n}y_{n-1}}_M^2\\
    &\quad- \tfrac{\ttheta_{n}(\lambda_{n}+\mu_{n})}{\htheta_{n}}\norm{(\alpha_n-1)x_n-\alpha_ny_{n-1}}_{M}^2 - \tfrac{\lambda_{n}\gamma_{n}\alpha_{n}\hbeta}{2}\norm{y_{n-1}-p_{n-1}}_M^2\\
    &\quad- \tfrac{\htheta_{n}(\lambda_{n}+\mu_{n})}{\theta_{n}}\norm{-\tfrac{\prt{2-\gamma_n\hbeta}\lambda_n}{\htheta_n}x_n+(\balpha_n-\alpha_n)p_{n-1}-\balpha_nz_{n-1}+\tfrac{\alpha_n\btheta_n\gamma_n\hbeta}{\htheta_n}y_{n-1}}_{M}^2
\end{align*}
Now, we show that all the coefficients of the terms containing $y_n$ are identically zero. For the coefficient of $\norm{y_n}_M^2$ we have
\begin{align*}
    &\prt{\theta_n\htheta_n\gamma_n\hbeta-2\theta_n\ttheta_{n}-2\btheta_n^2\gamma_n^2\hbeta^2}(\lambda_n+\mu_n)+\ttheta_n^2\htheta_n\\
    &\quad= \prt{\theta_n\gamma_n\hbeta(2\lambda_n+2\mu_n-\lambda_n^2\gamma_n\hbeta)-2\theta_n\gamma_n\hbeta(\lambda_n+\mu_n)-2\btheta_n^2\gamma_n^2\hbeta^2}(\lambda_n+\mu_n)\\
    &\qquad+\ttheta_n^2\htheta_n\\
    &\quad= -\prt{\theta_n\lambda_n^2+2\btheta_n^2}(\lambda_n+\mu_n)\gamma_n^2\hbeta^2+(\lambda_n+\mu_n)^2\gamma_n^2\hbeta^2\htheta_n\\
    &\quad= \prt{(\lambda_n+\mu_n)\htheta_n-\theta_n\lambda_n^2-2\btheta_n^2}(\lambda_n+\mu_n)\gamma_n^2\hbeta^2,
\end{align*}
which, by \labelcref{itm:prop-theta-iii}, is identical to zero. Now, for the coefficient of $\inpr{y_n}{x_n}_M$ we have
\begin{align*}
    &\ttheta_{n}\theta_{n}(\lambda_{n}+\mu_{n})(1-\alpha_n)-\lambda_n\ttheta_n\htheta_{n}-\btheta_n\lambda_n\gamma_n\hbeta(\lambda_{n}+\mu_{n})\prt{2-\gamma_n\hbeta}\\
    &\qquad= \prt{\theta_{n}(\lambda_{n}+\mu_{n})(1-\alpha_n)-\lambda_n\htheta_{n}-\btheta_n\lambda_n\prt{2-\gamma_n\hbeta}}\ttheta_n\\
    &\qquad= \prt{\theta_{n}-\htheta_{n}-\btheta_n\prt{2-\gamma_n\hbeta}}\lambda_n\ttheta_n
\end{align*}
which is identically zero by \labelcref{itm:prop-theta-i}. For the coefficient of $\inpr{y_n}{p_{n-1}}_M$ we have
\begin{align*}
    &\mu_{n}\gamma_{n}\hbeta\theta_n+2\btheta_n\gamma_n\hbeta(\lambda_{n}+\mu_{n})(\balpha_n-\alpha_n)-\ttheta_n\ptheta_n\\
    &\qquad= \mu_{n}\gamma_{n}\hbeta\theta_n+2\btheta_n\ttheta_n(\balpha_n-\alpha_n)-\ttheta_n\prt{\prt{2-\gamma_n\hbeta}\mu_n+2\balpha_n\btheta_n}\\
    &\qquad= \mu_{n}\gamma_{n}\hbeta\theta_n+2\btheta_n\ttheta_n\balpha_n-2\btheta_n\ttheta_n\alpha_n-\prt{2-\gamma_n\hbeta}\mu_n\ttheta_n-2\balpha_n\btheta_n\ttheta_n\\
    &\qquad= \mu_{n}\gamma_{n}\hbeta\theta_n-2\btheta_n\ttheta_n\alpha_n-\prt{2-\gamma_n\hbeta}\mu_n\ttheta_n\\
    &\qquad= \mu_{n}\gamma_{n}\hbeta\theta_n-2\btheta_n\mu_{n}\gamma_{n}\hbeta-\prt{2-\gamma_n\hbeta}(\lambda_n+\mu_n)\mu_n\gamma_{n}\hbeta\\
    &\qquad= \mu_{n}\gamma_{n}\hbeta\prt{\theta_n-2\btheta_n-\prt{2-\gamma_n\hbeta}(\lambda_n+\mu_n)}\\
    &\qquad= \mu_{n}\gamma_{n}\hbeta\prt{\theta_n-2\prt{\lambda_n+\mu_n-\lambda_n^2}-\prt{2-\gamma_n\hbeta}(\lambda_n+\mu_n)}\\
    &\qquad= \mu_{n}\gamma_{n}\hbeta\prt{\theta_n+2\lambda_n^2-(4-\gamma_n\hbeta)(\lambda_n+\mu_n)}
\end{align*}
which is equal to zero by the definition of $\theta_n$. The equivalence of the coefficient of $\inpr{y_n}{z_{n-1}}_M$ to zero follows from the definition of $\ttheta_n$. For the coefficient of $\inpr{y_n}{y_{n-1}}_M$ we have
\begin{align*}
    &2\alpha_n\btheta_n^2\gamma_n^2\hbeta^2(\lambda_{n}+\mu_{n})-\ttheta_n^2\htheta_{n}\alpha_n+2\theta_n\ttheta_{n}\alpha_n(\lambda_{n}+\mu_{n})-\mu_{n}\gamma_{n}\hbeta\theta_n\htheta_{n}\\
    &\qquad= 2\alpha_n\btheta_n^2\gamma_n\hbeta\ttheta_n-\ttheta_n^2\htheta_{n}\alpha_n+2\theta_n\ttheta_{n}\alpha_n(\lambda_{n}+\mu_{n})-\alpha_n\ttheta_n\theta_n\htheta_{n}\\
    &\qquad= \alpha_n\ttheta_n\prt{2\btheta_n^2\gamma_n\hbeta-\ttheta_n\htheta_{n}+2\theta_n(\lambda_{n}+\mu_{n})-\theta_n\htheta_{n}}\\
    &\qquad= \alpha_n\ttheta_n\prt{2\btheta_n^2\gamma_n\hbeta-\ttheta_n\htheta_{n}+\theta_n\prt{2(\lambda_{n}+\mu_{n})-2(\lambda_{n}+\mu_{n})+\lambda_n^2\gamma_n\hbeta}}\\
    &\qquad= \alpha_n\ttheta_n\prt{2\btheta_n^2\gamma_n\hbeta-\ttheta_n\htheta_{n}+\theta_n\lambda_n^2\gamma_n\hbeta}\\
    &\qquad= \alpha_n\ttheta_n\prt{2\btheta_n^2\gamma_n\hbeta-\htheta_{n}\gamma_n\hbeta(\lambda_n+\mu_n)+\theta_n\lambda_n^2\gamma_n\hbeta}\\
    &\qquad= \alpha_n\ttheta_n\gamma_n\hbeta\prt{2\btheta_n^2-\htheta_{n}(\lambda_n+\mu_n)+\theta_n\lambda_n^2}
\end{align*}
which by \labelcref{itm:prop-theta-iii} is identical to zero. Therefore, all the coefficients of the terms containing $y_n$ are zero and we can eliminate those terms. The remaining terms are
\begin{align*}
    &\Delta_n =\\
    &\quad\lambda_n^2\norm{-\balpha_nz_{n-1}+\balpha_np_{n-1}}_M^2 + 2\inpr{x_n}{\lambda_n\prt{-\balpha_np_{n-1}+\balpha_nz_{n-1}}}_M\\
    &\quad+ \tfrac{\mu_{n}\gamma_{n}\hbeta}{2}\norm{p_{n-1}-y_{n-1}}_M^2 + 2\inpr{p_{n-1}}{-\omega_n(z_{n-1}-p_{n-1})}_M\\
    &\quad+ \tfrac{1}{2}\theta_n\norm{-\tfrac{2\lambda_n}{\theta_n}x_n-\tfrac{\ptheta_n}{\theta_n}p_{n-1}+\tfrac{2\btheta_n\balpha_n}{\theta_n}z_{n-1}-\tfrac{\ttheta_n\alpha_n}{\theta_n}y_{n-1}}_M^2\\
    &\quad- \tfrac{\ttheta_{n}(\lambda_{n}+\mu_{n})}{\htheta_{n}}\norm{(\alpha_n-1)x_n-\alpha_ny_{n-1}}_{M}^2 - \tfrac{\lambda_{n}\gamma_{n}\alpha_{n}\hbeta}{2}\norm{y_{n-1}-p_{n-1}}_M^2\\
    &\quad- \tfrac{\htheta_{n}(\lambda_{n}+\mu_{n})}{\theta_{n}}\norm{-\tfrac{\prt{2-\gamma_n\hbeta}\lambda_n}{\htheta_n}x_n+(\balpha_n-\alpha_n)p_{n-1}-\balpha_nz_{n-1}+\tfrac{\alpha_n\btheta_n\gamma_n\hbeta}{\htheta_n}y_{n-1}}_{M}^2\\
    &= \lambda_n^2\norm{-\balpha_nz_{n-1}+\balpha_np_{n-1}}_M^2 + 2\inpr{x_n}{\lambda_n\prt{-\balpha_np_{n-1}+\balpha_nz_{n-1}}}_M\\
    &\quad+ \tfrac{\mu_{n}\gamma_{n}\hbeta}{2}\norm{p_{n-1}-y_{n-1}}_M^2 + 2\inpr{p_{n-1}}{-\omega_n(z_{n-1}-p_{n-1})}_M\\
    &\quad+ \tfrac{2\lambda_n^2}{\theta_n}\norm{x_n}_M^2 + 2\inpr{x_n}{\lambda_n\prt{\tfrac{\ptheta_n}{\theta_n}p_{n-1}-\tfrac{2\btheta_n\balpha_n}{\theta_n}z_{n-1}+\tfrac{\ttheta_n\alpha_n}{\theta_n}y_{n-1}}}_M\\
    &\quad+ \tfrac{1}{2}\theta_n\norm{\tfrac{\ptheta_n}{\theta_n}p_{n-1}-\tfrac{2\btheta_n\balpha_n}{\theta_n}z_{n-1}+\tfrac{\ttheta_n\alpha_n}{\theta_n}y_{n-1}}_M^2 - \tfrac{\lambda_{n}\gamma_{n}\alpha_{n}\hbeta}{2}\norm{y_{n-1}-p_{n-1}}_M^2\\
    &\quad- \tfrac{\ttheta_{n}(\lambda_{n}+\mu_{n})(1-\alpha_n)^2}{\htheta_{n}}\norm{x_n}_{M}^2 + 2\inpr{x_n}{\tfrac{\ttheta_{n}\alpha_n(\lambda_{n}+\mu_{n})(\alpha_n-1)}{\htheta_{n}}y_{n-1}}_M\\
    &\quad- \tfrac{(\lambda_{n}+\mu_{n})\prt{2-\gamma_n\hbeta}^2\lambda_n^2}{\theta_n\htheta_n}\norm{x_n}_{M}^2 - \tfrac{\ttheta_{n}\alpha_n^2(\lambda_{n}+\mu_{n})}{\htheta_{n}}\norm{y_{n-1}}_{M}^2 \\
    &\quad+ 2\inpr{x_{n}}{\tfrac{\lambda_n\prt{2-\gamma_n\hbeta}(\lambda_{n}+\mu_{n})}{\theta_{n}}\prt{(\balpha_n-\alpha_n)p_{n-1}-\balpha_nz_{n-1}+\tfrac{\alpha_n\btheta_n\gamma_n\hbeta}{\htheta_n}y_{n-1}}}_M\\
    &\quad- \tfrac{\htheta_{n}(\lambda_{n}+\mu_{n})}{\theta_{n}}\norm{(\balpha_n-\alpha_n)p_{n-1}-\balpha_nz_{n-1}+\tfrac{\alpha_n\btheta_n\gamma_n\hbeta}{\htheta_n}y_{n-1}}_{M}^2\\
    &= \prt{\tfrac{2\lambda_n^2}{\theta_n}-\tfrac{\ttheta_{n}(\lambda_{n}+\mu_{n})(1-\alpha_n)^2}{\htheta_{n}}-\tfrac{(\lambda_{n}+\mu_{n})\prt{2-\gamma_n\hbeta}^2\lambda_n^2}{\theta_n\htheta_n}}\norm{x_n}_M^2\\
    &\quad+ 2\inpr{x_n}{\prt{\tfrac{\lambda_n\ptheta_n}{\theta_n}+\tfrac{\lambda_n\prt{2-\gamma_n\hbeta}(\lambda_{n}+\mu_{n})(\balpha_n-\alpha_n)}{\theta_{n}}-\lambda_n\balpha_n}p_{n-1}}_M\\
    &\quad+ 2\inpr{x_n}{\prt{-\tfrac{2\lambda_n\btheta_n\balpha_n}{\theta_n}-\tfrac{\lambda_n\balpha_n\prt{2-\gamma_n\hbeta}(\lambda_{n}+\mu_{n})}{\theta_{n}}+\lambda_n\balpha_n}z_{n-1}}_M\\
    &\quad+ 2\inpr{x_n}{\prt{\tfrac{\lambda_n\ttheta_n\alpha_n}{\theta_n}+\tfrac{\ttheta_{n}\alpha_n(\lambda_{n}+\mu_{n})(\alpha_n-1)}{\htheta_{n}}+\tfrac{\btheta_n\alpha_n\lambda_n\gamma_n\hbeta\prt{2-\gamma_n\hbeta}(\lambda_{n}+\mu_{n})}{\theta_n\htheta_{n}}}y_{n-1}}_M\\
    &\quad+ \lambda_n^2\norm{-\balpha_nz_{n-1}+\balpha_np_{n-1}}_M^2 - \tfrac{\ttheta_{n}\alpha_n^2(\lambda_{n}+\mu_{n})}{\htheta_{n}}\norm{y_{n-1}}_{M}^2\\
    &\quad+ \tfrac{\mu_{n}\gamma_{n}\hbeta}{2}\norm{p_{n-1}-y_{n-1}}_M^2 + 2\inpr{p_{n-1}}{-\omega_n(z_{n-1}-p_{n-1})}_M\\
    &\quad+ \tfrac{1}{2}\theta_n\norm{\tfrac{\ptheta_n}{\theta_n}p_{n-1}-\tfrac{2\btheta_n\balpha_n}{\theta_n}z_{n-1}+\tfrac{\ttheta_n\alpha_n}{\theta_n}y_{n-1}}_M^2 - \tfrac{\lambda_{n}\gamma_{n}\alpha_{n}\hbeta}{2}\norm{y_{n-1}-p_{n-1}}_M^2\\
    &\quad- \tfrac{\htheta_{n}(\lambda_{n}+\mu_{n})}{\theta_{n}}\norm{(\balpha_n-\alpha_n)p_{n-1}-\balpha_nz_{n-1}+\tfrac{\alpha_n\btheta_n\gamma_n\hbeta}{\htheta_n}y_{n-1}}_{M}^2
\end{align*}
We want to show that all the coefficients of the terms containing $x_n$ are zero. For the coefficient of $\norm{x_n}_M^2$ we have
\begin{equation}\label{eq:xn-squared-coefficient}
\begin{aligned}
    &2\lambda_n^2\htheta_n-(\lambda_{n}+\mu_{n})\prt{2-\gamma_n\hbeta}^2\lambda_n^2-\theta_n\ttheta_{n}(\lambda_{n}+\mu_{n})(1-\alpha_n)^2\\
    &\qquad= 2\lambda_n^2\htheta_n-(\lambda_{n}+\mu_{n})\prt{2-\gamma_n\hbeta}^2\lambda_n^2-\theta_n\lambda_n^2\gamma_n\hbeta\\
    &\qquad= \lambda_n^2\prt{2\htheta_n-(\lambda_{n}+\mu_{n})\prt{2-\gamma_n\hbeta}^2-\theta_n\gamma_n\hbeta}\\
    &\qquad= \lambda_n^2\prt{2\prt{2\lambda_n+2\mu_n-\gamma_n\hbeta\lambda_n^2}-(\lambda_{n}+\mu_{n})\prt{4-4\gamma_n\hbeta+\gamma_n^2\hbeta^2}-\theta_n\gamma_n\hbeta}\\
    &\qquad= \lambda_n^2\Bigg(-2\gamma_n\hbeta\lambda_n^2-(\lambda_{n}+\mu_{n})\prt{-4\gamma_n\hbeta+\gamma_n^2\hbeta^2}\\
    &\hspace{30mm}-\prt{(4-\gamma_n\hbeta)(\lambda_n+\mu_n)-2\lambda_n^2}\gamma_n\hbeta\bigg)\\
    &\qquad= \lambda_n^2\prt{-2\gamma_n\hbeta\lambda_n^2+2\lambda_n^2\gamma_n\hbeta}=0
\end{aligned}
\end{equation}
where in the first equality $\ttheta_n$ and $\alpha_n$ and in the third equality $\htheta_n$ are substituted by their definition from \cref{alg:main}. For the coefficient of $\inpr{x_n}{p_{n-1}}$ we have
\begin{align*}
    &\lambda_n\ptheta_n+\lambda_n\prt{2-\gamma_n\hbeta}(\lambda_{n}+\mu_{n})(\balpha_n-\alpha_n)-\lambda_n\balpha_n\theta_n\\
    &\quad= \lambda_n\prt{\prt{2-\gamma_n\hbeta}\mu_n+2\balpha_n\btheta_n}+\lambda_n\prt{2-\gamma_n\hbeta}(\lambda_{n}+\mu_{n})(\balpha_n-\alpha_n)-\lambda_n\balpha_n\theta_n\\
    &\quad= 2\lambda_n\balpha_n\btheta_n+\prt{2-\gamma_n\hbeta}(\lambda_{n}+\mu_{n})\lambda_n\balpha_n-\lambda_n\balpha_n\theta_n\\
    &\quad\quad-\prt{2-\gamma_n\hbeta}(\lambda_{n}+\mu_{n})\lambda_n\alpha_n+\prt{2-\gamma_n\hbeta}\lambda_n\mu_n\\
    &\quad= \lambda_n\balpha_n\prt{2\btheta_n+\prt{2-\gamma_n\hbeta}(\lambda_{n}+\mu_{n})-\theta_n}\\
    &\quad\quad-\prt{2-\gamma_n\hbeta}\lambda_n\mu_n+\prt{2-\gamma_n\hbeta}\lambda_n\mu_n
\end{align*}
which by
\labelcref{itm:prop-theta-ii} is zero. For the coefficient of $\inpr{x_n}{z_{n-1}}$ we have
\begin{align*}
    &\lambda_n\balpha_n\theta_n-2\lambda_n\btheta_n\balpha_n-\lambda_n\balpha_n\prt{2-\gamma_n\hbeta}(\lambda_{n}+\mu_{n})\\
    &\qquad= \lambda_n\balpha_n\prt{\theta_n-2\btheta_n-\prt{2-\gamma_n\hbeta}(\lambda_{n}+\mu_{n})}
\end{align*}
which is identically zero by \labelcref{itm:prop-theta-ii}. For the coefficient of $\inpr{x_n}{y_{n-1}}$ we have
\begin{align*}
    &\lambda_n\ttheta_n\alpha_n\htheta_{n}+\btheta_n\alpha_n\lambda_n\gamma_n\hbeta\prt{2-\gamma_n\hbeta}(\lambda_{n}+\mu_{n})-\theta_{n}\ttheta_n\alpha_n(\lambda_{n}+\mu_{n})(1-\alpha_n)\\
    &\qquad= \lambda_n\ttheta_n\alpha_n\htheta_{n}+\btheta_n\alpha_n\lambda_n\ttheta_n\prt{2-\gamma_n\hbeta}-\theta_{n}\ttheta_n\alpha_n\lambda_n\\
    &\qquad= \lambda_n\ttheta_n\alpha_n\prt{\htheta_{n}+\btheta_n\prt{2-\gamma_n\hbeta}-\theta_n}
\end{align*}
which by \labelcref{itm:prop-theta-i} is identically zero. Now, expanding all the remaining terms, reordering and recollecting them give
\begin{align*}
    &\Delta_n=\\
    &\quad\lambda_n^2\norm{-\balpha_nz_{n-1}+\balpha_np_{n-1}}_M^2 - \tfrac{\ttheta_{n}\alpha_n^2(\lambda_{n}+\mu_{n})}{\htheta_{n}}\norm{y_{n-1}}_{M}^2\\
    &\quad+ \tfrac{\mu_{n}\gamma_{n}\hbeta}{2}\norm{p_{n-1}-y_{n-1}}_M^2 + 2\inpr{p_{n-1}}{-\omega_n(z_{n-1}-p_{n-1})}_M\\
    &\quad+ \tfrac{1}{2}\theta_n\norm{\tfrac{\ptheta_n}{\theta_n}p_{n-1}-\tfrac{2\btheta_n\balpha_n}{\theta_n}z_{n-1}+\tfrac{\ttheta_n\alpha_n}{\theta_n}y_{n-1}}_M^2 - \tfrac{\lambda_{n}\gamma_{n}\alpha_{n}\hbeta}{2}\norm{y_{n-1}-p_{n-1}}_M^2\\
    &\quad- \tfrac{\htheta_{n}(\lambda_{n}+\mu_{n})}{\theta_{n}}\norm{(\balpha_n-\alpha_n)p_{n-1}-\balpha_nz_{n-1}+\tfrac{\alpha_n\btheta_n\gamma_n\hbeta}{\htheta_n}y_{n-1}}_{M}^2\\
    &= \lambda_n^2\balpha_n^2\norm{p_{n-1}}_M^2 + 2\inpr{p_{n-1}}{-\lambda_n^2\balpha_n^2z_{n-1}}_M + \lambda_n^2\balpha_n^2\norm{z_{n-1}}_M^2\\
    &\quad + \tfrac{\mu_{n}\gamma_{n}\hbeta}{2}\norm{p_{n-1}}_M^2 + 2\inpr{p_{n-1}}{-\tfrac{\mu_{n}\gamma_{n}\hbeta}{2}y_{n-1}}_M + \tfrac{\mu_{n}\gamma_{n}\hbeta}{2}\norm{y_{n-1}}_M^2 \\
    &\quad- \tfrac{\ttheta_{n}\alpha_n^2(\lambda_{n}+\mu_{n})}{\htheta_{n}}\norm{y_{n-1}}_{M}^2 + 2\omega_n\norm{p_{n-1}}_M^2 + 2\inpr{p_{n-1}}{-\omega_nz_{n-1}}_M\\
    &\quad+ \tfrac{{\ptheta_n}^2}{2\theta_n}\norm{p_{n-1}}_M^2 + 2\inpr{p_{n-1}}{\tfrac{1}{2}\ptheta_n\prt{-\tfrac{2\btheta_n\balpha_n}{\theta_n}z_{n-1}+\tfrac{\ttheta_n\alpha_n}{\theta_n}y_{n-1}}}_M\\
    &\quad+ \tfrac{2\btheta_n^2\balpha_n^2}{\theta_n}\norm{z_{n-1}}_M^2 + 2\inpr{z_{n-1}}{-\tfrac{\btheta_n\balpha_n\ttheta_n\alpha_n}{\theta_n}y_{n-1}}_M + \tfrac{\ttheta_n^2\alpha_n^2}{2\theta_n}\norm{y_{n-1}}_M^2\\
    &\quad- \tfrac{\lambda_{n}\gamma_{n}\alpha_{n}\hbeta}{2}\norm{p_{n-1}}_M^2 + 2\inpr{p_{n-1}}{\tfrac{\lambda_{n}\gamma_{n}\alpha_{n}\hbeta}{2}y_{n-1}}_M - \tfrac{\lambda_{n}\gamma_{n}\alpha_{n}\hbeta}{2}\norm{y_{n-1}}_M^2\\
    &\quad- \tfrac{\htheta_{n}(\lambda_{n}+\mu_{n})(\balpha_n-\alpha_n)^2}{\theta_{n}}\norm{p_{n-1}}_{M}^2 - \tfrac{\alpha_n^2\btheta_n^2\gamma_n^2\hbeta^2(\lambda_{n}+\mu_{n})}{\htheta_n\theta_n}\norm{y_{n-1}}_{M}^2\\
    &\quad- \tfrac{\htheta_{n}\balpha_n^2(\lambda_{n}+\mu_{n})}{\theta_{n}}\norm{z_{n-1}}_{M}^2 + 2\inpr{z_{n-1}}{\tfrac{\balpha_n\alpha_n\btheta_n\gamma_n\hbeta(\lambda_{n}+\mu_{n})}{\theta_{n}}y_{n-1}}_M\\
    &\quad + 2\inpr{p_{n-1}}{-\tfrac{\htheta_{n}(\lambda_{n}+\mu_{n})(\balpha_n-\alpha_n)}{\theta_{n}}\prt{-\balpha_nz_{n-1}+\tfrac{\alpha_n\btheta_n\gamma_n\hbeta}{\htheta_n}y_{n-1}}}_M\\
    &=\prt{\lambda_n^2\balpha_n^2+\tfrac{\mu_{n}\gamma_{n}\hbeta}{2}+2\omega_n+\tfrac{{\ptheta_n}^2}{2\theta_n}-\tfrac{\lambda_{n}\gamma_{n}\alpha_{n}\hbeta}{2}-\tfrac{\htheta_{n}(\lambda_{n}+\mu_{n})(\balpha_n-\alpha_n)^2}{\theta_{n}}}\norm{p_{n-1}}_M^2\\
    &\quad+2\inpr{p_{n-1}}{\prt{-\lambda_n^2\balpha_n^2-\omega_n-\tfrac{\btheta_n\ptheta_n\balpha_n}{\theta_n}+\tfrac{\htheta_{n}\balpha_n(\lambda_{n}+\mu_{n})(\balpha_n-\alpha_n)}{\theta_{n}}}z_{n-1}}_M\\
    &\quad+ 2\inpr{p_{n-1}}{\prt{-\tfrac{\mu_{n}\gamma_{n}\hbeta}{2}+\tfrac{\ptheta_n\ttheta_n\alpha_n}{2\theta_n}+\tfrac{\lambda_{n}\gamma_{n}\alpha_{n}\hbeta}{2}-\tfrac{\alpha_n\btheta_n\gamma_n\hbeta(\lambda_{n}+\mu_{n})(\balpha_n-\alpha_n)}{\theta_{n}}}y_{n-1}}_M\\
    &\quad+ \prt{\lambda_n^2\balpha_n^2+\tfrac{2\btheta_n^2\balpha_n^2}{\theta_n}-\tfrac{\htheta_{n}\balpha_n^2(\lambda_{n}+\mu_{n})}{\theta_{n}}}\norm{z_{n-1}}_M^2\\
    &\quad+ 2\inpr{z_{n-1}}{\prt{-\tfrac{\btheta_n\balpha_n\ttheta_n\alpha_n}{\theta_n}+\tfrac{\balpha_n\alpha_n\btheta_n\gamma_n\hbeta(\lambda_{n}+\mu_{n})}{\theta_{n}}}y_{n-1}}_M\\
    &\quad+\prt{\tfrac{\mu_{n}\gamma_{n}\hbeta}{2}-\tfrac{\ttheta_{n}\alpha_n^2(\lambda_{n}+\mu_{n})}{\htheta_{n}}+\tfrac{\ttheta_n^2\alpha_n^2}{2\theta_n}-\tfrac{\lambda_{n}\gamma_{n}\alpha_{n}\hbeta}{2}-\tfrac{\alpha_n^2\btheta_n^2\gamma_n^2\hbeta^2(\lambda_{n}+\mu_{n})}{\htheta_n\theta_n}}\norm{y_{n-1}}_M^2
\end{align*}
We show that all the coefficients in the expression above are identically zero. Starting by the coefficient of $\norm{p_{n-1}}_M^2$, we have
\begin{align*}
    &2\theta_{n}\lambda_n^2\balpha_n^2+\theta_{n}\mu_{n}\gamma_{n}\hbeta+4\theta_{n}\omega_n+{\ptheta_n}^2-\theta_{n}\lambda_{n}\gamma_{n}\alpha_{n}\hbeta-2\htheta_{n}(\lambda_{n}+\mu_{n})(\balpha_n-\alpha_n)^2\\
    &\quad= 2\theta_{n}\lambda_n^2\balpha_n^2+\theta_{n}\mu_{n}\gamma_{n}\hbeta-4\theta_{n}\balpha_n\mu_n+\prt{\prt{2-\gamma_n\hbeta}\mu_n+2\balpha_n\btheta_n}^2\\
    &\qquad-\theta_{n}\lambda_{n}\gamma_{n}\alpha_{n}\hbeta-2\htheta_{n}(\lambda_{n}+\mu_{n})(\balpha_n-\alpha_n)^2\\
    &\quad= 2\balpha_n^2\prt{\theta_{n}\lambda_n^2+2\btheta_n^2-\htheta_n(\lambda_n+\mu_n)}+ 4\mu_n\balpha_n\prt{-\theta_n+\btheta_n\prt{2-\gamma_n\hbeta}+\htheta_n}\\
    &\qquad+ \theta_{n}\mu_{n}\gamma_{n}\hbeta+\prt{2-\gamma_n\hbeta}^2\mu_n^2-\theta_{n}\lambda_{n}\gamma_{n}\alpha_{n}\hbeta-2\htheta_{n}(\lambda_{n}+\mu_{n})\alpha_n^2\\
    &\quad= \theta_{n}\mu_{n}\gamma_{n}\hbeta+\prt{2-\gamma_n\hbeta}^2\mu_n^2-\theta_{n}\lambda_{n}\gamma_{n}\alpha_{n}\hbeta-2\htheta_{n}\mu_{n}\alpha_n\\
    &\quad= \theta_{n}\gamma_{n}\hbeta\prt{\mu_{n}-\lambda_n\alpha_n}+\prt{2-\gamma_n\hbeta}^2\mu_n^2-2\htheta_{n}\mu_{n}\alpha_n\\
    &\quad= \theta_{n}\gamma_{n}\hbeta\mu_{n}\alpha_n+\prt{2-\gamma_n\hbeta}^2(\lambda_n+\mu_n)\alpha_n\mu_n-2\htheta_{n}\mu_{n}\alpha_n\\
    &\quad= \mu_{n}\alpha_n\prt{\theta_{n}\gamma_{n}\hbeta+\prt{2-\gamma_n\hbeta}^2(\lambda_n+\mu_n)-2\htheta_{n}}
\end{align*}
where in the first equality $\omega_n=-\balpha_n\mu_n$ is used and $\ptheta_n$ is substituted from \labelcref{itm:def-param-theta-prime}, the third equality is attained from \labelcref{itm:prop-theta-i} and \labelcref{itm:prop-theta-iii}, and the expression to the right-hand side of the last equality is identically zero by \eqref{eq:xn-squared-coefficient}. For the coefficient of $\inpr{p_{n-1}}{z_{n-1}}_M$ we have
\begin{align*}
    &-\lambda_n^2\balpha_n^2\theta_{n}-\omega_n\theta_{n}-\btheta_n\ptheta_n\balpha_n+\htheta_{n}\balpha_n(\lambda_{n}+\mu_{n})(\balpha_n-\alpha_n)\\
    &\quad= -\lambda_n^2\balpha_n^2\theta_{n}-\omega_n\theta_{n}-\btheta_n\balpha_n\prt{\prt{2-\gamma_n\hbeta}\mu_n+2\balpha_n\btheta_n}\\
    &\qquad+\htheta_{n}\balpha_n(\lambda_{n}+\mu_{n})(\balpha_n-\alpha_n)\\
    &\quad= \balpha_n^2\prt{-\lambda_n^2\theta_{n}-2\btheta_n^2+\htheta_n(\lambda_n+\mu_n)}-\omega_n\theta_{n}-\btheta_n\balpha_n\mu_n\prt{2-\gamma_n\hbeta}\\
    &\qquad-\htheta_{n}\balpha_n(\lambda_{n}+\mu_{n})\alpha_n\\
    &\quad= \balpha_n\mu_n\theta_{n}-\btheta_n\balpha_n\mu_n\prt{2-\gamma_n\hbeta}-\htheta_{n}\balpha_n\mu_{n}\\
    &\quad= \balpha_n\mu_n\prt{\theta_{n}-\btheta_n\prt{2-\gamma_n\hbeta}-\htheta_{n}}
\end{align*}
which by \labelcref{itm:prop-theta-i} is equal to zero. The third equality above is attained by using \labelcref{itm:prop-theta-iii}. For the coefficient of $\inpr{p_{n-1}}{y_{n-1}}_M$ we have
\begin{align*}
    &\ptheta_n\ttheta_n\alpha_n-\mu_{n}\gamma_{n}\hbeta\theta_{n}+\lambda_{n}\gamma_{n}\alpha_{n}\hbeta\theta_{n}-2\alpha_n\btheta_n\gamma_n\hbeta(\lambda_{n}+\mu_{n})(\balpha_n-\alpha_n)\\
    &\quad= \prt{\prt{2-\gamma_n\hbeta}\mu_n+2\balpha_n\btheta_n}\ttheta_n\alpha_n-\mu_{n}\gamma_{n}\hbeta\theta_{n}+\lambda_{n}\gamma_{n}\alpha_{n}\hbeta\theta_{n}\\
    &\qquad -2\alpha_n\btheta_n\ttheta_n(\balpha_n-\alpha_n)\\
    &\quad= \prt{2-\gamma_n\hbeta}\mu_n\ttheta_n\alpha_n-\mu_{n}\gamma_{n}\hbeta\theta_{n}+\lambda_{n}\gamma_{n}\hbeta\theta_{n}\alpha_{n}+2\alpha_n\btheta_n\ttheta_n\alpha_n\\
    &\quad= \prt{2-\gamma_n\hbeta}\mu_n(\lambda_n+\mu_n)\gamma_n\hbeta\alpha_n-(\lambda_n+\mu_n)\alpha_n\gamma_{n}\hbeta\theta_{n}\\
    &\qquad+\lambda_{n}\gamma_{n}\hbeta\theta_{n}\alpha_{n}+2\alpha_n\btheta_n\mu_n\gamma_n\hbeta\\
    &\quad= \prt{2-\gamma_n\hbeta}\mu_n(\lambda_n+\mu_n)\gamma_n\hbeta\alpha_n-\mu_n\alpha_n\gamma_{n}\hbeta\theta_{n}+2\alpha_n\btheta_n\mu_n\gamma_n\hbeta\\
    &\quad= \mu_n\gamma_n\hbeta\alpha_n\prt{\prt{2-\gamma_n\hbeta}(\lambda_n+\mu_n)-\theta_{n}+2\btheta_n}
\end{align*}
which by \labelcref{itm:prop-theta-ii} is identical to zero. For the coefficient of $\norm{z_{n-1}}_M^2$, it is straightforward to see its equivalence to zero by \labelcref{itm:prop-theta-iii}. Likewise, the coefficient of $\inpr{z_{n-1}}{y_{n-1}}_M$ is identically zero by definition of $\ttheta_n$. The coefficient of $\norm{y_{n-1}}_M^2$ is
\begin{align*}
    &\mu_{n}\gamma_{n}\hbeta\htheta_n\theta_n-2\theta_n\ttheta_{n}\alpha_n^2(\lambda_{n}+\mu_{n})+\htheta_n\ttheta_n^2\alpha_n^2-\lambda_{n}\gamma_{n}\alpha_{n}\hbeta\htheta_n\theta_n\\
    &\quad\qquad-2\alpha_n^2\btheta_n^2\gamma_n^2\hbeta^2(\lambda_{n}+\mu_{n})\\
    &\quad= (\lambda_n+\mu_n)\alpha_n\gamma_{n}\hbeta\htheta_n\theta_n-2\theta_n\ttheta_{n}\alpha_n^2(\lambda_{n}+\mu_{n})+\htheta_n\ttheta_n^2\alpha_n^2\\
    &\quad\qquad-\lambda_{n}\gamma_{n}\alpha_{n}\hbeta\htheta_n\theta_n-2\alpha_n^2\btheta_n^2\gamma_n^2\hbeta^2(\lambda_{n}+\mu_{n})\\
    &\quad= \mu_n\alpha_n\gamma_{n}\hbeta\htheta_n\theta_n-2\theta_n\ttheta_{n}\alpha_n^2(\lambda_{n}+\mu_{n})+\htheta_n\ttheta_n^2\alpha_n^2-2\alpha_n^2\btheta_n^2\gamma_n^2\hbeta^2(\lambda_{n}+\mu_{n})\\
    &\quad= \mu_n\alpha_n\gamma_{n}\hbeta\htheta_n\theta_n-2\theta_n\gamma_n\hbeta\mu_n\alpha_n(\lambda_{n}+\mu_{n})+\htheta_n\ttheta_n\mu_n\gamma_n\hbeta\alpha_n-2\alpha_n\mu_n\btheta_n^2\gamma_n^2\hbeta^2\\
    &\quad= \mu_n\alpha_n\gamma_{n}\hbeta\prt{\htheta_n\theta_n-2\theta_n(\lambda_{n}+\mu_{n})+\htheta_n\ttheta_n-2\btheta_n^2\gamma_n\hbeta}\\
    &\quad= \mu_n\alpha_n\gamma_{n}\hbeta\prt{\theta_n\prt{\htheta_n-2\lambda_{n}-2\mu_{n}}+\htheta_n\gamma_n\hbeta(\lambda_n+\mu_n)-2\btheta_n^2\gamma_n\hbeta}\\
    &\quad= \mu_n\alpha_n\gamma_{n}\hbeta\prt{-\theta_n\lambda_n^2\gamma_n\hbeta+\htheta_n\gamma_n\hbeta(\lambda_n+\mu_n)-2\btheta_n^2\gamma_n\hbeta}\\
    &\quad= \mu_n\alpha_n\gamma_{n}^2\hbeta^2\prt{-\theta_n\lambda_n^2+\htheta_n(\lambda_n+\mu_n)-2\btheta_n^2}
\end{align*}
which by \labelcref{itm:prop-theta-iii} is equivalent to zero. This concludes the proof.
\end{proof}

\section{Conclusions}
\label{sec:conclusions}

We have presented a variant of the well-known forward--backward algorithm. Our method incorporates momentum-like terms in the algorithm updates as well as deviation vectors. These deviation vectors can be chosen arbitrarily as long as a safeguarding condition that limits their size is satisfied. We propose special instances of our method that fulfill the safeguarding condition by design. Numerical evaluations reveal that these novel methods can significantly outperform the traditional forward--backward method as well as the accelerated proximal point method and the Halpern iteration, all of which are encompassed within our framework. This demonstrates the potential of our proposed methods for efficiently solving structured monotone inclusions.

\printbibliography

\end{document}